\documentclass[a4paper,12pt,reqno]{amsart}

 \usepackage{amsmath,amssymb,latexsym,amstext}

\theoremstyle{definition}
\newtheorem{defi}{Definition}[section]

\theoremstyle{plain}
\newtheorem{thm}[defi]{Theorem}

\newtheorem{lem}[defi]{Lemma}
\newtheorem{prop}[defi]{Proposition}
\newtheorem*{prop*} {Proposition}
\newtheorem{cor}[defi]{Corollary}
\theoremstyle{remark}
\newtheorem{rmk}[defi]{Remark}
\newtheorem{nota}[defi]{Notation}

\usepackage{bm,bbm,mathtools}

\newcommand{\mathds}{\mathbb}

\numberwithin{equation}{section}

\usepackage{soul, enumitem}

\newcommand{\comm}[1]{ }

\newcommand{\Ktau}{V^\tau_{K_{12}}}

\DeclareMathOperator{\ch}{ch}
\DeclareMathOperator{\tr}{tr}
\DeclareMathOperator{\Aut}{Aut}
\DeclareMathOperator{\wt}{wt}
\DeclareMathOperator{\supp}{Supp}

\DeclareMathOperator{\qdim}{qdim}
\DeclareMathOperator{\glob}{glob}

\newcommand{\idd}{\mathbbm{1}}

\newcommand{\vt}[4]{V_{#1}^{T,#2}(#3)[#4]}

\newcommand{\cK}{ \mathcal{K}}
\newcommand{\GF}{ \mathds{F}_4}
\newcommand{\F}{ \mathds{F}}

\newcommand{\cC}{ \bm{\mathcal{C}}} 
\newcommand{\cD}{ \bm{\mathcal{D}}} 
\newcommand{\cS}{ \bm{\mathcal{S}}}
\newcommand{\cZ}{    \bm{ 0 }   } 
\newcommand{\cH}{ \bm{\mathcal{H}}}

\newcommand{\cgg}{ {\bm{  \gamma }}}
\newcommand{\cdd}{{  \bm{ \delta}}}
\newcommand{\cll}{{ \bm{ \lambda }}}
\newcommand{\cee}{{  \bm{ \eta}}}
\newcommand{\cB}{ {\bm{\mathcal{B}}}}
\newcommand{\LCD}{{L_{ \cC \times \cD}} } 

\newcommand{\LC}{{L_{ \cC \times \cZ}}}
\newcommand{\VL}[1]{V_{L_{ #1}}}

\newcommand{\VLCD}{{V_{\LCD }} }

\newcommand{\VLC}{ {V_{\LC}}}
\newcommand{\VLCd}{{ V_{ L_{ \cC \times \bm{\delta} }}}}
\newcommand{\VLtl}{ { ( V_L^\tau)^{ \otimes \ell} }}
\newcommand{\Ceqt}{\cC_{\equiv \tau}}
\newcommand{\bom}{{ \bar{\omega}}}

\newcommand{\fusion}[3]{{\binom{#3}{#1,\;#2}}}
\newcommand{\fuCD}{ N_{ \cC \times \cD}\fusion}
\newcommand{\fuCDt}{N^\tau_{ \cC \times \cD}\fusion}
\newcommand{\fult}{N_ \otimes \fusion}
\newcommand{\fut}{N_\circ^ \tau\fusion}

\newcommand{\abs}[1]{\left\vert #1 \right\vert}

\DeclareMathOperator{\Syl}{Syl}

\DeclareMathOperator{\Suz}{Suz}

\DeclareMathOperator{\Ker}{Ker}
\DeclareMathOperator{\Hom}{Hom}

\DeclareMathOperator{\Ct}{Cent}
\DeclareMathOperator{\Span}{Span}

\mathtoolsset{centercolon}

\begin{document}

\title{Quantum dimensions and fusion rules of the VOA $ V^\tau_\LCD$ }

 \author{Hsian-Yang Chen}
  \address[H.Y.  Chen]{ Institute of Mathematics, Academia Sinica, Taipei  10617, Taiwan}
\email{hychen@math.sinica.edu.tw}
 \author{Ching Hung Lam}
 \address[C.H. Lam]{ Institute of Mathematics, Academia Sinica, Taipei  10617, Taiwan
 and National Center for Theoretical Sciences, Taiwan}
\email{chlam@math.sinica.edu.tw}



\begin{abstract}
In this article, we determine  quantum dimensions and   fusion rules for  the orbifold code VOA $ V^\tau_\LCD$.  As an application,  we also construct certain $3$-local subgroups inside the automorphism group of the VOA $V^\sharp$, where $V^\sharp$ is a holomorphic VOA obtained by the $\mathbb{Z}_3$-orbifold construction on the Leech lattice VOA.
\end{abstract}
\maketitle

\section{Introduction}

The study of vertex operator algebras as modules of Virasoro VOA 
was first initiated by Dong, Mason and Zhu\,\cite{DMZ}. They proved that the Moonshine VOA $V^{\natural }$ contains $48$
mutually orthogonal elements such that each of them will generate a copy of
the rational Virasoro VOA $L\left( \frac{1}{2},0\right) $ inside $V^{\natural }$ and the sum of these $48$ conformal vectors is the Virasoro
element of $V^{\natural }$.  This discovery turns out to be very important for the study of the Moonshine VOA \cite{D,M2}.  It also leads to the development of the theory of framed vertex operator algebras (cf. \cite{M1,M2} and \cite{DGH}). Roughly speaking, a framed VOA is a simple VOA which contains a full sub VOA $F\cong L(\frac{1}2,0)^{\otimes n}$ such that $\mathrm{rank}(V)=\mathrm{rank}(F)=n/2$ .  There are many interesting examples, which include the famous Moonshine VOA. Moreover, it is known that if $V$ is a framed VOA with the weight one subspace $V_1=0$, then the full automorphism group $\Aut(V)$ is finite \cite{M2,GL}. Therefore, the theory of framed VOAs is very useful for studying certain finite groups such as the Monster.

In \cite{DMZ}, the Virasoro VOA $L(\frac{1}2, 0)$  was constructed inside the lattice type VOA $V_{\mathbb{Z}\alpha}^+$, where $\langle \alpha, \alpha \rangle =4$. In fact, $V_{\mathbb{Z}\alpha}^+ \cong V_{\sqrt{2}A_1}^+ \cong L(\frac{1}2,0)\otimes L(\frac{1}2,0)$. Therefore, a framed VOA with integral central charge $k$ may also be considered as an extension of
the tensor product of the orbifold VOA $V_{\sqrt{2}A_1}^+$.  In this article, we consider a generalization of framed VOAs. Namely, we replace the VOA $V_{\sqrt{2}A_1}^+$ by another orbifold VOA $(V_{\sqrt{2}A_2}^\tau)$, where $\tau$ is a lift of a fixed point free isometry of order $3$ in $O(\sqrt{2}A_2)$, and study certain extensions of the VOA $(V_{\sqrt{2}A_2}^\tau)^{\otimes n}$. We first study a certain integral lattice $\LCD$ that are constructed from an $\mathbb{F}_4$-code $\cC$ and an  $\mathbb{F}_3$-code $\cD$ as an extension of the lattice $(\sqrt{2}A_2)^{\oplus n}$. We also study the irreducible  modules for the orbifold VOA $V_{\LCD}^\tau$. As our main result, we determine the quantum dimensions and the fusion rules for all irreducible $V_{\LCD}^\tau$-modules. In particular, we show that all irreducible $V_{\LCD}^\tau$-modules are simple current modules if the $\mathbb{F}_4$-code $\cC$ is self-dual. Moreover, the fusion ring for $V_{\LCD}^\tau$ is isomorphic to a group ring of an elementary 
abelian $3$-group and the set of all inequivalent irreducible $V_{\LCD}^\tau$-modules forms a quadratic space over $\mathbb{F}_3$ if $\cC$ is self-dual.

As an application, we study the case when $\cC$ is isomorphic to the Hexacode in detail. In this case, $L_{\cC\times \{0\}}$ is isomorphic to the Coxeter-Todd lattice $K_{12}$ of rank $12$. We show that the full automorphism group of $V_{K_{12}}^\tau$ is isomorphic to $\Omega^-_8(3).2$. Several $3$-local subgroups of the VOA $V^\sharp$, obtained from
$\mathbb{Z}_3$-orbifold construction using the Leech lattice VOA, are also studied and computed explicitly.

This article is organized as follows. In Section 2, we review some basic properties of the VOA $V_{\sqrt{2}A_2}^\tau$ and the notion of quantum dimensions. In Section 3, we review a construction of the integral lattice $\LCD$ from some $\mathbb{F}_4$ and $\mathbb{F}_3$-codes. Some basic facts about the lattice VOA $V_\LCD$ and its $\mathbb{Z}_3$-orbifold $V_\LCD^\tau$ will also be recalled. In Section 4, we compute the quantum dimensions of the orbifold VOA $ V^ \tau_{\LCD}$. In Section 5, we compute the fusion rules among irreducible $V^ \tau_{\LCD}$-modules.  As an application,   we construct in Section 6 certain $3$-local subgroups inside the automorphism group of the VOA $V^\sharp$, where $V^\sharp$ is a holomorphic VOA obtained by a $\mathbb{Z}_3$-orbifold construction on the Leech lattice VOA.

\section{Preliminaries and basic properties}

\paragraph{\textbf{The VOAs $V_{\sqrt{2}A_2}$ and $ V_{\sqrt{2}A_2}^\tau$}}

In this section we review some facts about the orbifold VOA $V_{\sqrt{2} A_2}^\tau$. For general background concerning  lattice VOA, we refer to~\cite{FLM,LL}.

Let $ \alpha_1, \alpha_2$ be the simple roots of type $A_2$ and set
$ \alpha_0=-( \alpha_1+ \alpha_2)$. Then $ \langle  \alpha_i, \alpha_i \rangle  = 2$ and $ \langle \alpha_i, \alpha_j  \rangle =-1$ if $i \ne j$, $i,j\in \{0,1,2\}$. Set $ \beta_i=\sqrt{2} \alpha_i$ and
let $L= \mathds{Z} \beta_1+ \mathds{Z} \beta_2$ be the lattice spanned by $ \beta_1$ and
$ \beta_2$. Then  $L$ is isometric to $\sqrt{2}A_2$.

Let  $ \GF  = \{ 0, 1, \omega, \bom \}$ denote the Galois field of four elements, where $ \omega$ is a   root of  $ x^2 + x +1 =0$ over $\mathbb{F}_2$. 
We adopt the similar notation as in \cite{KLY03, DLTYY04} and denote the cosets of $L$ in the dual lattice  $L^\perp = \{ \alpha \in \mathds{Q} \otimes_{ \mathds{Z} }  L \,|\, \langle \alpha, L \rangle   \subset \mathds{Z} \}$, as follows:
\begin{equation*}
L^0=L,\quad  L^1=\frac{- \beta_1+ \beta_2}{3}+L ,\quad
L^2=\frac{ \beta_1- \beta_2}{3}+L,
\end{equation*}
\begin{equation}\label{eq:def:L_i}
L_0=L,\quad L_1=\frac{ \beta_2}{2}+L,\quad
L_ \omega=\frac{ \beta_0}{2}+L,\quad L_{\bom}=\frac{ \beta_1 }{2}+L,
\end{equation}
and
\begin{equation*}
L^{(i,j)} = L_i + L^j,
\end{equation*}
for $i=0,1, \omega, \bom$ and $j=0,1,2$.  Then, $L^{(i,j)},i \in \GF, j\in \mathds{Z}_3 = \{0,1,2\}$
are all the cosets of $L$ in $ L^{\perp }$ and $L^\perp/L \cong
\mathds{Z}_2 \times \mathds{Z}_2 \times \mathds{Z}_3$.
It is shown in ~\cite{Dong93} that there are exactly $12$ isomorphism classes of
irreducible $V_L$-modules, which are given by
$V_{L^{(i,j)}}$, $i=0,1, \omega, \bom$ and $j=0,1,2$.

Consider the isometry $ \tau : L \to L $ defined by
\[
  \beta_1 \mapsto \beta_2 \mapsto \beta_0 \mapsto \beta_1.
\]
Then $ \tau$ is fixed point free of order three and can be lifted naturally to an automorphism of $V_L$ by mapping
\begin{equation*}
 a^1( -n_1) \cdot a^k( - n_k) e^b \mapsto (\tau a^1)( -n_1) \cdot ( \tau a^k) ( - n_k) e^{ \tau b}.
\end{equation*}
By abuse of notation, we also use $ \tau$ to denote the lift.

In~\cite{KLY03}, it was shown that  there are exactly three
irreducible $ \tau$-twisted $V_L$-modules and  three irreducible $ \tau^2$-twisted $V_L$-modules, up to isomorphism. They are denoted by $ V_L^{ T, j}( \tau) $ or $ V_L^{ T, j}( \tau^2) $ for $j = 0, 1,2$.

The automorphism $ \tau$ acts on the set of inequivalent irreducible $V_L$-modules by $ V_{L^{(i, j)}} \circ \tau$. Note also that $ V_{L^{(i, j)}} \circ \tau \cong V_{L^{( \bar{\omega} i, j)}}$.  We denote
\begin{align*}
  U[\varepsilon] = \{u \in U\,|\, \tau u =\exp(2\pi\sqrt{-1}\varepsilon/3)  u\},
\end{align*}
for any $\tau$-invariant $V_L$-module $U$ and $\varepsilon=0,1,2$. Irreducible modules for the orbifold VOA $V_L^ \tau$ are classified by Tanabe and Yamada~\cite{TY07} and the following result was proved.

\begin{prop}\cite{TY07}
 The VOA $V_{L}^{\tau}$ is a simple, rational, $C_2$-cofinite, and of CFT type.
There are exactly $30$ inequivalent irreducible $V_{L}^{\tau}$-modules.
They are represented as follows.
\begin{enumerate}[label=(\roman*)]
   \item  $V_{L^{(0,j)}}[\varepsilon]$ for $j,\varepsilon=0,1,2$.
   \item  $V_{L^{( \omega,j)}}$ for $j=0,1,2$.
   \item  $V_L^{T,j}(\tau^i)[ \varepsilon]$  for $i=1,2$ and $j,\varepsilon=0,1,2$.
\end{enumerate}
Weights of these modules are given by given by (see Tanabe and Yamada\cite[(5,10)]{TY07}):
  \begin{align*}
 \wt V_{ L^{0, j }} [ \varepsilon ] &\in  \frac{ 2 j^2 }{3} + \mathds{Z}, \\
 \wt V_L^{T, j}( \tau^i ) [ \varepsilon ] &\in  \frac{ 10 -  3 (  j^2 + \varepsilon ) }{9 } + \mathds{Z},
\end{align*}
for $ i = 1, 2, j, \varepsilon \in \mathds{Z}_3$.	
\end{prop}

\paragraph{\textbf{Quantum Dimension}}

 We now review the notion of quantum dimension introduced by Dong et al.~\cite{DJX12}.

Let $V$ be a VOA of central charge $c$ and let $ M = \oplus_{ n \in \mathds{Z}_+} M_{ \lambda +n} $ be a $V$-module, where $ \lambda$ is the lowest conformal weight of $M$. The \emph{normalized character} of $M$ is defined as
\begin{equation*}
     \ch M(q) := q^{\lambda-c/24}\sum_{n\in \mathds{Z}_+} \dim {M_{ \lambda +n}} q^n,
\end{equation*}
where $ q=e^{2\pi \sqrt{-1} z }$ and $ z = x +  \sqrt{-1} y$ is in the complex upper half-plane $ \mathbb{H}$.

The following notion of quantum dimension is introduced by Dong et al.~\cite{DJX12}.
\begin{defi}
   Suppose $  \ch V(q) $ and $  \ch M(q) $ exist. The \emph{quantum dimension of $M$ over $V$} is defined as
   \begin{equation}
      \qdim_V M := \lim_{ y \to 0^+} \frac{ \ch M (\sqrt{-1} y)}{ \ch V ( \sqrt{-1} y)},
   \end{equation}
where $ y$ is a positive real number.
\end{defi}

From now on, we will omit the variable $q$ and write the character $ \ch M(q)$ as $ \ch M$ instead.
Fundamental properties of quantum dimension are also proved in their paper.

\begin{prop}\cite{DJX12}\label{thm:prop_qdim}
Let $V$ be a simple, rational,  $C_2$-cofinite VOA of  CFT-type and  $V \cong V^\prime$.    Let $W, W^1, W^2$ be $V$-modules. Then
\begin{enumerate}[label=(\roman*)]
   \item $  \qdim_{V} W \ge 1$.
   \item $\qdim_V$ is multiplicative, that is $ \qdim_V ( W^1 \times W^2) =\qdim_V  W^1 \cdot \qdim_V W^2$, where $ W^1 \times W^2$ denotes the fusion product.
   \item A $V$-module $W$ is a simple current  if and only if $ \qdim_V W^1 =1$.
   \item  $ \qdim_V W = \qdim_V W^\prime$, where $ W ^\prime $ is the contragredient dual of $W$.
\end{enumerate}
\end{prop}
\begin{rmk}
   Recall that    a simple $V$-module $M$ is a \emph{simple current} module if and only if for every simple $V$-module $W$, $M \times W$ exists and is also a simple $V$-module.
\end{rmk}

Quantum dimensions of irreducible $V_L^\tau$-modules are computed in~\cite{C13a}.
\begin{prop}\cite{C13a}\label{thm:qd_VL0}
We have
 \begin{enumerate}[label=(\roman*)]
   \item  $\qdim_{V_L^\tau} V_{L^{(0,j)}}[\varepsilon]=1$ for $j,\varepsilon=0,1,2$.
   \item  $\qdim_{V_L^\tau} V_{L^{( \omega,j)}}=3$ for $j=0,1,2$.
   \item  $\qdim_{V_L^\tau} V_L^{T,j}(\tau^i)[ \varepsilon]=2$ for $i=1,2$ and $j,\varepsilon=0,1,2$.
\end{enumerate}
\end{prop}

\section{The VOAs $\VLCD$ and $ V^\tau_\LCD$}
\paragraph{\textbf{$ \mathds{Z}_3$ and $\GF$-codes}}
We first review the coding  theory concerned in this paper. All codes mentioned in this paper are linear codes. From now on, we fix $ \ell \in \mathds{N}$.

Let $\bm{ \lambda } = ( \lambda_1, \cdots, \lambda_\ell)$ be a codeword of length $\ell$, its \emph{support} is defined to be
$\supp( \bm{ \lambda }) = \{ i\,|\, \lambda_i \ne 0\}$. The cardinality of
$\supp( \bm{ \lambda })$, denoted by $\wt( \bm{ \lambda })$, is called the \emph{(Hamming) weight} of $ \bm{ \lambda }$.
A code $\cS$ is said to be \emph{even} if $\wt( \bm{ \lambda })$ is even for every $ \bm{ \lambda } \in \cS$.
Let $ \cS $ be a code of length $ \ell$. The \emph{(Hamming) weight enumerator} of $ \cS$ is defined to be
   \begin{equation}
      W_{ \cS} ( X, Y) = \sum_{ \bm{ \lambda} \in \cS} X^{ \ell - \wt( \bm{ \lambda})} Y^{ \wt(  \bm{ \lambda})},  
   \end{equation}
which is a homogeneous polynomial of degree $\ell$.

We consider the inner products for codes over  $ \GF$ and $  \mathds{Z}_3$  as follows.
For codes over  $ \GF$, we use the Hermitian inner product, i.e,
$$ \bm{ x } \cdot \bm{ y } := \sum_{i=1}^\ell x_i  \bar{ y_i}$$
for $ \bm{x} = ( x_1, \cdots, x_\ell), \bm{y} = (y_1, \cdots, y_\ell) \in \GF^\ell$, where $ \bar{x} = x^2$ is the  \emph{conjugate} of $ x \in \GF$. For $\mathbb{Z}_3$-codes, we use the usual Euclidean inner product: $$ \bm{ x } \cdot \bm{ y } := \sum_{i=1}^\ell x_i   y_i \quad \text{ for } \bm{ x } , \bm{ y } \in \mathds{Z}_3^\ell.$$


Let $ \cK = \GF $ or $ \mathds{Z}_3$. For a $ \cK$-code $ \cS$ of length $ \ell$ with inner product given as above, we define its dual code by
\begin{align*}
 \cS^\perp = \{ \bm{ \lambda}  \in \cK^\ell \mid  \bm{   \lambda } \cdot \bm{  \mu}   = 0 \textrm{ for all }  \bm{ \mu}  \in \cS \}.
\end{align*}
A $\cK$-code $\cS$ is said to be \emph{self-orthogonal} if $\cS \subset \cS^\perp$
and \emph{self-dual} if $\cS = \cS^\perp$.


\begin{rmk}
By ~\cite[Thm.1.4.10]{HP:FEC}, an $\mathbb{F}_4$-code $\cC$  is even if and only if $\cC$ is Hermitian self-orthogonal. Note that the underlying ``additive" group structure of $\GF$ is $\mathbb{Z}_2\times \mathbb{Z}_2$. Therefore, an even $\mathbb{F}_4$-code $\cC$ is also an even ``additive'' $ \mathds{Z}_2 \times \mathds{Z}_2$ code. Moreover, $\cC$ is $\tau$-invariant since it is $\mathbb{F}_4$-linear.
In Tanabe and Yamada~\cite{TY07}, even $\tau$-invariant $ \mathds{Z}_2 \times \mathds{Z}_2$ codes are used. Instead of the Hermitian inner product, they used the trace Hermitian inner product defined by $ \bm{ x } \cdot \bm{ y } := \sum_{i=1}^\ell x_i  \bar{ y_i} + \bar{x_i} y_i$.

In the notation of ~\cite{RS98}, our code $ \cC$   belongs to the family $ 4^{ H}$, while Tanabe and Yamada's code belongs to the family $ 4^{ H+}$. If the code $ \cC$ is also linear, then its dual $ \cC ^\perp$ in $ 4^{ H+}$ coincides with the dual of $ \cC$ in $ 4^{ H}$. 
Therefore, these two notions are essentially the same and almost all theorems we proved in this paper have analogous statements in their setting.
\end{rmk}

\paragraph{\textbf{The lattice $L_{\cC\times \cD}$ and the VOAs $\VLCD$ and $ V^\tau_\LCD$}} In this paper, we use  a boldface  lowercase letter $ \bm{ x }$ to denote a vector or a sequence of length $ \ell$ and its $i$-th coordinate is denoted by $ x_i$. That is
\begin{align*}
 \bm{ x } = ( x_1, \cdots, x_ \ell).
\end{align*}

From now on, we let $ \cC$ be a  self-orthogonal $ \GF$-code of length $ \ell$ and let $ \cD$ be a self-orthogonal $ \mathds{Z}_3$-code of the same length. First we review a construction of an even lattice from $\cC$ and $\cD$ \cite{KLY03,TY06}.

For $ \bm{ \lambda} = ( \lambda_1, \cdots, \lambda_\ell) \in \GF^\ell$ and $ \bm{   \delta }  = ( \delta_1, \cdots, \delta_\ell) \in \mathds{Z}_3^\ell$, we define
\begin{equation*}
   L_{ \bm{ \lambda} \times \bm{  \delta } } := \{ (x_1,\dots, x_\ell)\in ( L ^\perp)^{ \oplus \ell}\mid x_i\in L^{( \lambda_i, \delta_i)}, i=1, \dots, \ell\}.
\end{equation*}
For subsets $ \bm{  P} \subset \GF^\ell$ and $ \bm{ Q} \subset  \mathds{Z}_3^\ell$,  we define 
\begin{equation*}
   L_{\bm{  P} \times \bm{ Q}} := \bigcup_{ \bm{ \lambda} \in \bm{  P}, \bm{  \delta } \in \bm{ Q}} L_{ \bm{ \lambda}  \times \bm{   \delta }} \subset ( L ^\perp)^{ \oplus \ell}.
\end{equation*}
Let $ \tau$ acts diagonally on $ (L ^\perp)^{ \oplus \ell}$ and hence it induces an action on $ V_{ (L ^\perp)^{ \oplus \ell}}$. The purpose of this paper is to determine  quantum dimensions and fusion rules of  irreducible $ V^\tau_\LCD$-modules. 

\begin{prop}[\cite{TY06}]
Let $ \cC$ be a  self-orthogonal $ \GF$-code of length $ \ell$ and $ \cD$ be a self-orthogonal $ \mathbb{Z}_3$-code of the same length.  Then the subset $ \LCD$ is an even sublattice of $( L ^\perp)^{ \oplus \ell}$. Moreover, the dual lattice $(\LCD)^\perp = L_{\cC^\perp \times \cD^\perp}$.
\end{prop}

\begin{prop} \cite{Dong93,Dl96,TY06} Let $ \cC$ be a  self-orthogonal $ \GF$-code of length $ \ell$ and $ \cD$ be a self-orthogonal $ \mathds{Z}_3$-code of the same length. Let $\VLCD$ be the lattice VOA associated to $\LCD$. Then we have the following.
\begin{enumerate}[label=(\roman*)]
   \item The set of all inequivalent irreducible $\VLCD$-modules is given by
    \begin{equation*}
       \{  V_{ L_{  ( \bm{ \lambda} +  \cC) \times  ( \bm{    \delta  } +    \cD) }} \mid \bm{ \lambda} + \cC \in \cC ^\perp / \cC, \bm{ \delta}  + \cD \in \cD ^\perp/ \cD \}.
    \end{equation*}
    \item We have $ V_{ L_{  ( \bm{ \lambda} +  \cC) \times  ( \bm{  \delta } +    \cD) }} \circ \tau \cong V_{ L_{  ( \tau ^{-1}( \bm{ \lambda} ) +  \cC) \times  (   \bm{ \delta } +    \cD) }}$.
    \item For $ i = 1,2$, there are exactly $ \left| \cD ^\perp / \cD \right|$ inequivalent irreducible  $ \tau^i$-twisted $\VLCD$-modules. They are represented by $ ( V^{T,\bm{ \eta}}_\LCD( \tau^i), Y^{ \tau^i})$ for $ \bm{ \eta}  \in \cD ^\perp  \bmod \cD $.
\end{enumerate}
\end{prop}
\begin{rmk}
    As $ \tau$-twisted $V_{L^{ \oplus L} }$-modules,
    \begin{align*}
  V^{T,\bm{ \eta}}_\LCD( \tau ) \cong \bigoplus_{ \cgg \in \cD} V_{L^{ \oplus L}}^{T, \cee - \cgg}( \tau),
\end{align*}
for $ i =1,2.$
Furthermore, we have the following  decomposition of 
$V_{L^{\oplus\ell}}^{T, \cee}(\tau)$ into a direct sum of 
irreducible $(V_L^{\tau})^{\otimes\ell}$-modules.
\begin{align}
V_{L^{\oplus\ell}}^{T,\cee }(\tau)&\cong\bigoplus_{( \varepsilon_1, \ldots, \varepsilon_\ell)
\in \mathds{Z}_3^\ell} \vt{L}{\eta_1}{\tau}{ \varepsilon_1} \otimes \cdots
\otimes \vt{L}{ \eta_\ell}{\tau}{ \varepsilon_\ell}.
\label{eq:deco-T}
\end{align}
Similar for $ \tau^2$-twisted modules. 
\end{rmk}

Since $ \tau$ acts trivially on $ \cD$,
\begin{align*}
V_{ L_{  ( \bm{ \lambda} +  \cC) \times  ( \bm{  \delta} +    \cD) }}  \cong V_{ L_{  ( \bm{ \lambda ^\prime} +  \cC) \times  ( \bm{  \delta ^\prime} +    \cD) }}
\end{align*}
if and only if
(1) $ \bm{ \lambda} + \cC $ and  $ \bm{ \lambda ^\prime} +  \cC$ belong to the same $ \tau$-orbit of $ \cC ^\perp$; and  (2) $ \bm{  \delta } +    \cD  = \bm{  \delta ^\prime} +    \cD $ in $ \cD ^\perp/ \cD$.
Let  $ \cC ^\perp_{ \equiv \tau}$ denote the set of all $ \tau$-orbits in $ \cC ^\perp$. Then
\begin{align*}
\{   V_{ L_{    \cC \times  ( \bm{  \delta } +    \cD) }} [ \varepsilon ], V_{ L_{  ( \bm{ \lambda} +  \cC) \times  ( \bm{  \delta} +    \cD) }} \mid \cZ \neq \bm{ \lambda} + \cC \in \cC _{ \equiv \tau}^\perp\bmod \cC,  \bm{  \delta}  + \cD \in \cD ^\perp/\cD \}
\end{align*}
is a set of inequivalent irreducible $ V^\tau_\LCD$-modules, which are obtained from the irreducible (untwisted) $ V_\LCD$-modules.

It is usually very difficult to classify all irreducible modules of an orbifold VOA. Recently,  Miyamoto gave a classification in the $ \mathds{Z}_3$-orbifold case.

\begin{prop}~\cite[Thm.A]{Mi13a}
Let $V$ be a simple VOA of CFT-type. Assume $ V \cong V ^\prime$, its contragredient dual, and all $V$-modules are completely reducible. If  $ \sigma$ is an automorphism of $V$ of   order three and a fixed point
subVOA  $V^ \sigma$ is  $C_2$-cofinite, then all $V^ \sigma$-modules are completely reducible. Moreover, every
simple $V^ \sigma$-module appears as a $ V^ \sigma$-submodule of some $ \sigma^j$-twisted (or ordinary) $V$-module.
\end{prop}

By this proposition,  we can classify irreducible $ V^ \tau_\LCD$-modules.
\begin{prop}
 An irreducible $ V^ \tau_\LCD$-module must belong to one of the following types.
 \begin{align*}
 \mathrm{ (i) }& V_{ L_{    \cC \times  (  \cdd +    \cD) }} [ \varepsilon ],&
 \mathrm{ (ii) }& V_{ L_{     ( \bm{ \lambda} +  \cC)  \times  ( \cdd  +    \cD) }},&
 \mathrm{ (iii) }&V^{T, \bm{ \eta} }_\LCD( \tau^i) [ \varepsilon],
\end{align*}
where $i =1,2$, $ \varepsilon \in \mathds{Z}_3$,   $ \cZ \neq \bm{ \lambda}  + \cC \in  \cC ^\perp_{\equiv \tau} \bmod \cC$  , $ \bm{  \eta}  \in     \cD ^\perp (\bmod \cD)$ and $\bm{  \delta } +    \cD \in  \cD ^\perp/\cD$.
 In particular, there is no modules of the second type (ii)  if $ \cC$ is self-dual.
\end{prop}

In this paper, our calculations depend heavily on the decomposition of the irreducible  $ V^ \tau_\LCD$-modules as $ \VLtl$-modules.
\begin{prop}[\cite{KLY03,TY06}]\label{thm:decomp:into:L}
As modules of $\VLtl$, we have the following decomposition.
\begin{subequations}
\begin{align}
&  \mathrm{ (i)}&    V_{L_{ \cC \times ( \bm{ \delta} + \cD)}}[ \varepsilon]   &\cong
  V_{L_{\cZ \times ( \bm{ \delta} + \cD)}}[ \varepsilon] \oplus \bigoplus_{ \cZ \neq \bm{ \gamma} \in   \Ceqt} V_{L_{ \bm{ \gamma}  \times ( \bm{ \delta}  + \cD)}};\\
&  \mathrm{ (ii)}& V^{ T, \bm{\eta}}_\LCD ( \tau^i)  [ \varepsilon ] &\cong
    \bigoplus_{  \cdd  \in \cD}
    \bigoplus_{e_1 +  \cdots +  e_\ell \equiv \varepsilon  \bmod 3 }
    V_L^{ T, \eta_1 - i \delta_1} ( \tau^i) [ e_1] \otimes \cdots \otimes V_L^{ T, \eta_\ell - i \delta_\ell} ( \tau^i) [ e_\ell],
\end{align}
\end{subequations}
where $ \cdd = ( \delta_1, \cdots, \delta_ \ell)\in \mathbb{Z}_3^\ell$ and  $ \cC_{ \equiv \tau}$ denote the set of all $ \tau$-orbits in $ \cC$. 
 \end{prop}

In order to using properties about quantum dimensions, we need  the following proposition.
\begin{prop}\label{thm:orbV:self_dual}
The VOAs $V_{\LCD } $ and $ V^ \tau_{\LCD }$
   are  simple, rational, $C_2$-cofinite VOAs of CFT type and are isomorphic  to  their contragredient dual, respectively.
\end{prop}
Assertions about lattice VOA is well-known.
The  simplicity  and rationality of  $ V^ \tau_{\LCD }$ are proved in \cite{TY06}.  The  $C_2$-cofiniteness of  $ V^ \tau_{\LCD }$ is given in \cite[Thm B]{Mi13a}. By~\cite[Cor.3.2]{L94}, $ V^ \tau_{\LCD }$ is self-dual.


\section{Quantum dimensions of irreducible $ V^ \tau_{\LCD}$-modules}
In this section, we compute the quantum dimensions of irreducible $ V^\tau_{\LCD}$-modules. We will first consider the case when $ \cD = \{ \cZ\}$ is the trivial $ \mathds{Z}_3$-code. Results in this case are summarized in Theorem~\ref{thm:qd:CZ}.    The general case will be considered in Section~\ref{sec:qd:CD} (see Theorem~\ref{thm:qd:CD}). In addition, we verify one conjecture about global dimensions proposed by Dong et. al. for the VOA $  V^\tau_{\LCD }$ in this section.

\medskip

\paragraph{\textbf{Weight enumerators}}

For $ \varepsilon = 0,1,2$,  let
\begin{align*}
  S_\varepsilon := \{  \bm{ x } := ( x_1, \cdots, x_\ell) \in  \mathds{Z}_3^\ell \mid \sum x_i \equiv \varepsilon \bmod 3 \}.
\end{align*}
We define the \emph{weight enumerator} as follows. Let
\begin{equation}\label{eq:def_WE:congr}
   W_{\varepsilon}(X,Y) : =\sum_{  \bm{ x } \in S_ \varepsilon} X^{ \ell -\wt(  \bm{ x }  )} Y^ { \wt ( \bm{ x }  )},
\end{equation}
where $ \wt ( \bm{ x }) $ is the number of nonzero coordinates of $ \bm{ x }$. That is,
\begin{equation*}
   \wt ( \bm{ x })  := \# \{ x_i \mid x_i \neq 0 \textrm{ for  } 1 \le i \le \ell \}.
\end{equation*}

We   also consider a weight enumerator induced from an $\GF$-code $ \cC$.
Let $ W_{ \cC} ( X, Y)  $ denote the Hamming weight enumerator, we define
\begin{equation}\label{eq:def:W/tau}
   W_{\cC}^\prime (X, Y) := \frac{1}{3} ( W_{ \cC} ( X, Y) - X^\ell).
\end{equation}
Note that $ W_ \varepsilon(X,Y), W_{\cC}^\prime(X,Y) $ are homogeneous polynomials in $ X, Y$ of the same degree $\ell$.

\begin{lem}\label{thm:W_3:11}
   The self-orthogonal $\GF$-code  $ \cC$ is self-dual if and only if  $ W_{\cC}^\prime(1,1) = \frac{ 2^\ell -1}{3} $.
\end{lem}
 \begin{proof}
 First we note that $W_{\cC}(1,1)=|\cC|$ is equal to the number of elements in $\cC$; hence
\[
       W_{\cC}^\prime(1,1) = \frac{ W_{\cC}(1,1) -1 }{3} =  \frac{ \left| \cC \right| -1 }{3}.
\]
Since $ \cC$ is self-orthogonal,  we know $ \cC ^\perp \supset \cC$ and $ \dim \cC ^\perp + \dim \cC = \ell$. Therefore, $|\cC|\leq 2^\ell$ and the equality holds if and only if $\cC$ is self-dual.  The lemma now follows.
 \end{proof}

The following lemmas   explain why we introduce these weight enumerators.
Recall that the module $ \VLCd[ \varepsilon]$ admits a decomposition of $ ( V^ \tau_L) ^{ \otimes \ell }$-modules as
\begin{equation}
 \VLCd[ \varepsilon] \cong V_{L_{\cZ \times \bm{ \delta}}}[ \varepsilon] \oplus \bigoplus_{ \cZ \neq \bm{ \gamma} \in   \Ceqt} V_{L_{ \bm{ \gamma} \times \bm{ \delta}}},
\end{equation}
where    $ \cdd   \in  \mathds{Z}_3^\ell$ and $ \Ceqt  $ denotes the set of all orbits of $ \tau $ in $ \cC$. In particular, when $ \cdd =\cZ$,  we have   $ V_{L_{ \cZ \times \cZ}}[ 0] \cong \big( V_{L ^{ \otimes \ell}}  \big) ^\tau$, which should not be confused with the subVOA $ \VLtl \subsetneq \big( V_{L ^{ \otimes \ell}}  \big) ^\tau$.

\begin{lem} For $ \varepsilon = 0,1,2$,
   the character of $ V_{L_{ \cZ \times  \cZ}} [ \varepsilon] $ is given by
 \begin{align*}
\ch V_{ L^{ \otimes \ell }} [ \varepsilon ] =& W_{ \varepsilon}( Z_0(q), Z_1(q)),
\end{align*}
where
\begin{equation}
\begin{aligned}
Z_0(q):=&\ch V_L[0],&   Z_1(q):=&\ch V_L[1] =\ch V_L[2].
\end{aligned}
\end{equation}
\end{lem}
\begin{proof}
For $ \varepsilon = 0,1,2$, we have a decomposition of $ (V_L^\tau)^ { \otimes \ell}$-modules:
   \begin{equation}
      V_{ L^{ \otimes \ell }} [ \varepsilon ]=
      \bigoplus_{ \sum r_i \equiv \varepsilon \bmod 3} V_L[ r_1] \otimes \cdots \otimes V_L[ r_\ell].
   \end{equation}
 We also know
   \begin{equation*}
      \ch \left(V_L[ r_1] \otimes \cdots \otimes V_L[ r_\ell]\right) = \ch V_L[ r_1] \times \cdots \times \ch V_L[ r_\ell] = Z_0^{ \ell -r} Z_1^{ r},
   \end{equation*}
   where $ r $ is the weight of $ \bm{ r } := ( r_1, \cdots, r_\ell) \in  \mathds{Z}_3^ \ell$.
Using the definition of weight enumerator given in~\eqref{eq:def_WE:congr}, we can rewrite
\begin{equation*}\begin{aligned}
\ch V_{ L^{ \otimes \ell }}[ \varepsilon] &=
\sum_{ \sum r_i \equiv \varepsilon \bmod 3} \ch V_L[ r_1] \times \cdots \times \ch V_L[ r_\ell] \\
&= \sum_{ \bm{ r }  \in S_\varepsilon }  Z_0^{ \ell  - \wt ( \bm{ r }) } Z_1^{ \wt( \bm{ r })}    = W_ \varepsilon ( Z_0, Z_1)
 \end{aligned} \end{equation*}
as desired.
\end{proof}

\begin{lem}
   We have the character
   \begin{equation*}\begin{aligned}
&\ch \left(\bigoplus_{ \cZ \neq \bm{ \gamma} \in   \Ceqt} V_{L_{ \bm{ \gamma} \times \cZ}}\right)
      = W_{\cC}^\prime( Y_0, Y_1),
 \end{aligned} \end{equation*}
   where
$Y_0 (q) :=  \ch V_{L^{( 0 ,  0 )}}$, and $Y_1(q) := \ch V_{L^{( 1,  0 )}}.$
\end{lem}
\begin{proof}
We first note that $ Y_1(q)  = \ch V_{L^{( x,  0 )}} $ for $ x = 1, \omega, \bom \in \GF.$
Let $ \cZ \neq \bm{ \gamma}   \in \Ceqt$. Then
\begin{equation*}
   \ch V_{L_{ \bm{ \gamma}  \times \cZ}} = \prod_{i} \ch V_{L^{ (\gamma_i, 0 )}} = Y_0^{ \ell - \wt( \bm{ \gamma})} Y_1 ^{\wt( \bm{ \gamma} ) }.
\end{equation*}
 We know the $ \tau$-orbit of $ \bm{ \gamma}  $ is the set $ \{ \bm{ \gamma} , \omega \bm{ \gamma}, \omega^2 \bm{ \gamma}  \}$, where $ \omega \bm{ \gamma} := ( \omega \gamma_1, \cdots, \omega \gamma_\ell)$. Note that   $ \omega \gamma_i = 0$ if and only if $ \gamma_i =0$. This means $ \wt \bm{ \gamma}  = \wt \tau \bm{ \gamma}$ and hence
 \begin{equation*}
    \ch V_{L_{ \bm{ \gamma}  \times \cZ}}  = \ch V_{L_{ \omega \bm{ \gamma} \times \cZ}}  =\ch V_{L_{  \omega^2 \bm{ \gamma}  \times \cZ}}.
 \end{equation*}
 Therefore, by the definition of ~\eqref{eq:def:W/tau} we have
 \begin{equation}
    \ch \left(\bigoplus_{ \cZ \neq \bm{ \gamma} \in   \Ceqt} V_{L_{ \bm{ \gamma} \times \cZ}}\right)
= \frac{1}{3} \sum_{ \cZ \neq \bm{ \gamma} \in \cC} \ch V_{L_{ \bm{ \gamma} \times \cZ}}
= \frac{1}{3} \sum_{ \cZ \neq \bm{ \gamma} \in \cC}  Y_0^{ \ell - \wt( \bm{ \gamma} )} Y_1 ^{\wt( \bm{ \gamma} ) }
    = W_{\cC}^\prime( Y_0, Y_1).
 \end{equation}
\end{proof}

\begin{prop}\label{thm:qd_VLCe}
   We have
   \begin{equation*}\begin{aligned}
 \ch \VLC[ \varepsilon]  =& W_ \varepsilon ( Z_0, Z_1) + W_{\cC}^\prime( Y_0, Y_1),
%
\end{aligned} \end{equation*}
   for $ \varepsilon = 0,1,2$.  Moreover, we have
   \begin{equation}
      W_\varepsilon( 1, 1)  = 3^{ \ell -1}.
   \end{equation}
\end{prop}
\begin{proof}
The first statement on  characters follows directly from the above two lemmas. Using a basic  combinatorial argument, it is easy to
show  $ W_i(1,1)= 3^{ \ell -1}$ for all $ 0\le i \le 2$; note that $S_\varepsilon= (\varepsilon, \dots,0)+ S_0$ for any
$\varepsilon=1,2$.
\end{proof}

\paragraph{\textbf{Quantum dimensions of  $V^\tau_\LC$-Modules}} 
We first compute the quantum dimensions of irreducible $V^\tau_\LC$-modules in  the case that the code $\cD = \{ \cZ \} \in  \mathds{Z}_3^\ell$ is the trivial code.
Note that in this case $ \cD ^\perp =  \mathds{Z}_3^\ell $.

\begin{prop}\label{thm:SD:VLCe:SC}
 For $ \varepsilon \in \mathds{Z}_3$ and $  \cdd \in  \mathds{Z}_3^ \ell$ ,  the  irreducible $ \VLC[0]$-module $ \VL{ \cC \times \cdd}[ \varepsilon] $ has the quantum dimension one.
\end{prop}
\begin{proof}
Using the same proof as in ~\cite[Lemma 3.2]{C13a},  the irreducible $ V^\tau_\LC$-modules $\VLC[ \varepsilon]$, $\varepsilon =
0,1,2$, are simple current and hence have quantum dimension  1.  Thus
 \begin{equation*}
  1= \qdim_{ \VLC[0] } \VLC[ \varepsilon ]    = \lim_{ y \to 0 ^+} \frac{ \ch \VLC[ \varepsilon]}{\ch \VLC[ 0 ]}.
 \end{equation*}

If $ \bm{ \delta} \neq  \cZ$, the computation is similar  with some modification.

 Fix $ 0\le \varepsilon \le 2$ and $ \cZ \neq \bm{ \delta} \in \cD^\perp$. Let
 \begin{equation*}\begin{aligned}
    &Z(q) := \frac{\ch V_{ L_{    \cC \times    \bm{ \delta}  }} [ \varepsilon ]   }{ \ch V_{ L_{    \cC \times    \cZ  }} [0]  }
    =\frac{ \ch V_{L_{\cZ \times \bm{ \delta}}}[ \varepsilon] + \sum_{ \cZ \neq \bm{ \gamma} \in   \Ceqt} \ch V_{L_{ \bm{ \gamma} \times \bm{ \delta}}}   }{ \ch V_{L_{\cZ \times \cZ}}[ 0 ] + \sum_{ \cZ \neq \bm{ \gamma} \in   \Ceqt} \ch V_{L_{ \bm{ \gamma}  \times \cZ}}  } \\
    =& \frac{ \sum_{ \bm{ r } \in S_{ \varepsilon }} \ch V_{L^{( 0, \delta_1)}}[ r_1] \times \cdots \times \ch V_{L^{( 0, \delta_1)}} [ r_\ell]  +  \frac{1}{3 } \sum_{ \cZ \neq \bm{ \gamma} \in \cC} \ch V_{ L^{ ( r_1, \delta_1)}} \times \cdots \times  \ch V_{ L^{ ( r_\ell, \delta_\ell)}} }
    { \sum_{  \bm{ r } \in S_{ 0} } \ch V_L[ r_1] \times \cdots \times \ch V_L [ r_\ell]  + \frac{1}{3 } \sum_{ \cZ \neq \bm{ \gamma}  \in \cC} \ch V_{ L^{ ( r_1, 0)}} \times \cdots \times  \ch V_{ L^{ ( r_\ell, 0)}}}.
 \end{aligned} \end{equation*}
 Dividing both denominator and numerator by $  (\ch V_L[0])^\ell$. We have
  \begin{equation*}\begin{aligned}
   Z(q)=& \frac{ \sum_{  \bm{ r } \in S_{\varepsilon }} \frac{\ch V_{L^{( 0, \delta_1)}}[ r_1]}{\ch V_L[0]} \times \cdots \times \frac{\ch V_{L^{( 0, \delta_1)}} [ r_\ell]}{\ch V_L[0]}
    +  \frac{1}{3 } \sum_{ \cZ \neq \bm{ \gamma}  \in \cC} \frac{\ch V_{ L^{ ( r_1, \delta_1)}}}{\ch V_L[0]}  \times \cdots \times  \frac{\ch V_{ L^{ ( r_\ell, \delta_\ell)}}}{\ch V_L[0]}  }
    { \sum_{ \bm{ r } \in S_{0}  } \frac{\ch V_L[ r_1]}{\ch V_L[0]}  \times \cdots \times \frac{\ch V_L [ r_\ell]}{\ch V_L[0]}   + \frac{1}{3 } \sum_{ \cZ \neq \bm{ \gamma} \in \cC} \frac{\ch V_{ L^{ ( r_1, 0)}} }{\ch V_L[0]}  \times \cdots \times
     \frac{\ch V_{ L^{ ( r_\ell, 0)}}}{\ch V_L[0]} }.
 \end{aligned} \end{equation*}
 Recalling  the quantum dimensions of $V_L[0]$-modules given in Prop.~\ref{thm:qd_VL0}, we have
 \begin{equation*}\begin{aligned}
\qdim_{ \VLC[0]}{ V_{ L_{    \cC \times    \bm{ \delta}  }} [ \varepsilon ]} = \lim_{ y \to 0^+} Z(q)
= \frac{ W_ \varepsilon (1,1) + W_{\cC}^\prime(1,1)}{ W_0 (1,1)   + W_{\cC}^\prime(1,1)}
= \frac{ 3^{ \ell -1} + W_{\cC}^\prime(1,1)}{ 3^{ \ell -1} + W_{\cC}^\prime(1,1)} =1
 \end{aligned} \end{equation*}
 as desired.
\end{proof}

\begin{prop}
  Let $ \varepsilon = 0,1,2$ and $ \cee \in \mathds{Z}_3 ^ \ell \comm{\cD ^\perp \bmod \cD}$.
  The irreducible $ \VLC[0]$-module  $   V ^{ T, \bm{ \eta} }_\LC ( \tau^i)  [ \varepsilon ] $ has the quantum dimension $ \frac{ 2^\ell}{  \left| \cC \right| }$ .
\end{prop}
\begin{proof}
   We know  an irreducible $ \VLC[0]$-module of twisted type admits a decomposition of $ (V_L^\tau)^ { \otimes \ell}$-modules:
   \begin{equation}\label{dec:VLC:twisted}
   V ^{ T, \bm{ \eta} }_\LC ( \tau^i)  [ \varepsilon ] \cong
   \bigoplus_{ \bm{ e } \in S_\varepsilon }
    V_L^{ T, \eta_1 } ( \tau^i) [ e_1] \otimes \cdots \otimes V_L^{ T, \eta_\ell } ( \tau^i) [ e_\ell].
\end{equation}
It was shown in \cite{C13a} that  $ \ch {V_L^{ T, j} ( \tau^i)[1] }  = \ch {V_L^{ T, j} ( \tau^k) [2]}$ and $ \ch {V_L^{ T, j} ( \tau^i)[0] } =
\ch {V_L^{ T, 0} ( \tau)[0] }$ for $ i,k =1,2$ and $ j = 0, 1,2$. Denote   $ T_0 (q):= \ch  {V_L^{ T, j} ( \tau^i)[0] }$ and $ T_1(q) := \ch
{V_L^{ T, j} ( \tau^i)[1] }$. Similar to the untwisted case, we have
\begin{equation}
   \ch V^{ T, \bm{  \eta} }_\LC ( \tau^i)  [ \varepsilon ]  = W_ \varepsilon ( T_0, T_1).
\end{equation}
%

Since $ W_ \varepsilon$ are homogeneous polynomials of degree $ \ell$, we have
\begin{equation}\begin{aligned}
&\qdim_{ V^\tau_\LC}  V^{ T, \bm{ \eta} }_\LC ( \tau^i)  [ \varepsilon ]  =
  \lim_{ y \to 0^+} \frac{ W_ \varepsilon ( T_0, T_1)}{ W_0(Z_0, Z_1) + W_{\cC}^\prime ( Y_0, Y_1)}\\
= &\frac{ W_ \varepsilon( 2,2)}{ W_0(1,1) + W_{\cC}^\prime (3,3)} = \frac{ 2^\ell W_ \varepsilon(1,1)}{ W_0(1,1 ) + 3^ \ell W_{\cC}^\prime (1,1)}.
 \end{aligned} \end{equation}
Note that all $V_L ^\tau$-modules  $V_L^{ T, j} ( \tau^i)[\varepsilon]$ have   quantum dimension 2.

Now by Thm.~\ref{thm:qd_VLCe} and Lemma~\ref{thm:W_3:11} we know
\begin{align*}\label{eq:qdim_twist_LC}
\qdim_{ V^\tau_\LC}   V^{ T, \bm{ \eta} }_\LC ( \tau^i)  [ \varepsilon ]
= \frac{ 2^\ell \cdot 3^{ \ell -1}}{ 3^{ \ell -1} +  3^{ \ell -1}   ( \left| \cC \right| -1)}
= \frac{ 2^\ell}{  \left| \cC \right|}.
\end{align*}
 \end{proof}

 \begin{rmk}
 Note that  $ \frac{ 2^\ell}{  \left| \cC \right|}  = \sqrt{ \left| \cC ^\perp/ \cC \right|}$ since $ | \cC  ^\perp |\cdot | \cC | = \GF^ \ell = ( 2^\ell)^2$.
\end{rmk}

\begin{cor}\label{thm:VLC:simp_curr:iff}
 Let $ \cC$ be a self-dual code.  Then all irreducible $ V^\tau_{\LC}$-modules are simple current modules. 
\end{cor}
\begin{proof}
If $ \cC$ is self-dual, then $ V^\tau_{\LC}$ has only two types of irreducible modules. Moreover,
\begin{align*}
\qdim_{ V^\tau_\LC}   V^{ T, \bm{ \eta} }_\LC ( \tau^i)  [ \varepsilon ] = \frac{ 2^\ell}{  \left| \cC \right|} =1
\end{align*}
by Lemma ~\ref{thm:W_3:11}. That means all irreducible modules of the type $ V^{ T, \bm{ \eta} }_\LC ( \tau^i)  [ \varepsilon ] $
are simple current modules.  By Proposition \ref{thm:SD:VLCe:SC}, the irreducible modules of the type $ \VL{ \cC \times \cdd}[
\varepsilon] $ are simple current modules, also.
%
\end{proof}

Now suppose $ \cC$ is self-orthogonal but not self-dual.  Then the quantum dimension of the  $ V^ \tau_{\LCD }$-module $V^{ T,
\bm{ \eta} }_\LC ( \tau^i)  [ \varepsilon ]$ is strictly greater than 1. In addition,   $ V^ \tau_{\LCD }$ has irreducible modules of
the type $V_{ L_{ ( \bm{ \lambda} +  \cC)  \times   \bm{  \delta} }}$.

\begin{prop}
   Let $  \cll + \cC \in \cC ^\perp/\cC$ and $ \cdd \in \mathds{Z}_3 ^ \ell$, we have
   \begin{align*}
\qdim_{ V^\tau_\LC} \VL{     ( \cll +  \cC)  \times   \cdd}  =  3.
\end{align*}
\end{prop}

\begin{proof}
 By definition,
 \begin{equation*}\begin{aligned}
 \qdim_{ V^\tau_\LC} V_{ L_{     ( \bm{ \lambda} +  \cC)  \times   \bm{  \delta}  }}
 = \lim_{ y \to 0^+} \frac{ \ch V_{ L_{     ( \bm{ \lambda} +  \cC)  \times   \bm{  \delta}  }} }{ \ch V^\tau_\LC }
 = \lim_{ y \to 0^+} \frac{ \sum_{ \bm{ \mu } \in \cC }\ch V_{ L_{     ( \bm{ \lambda} +  \bm{ \mu })  \times   \bm{  \delta}  }} }{ \ch V^\tau_\LC }.
 \end{aligned} \end{equation*}
Dividing both the denominator and the numerator by $  (\ch V_L[0])^\ell$. Since $ \qdim_{ V_L[0]} V_{ L^{( i,j )}} =3$ for any $ i \in
\GF, j \in \mathds{Z}_3$, we have
\begin{equation*}\begin{aligned}
\qdim_{ V^\tau_\LC} V_{ L_{     ( \bm{ \lambda} +  \cC)  \times   \bm{  \delta}  }}
=  \frac{ \left| \cC \right|  \cdot 3^\ell }{ 3^{ \ell -1} +  3^{ \ell -1}   (  \left| \cC \right|  -1)}
= 3
 \end{aligned} \end{equation*}
 as desired.
\end{proof}

To summarize, we have the theorem.
\begin{thm}\label{thm:qd:CZ}
   The  quantum dimensions for irreducible $ \VLC$-modules are as follows.
   \begin{enumerate}[label=(\roman*)]
 \item $ \qdim_{ V^\tau_\LC} V_{ L_{    \cC \times    \bm{ \delta}  }} [ \varepsilon ]  =  1$;
 \item $ \qdim_{ V^\tau_\LC} V_{ L_{     ( \bm{ \lambda} +  \cC)  \times   \bm{  \delta}  }}=  3$;
 \item $ \qdim_{ V^\tau_\LC}    V^{ T, \bm{ \eta} }_\LC ( \tau^i)  [ \varepsilon ]  = \frac{ 2^\ell}{ \left|  \cC \right|}$,
 \end{enumerate}
   where $i =1,2$, $ \varepsilon \in \mathds{Z}_3$,   $ \cZ \neq \bm{ \lambda}  + \cC \in  \cC ^\perp_{\equiv \tau} \bmod \cC$  and $\cee, \cdd    \in \mathds{Z}_3^ \ell$.
\end{thm}

\paragraph{\textbf{Quantum dimension of   $V^\tau_{\LCD}$-modules}}\label{sec:qd:CD}
We now deal with the general case. Let $ \cD$ be a self-orthogonal $ \mathds{Z}_3$-code. The basic idea is to  express the
characters of $V^\tau_{\LCD} $-modules in terms of the characters of $ V^\tau_{\LC}$-modules.

\begin{thm}\label{thm:qd:CD}
  The quantum dimensions of irreducible $V^\tau_{\LCD} $-modules are as follows.
   \begin{enumerate}[label=(\roman*)]
 \item   $ \qdim_{ V^\tau_\LCD} V_{ L_{    \cC \times     ( \cdd + \cD)  }} [ \varepsilon ]  =  1$;
 \item $ \qdim_{ V^\tau_\LCD} V_{ L_{     ( \bm{ \lambda} +  \cC)  \times   ( \bm{  \delta}  + \cD) }}=  3$;
 \item $ \qdim_{ V^\tau_\LCD}    V^{ T, \bm{ \eta} }_\LCD ( \tau^i)  [ \varepsilon ]  = \frac{ 2^\ell}{  \left|  \cC \right|}$,
 \end{enumerate}
   where $i =1,2$, $ \varepsilon \in \mathds{Z}_3$,   $ \cZ \neq \bm{ \lambda}  + \cC \in  \cC ^\perp_{\equiv \tau} \bmod \cC$ ,  $\cee \in \cD ^\perp \bmod \cD $ and $ \cdd  + \cD   \in \cD ^\perp/ \cD$.
\end{thm}
\begin{proof}
    (i)  For the module $V_{ L_{    \cC \times     ( \cdd + \cD)  }} [ \varepsilon ] $  we have a decomposition of $ ( V_L^\tau)^{ \otimes \ell}$-modules:
\begin{equation}\label{eq:dec:VLCDe}
   V_{L_{ \cC \times ( \bm{ \delta} + \cD)}}[ \varepsilon]   \cong
  V_{L_{\cZ \times ( \bm{ \delta} + \cD)}}[ \varepsilon] \oplus \bigoplus_{ \cZ \neq \bm{ \gamma} \in   \Ceqt} V_{L_{ \bm{ \gamma}  \times ( \bm{ \delta}  + \cD)}}.
\end{equation}
Although  the characters  $\ch V_{L_{ \bm{\gamma}  \times {\Delta} }} $ may vary as $ \Delta$ varies in $ \cD$,  we still have
\begin{equation*}\begin{aligned}
\lim_{ y \to 0^+} \frac {\ch V_{L_{  \bm{\gamma}  \times ( \bm{\delta} +  \Delta)}}}{ \ch \VLtl}
= \prod_{ i =1}^ \ell  \lim_{ y \to 0^+} \frac { \ch   V_{ L^{( \gamma^i, \cdd^i + \Delta^i )} }}{ \ch V_L^ \tau}
= \prod_{ i =1}^ \ell \qdim V_{ L^{( \gamma^i, \cdd^i + \Delta^i )}}
= \lim_{ y \to 0^+} \frac{\ch V_{L_{  \bm{\gamma}  \times \bm{\delta} }}}{\ch \VLtl},
 \end{aligned} \end{equation*}
 for all $\Delta \in \cD $.
This implies
\begin{equation}\label{eq:lim:VLgd:VLtl}\begin{aligned}
\lim_{ y \to 0^+} \frac {\ch V_{L_{  \bm{\gamma}  \times ( \bm{\delta} +  \cD)}}}{ \ch \VLtl}
= \lim_{ y \to 0^+} \frac{ \sum_{ \Delta \in \cD }\ch V_{L_{  \bm{\gamma}  \times  \Delta }}}{\ch \VLtl}
= \left|  \cD \right|\lim_{ y \to 0^+} \frac{\ch V_{L_{  \bm{\gamma}  \times  \bm{\delta} }}}{\ch \VLtl}.
 \end{aligned} \end{equation}

Similarly, we have
\begin{equation*}\begin{aligned}
\lim_{ y \to 0^+} \frac {\ch V_{L_{ \cZ \times (  \bm{\delta} + \Delta)}}[ \varepsilon]    }{ \ch \VLtl}
= \lim_{ y \to 0^+} \frac{\ch  V_{L_{ \cZ \times   \bm{\delta} }}[ \varepsilon]    }{\ch \VLtl}, \textrm{ for all }   \Delta \in \cD.
 \end{aligned} \end{equation*}
 Therefore,
 \begin{equation}\label{eq:lim:VLzg}\begin{aligned}
\lim_{ y \to 0^+} \frac {\ch V_{L_{ \cZ \times (  \bm{\delta} + \cD)}}[ \varepsilon]    }{ \ch \VLtl}
= \lim_{ y \to 0^+}  \frac { \sum_{ \Delta \in \cD }\ch V_{L_{ \cZ \times (  \bm{\delta} + \Delta)}}[ \varepsilon]    }{ \ch \VLtl}
= \left| \cD \right|\lim_{ y \to 0^+} \frac{\ch  V_{L_{ \cZ \times   \bm{\delta} }}[ \varepsilon]    }{\ch \VLtl}.
 \end{aligned} \end{equation}
  Thus by ~\eqref{eq:dec:VLCDe}, ~\eqref{eq:lim:VLgd:VLtl} and ~\eqref{eq:lim:VLzg} we know
\begin{align*}
&\lim_{ y \to 0^+} \frac {\ch \VL{ \cC \times (  \bm{\delta} + \cD)} [ \varepsilon]    }{ \ch \VLtl}
= \lim_{ y \to 0^+} \frac {\ch \VL{ \cZ \times (  \bm{\delta} + \cD)}[ \varepsilon]  + \ch \bigoplus_{ \cZ \neq \bm{ \gamma} \in   \Ceqt} \VL{ \bm{ \gamma}  \times ( \bm{ \delta}  + \cD)} }    { \ch \VLtl}\\
=& \lim_{ y \to 0^+} \frac { \left| \cD \right|\ch \VL{ \cZ \times   \bm{\delta} }[ \varepsilon]  + \left| \cD \right|\ch \bigoplus_{ \cZ \neq \bm{ \gamma} \in   \Ceqt} \VL{ \bm{ \gamma}  \times  \bm{ \delta}  } }    { \ch \VLtl}
= \lim_{ y \to 0^+} \frac { \left| \cD \right|\ch \VL{ \cC \times   \bm{\delta} } [ \varepsilon]    }{ \ch \VLtl}.
\end{align*}
Moreover,
 \begin{equation*}\begin{aligned}
&\qdim_{ V^\tau_\LCD}  V_{L_{ \cC \times (  \bm{\delta} + \cD)}}[ \varepsilon]
= \lim_{ y \to 0^+} \frac{ \ch V_{L_{ \cC \times (  \bm{\delta} + \cD)}}[ \varepsilon] }{ \ch V^\tau_\LCD}
= \lim_{ y \to 0^+} \frac{ \frac{1}{ \ch \VLtl} \ch V_{L_{ \cC \times (  \bm{\delta} + \cD)}}[ \varepsilon] }{ \frac{1}{ \ch \VLtl} \ch V^\tau_\LCD} \\
=& \lim_{ y \to 0^+} \frac{ \left| \cD \right|\ch V_{L_{ \cC \times   \bm{\delta} }}[ \varepsilon] }{ \left| \cD \right|\ch V_\LC[0] }
=  \qdim_{ V^\tau_\LC}  V_{L_{ \cC \times   \bm{\delta} }}[ \varepsilon] =1.
 \end{aligned} \end{equation*}
 \medskip
 (ii) By the similar arguments as (i), we have
 \begin{align*}
 \lim_{ y \to 0^+} \frac{ \ch \VL{     ( \bm{ \lambda} +  \cC)  \times   ( \bm{  \delta}  + \cD) }}{ \ch \VLtl }
 = \lim_{ y \to 0^+} \frac{ \ch\left( \bigoplus_{ \Delta \in \cD} \VL{     ( \bm{ \lambda} +  \cC)  \times   ( \bm{  \delta}  + \Delta ) }\right)}{ \ch \VLtl }
 = \lim_{ y \to 0^+} \frac{  \left| \cD \right| \ch   \VL{     ( \bm{ \lambda} +  \cC)  \times    \bm{  \delta}    }}{ \ch \VLtl },
\end{align*}
and hence
\begin{align*}
 \qdim_{ V^\tau_\LCD} V_{ L_{     ( \bm{ \lambda} +  \cC)  \times   ( \bm{  \delta}  + \cD) }}
 =   \qdim_{ V^\tau_\LC} V_{ L_{     ( \bm{ \lambda} +  \cC)  \times    \bm{  \delta}   }} =3.
\end{align*}
\medskip
(iii) For the irreducible $V^\tau_{\LCD} $-modules  of the type $V^{ T, \bm{\eta}}_\LCD ( \tau^i)  [ \varepsilon ] $,   we have the
decomposition of $ ( V_L^\tau)^{ \otimes \ell}$-modules:
 \begin{equation*}
   V^{ T, \bm{\eta}}_\LCD ( \tau^i)  [ \varepsilon ] \cong
    \bigoplus_{ \cgg \in \cD}
    \bigoplus_{  \bm{e} \in S_\varepsilon  }
    V_L^{ T, \eta_1 - i \gamma_1} ( \tau^i) [ e_1] \otimes \cdots \otimes V_L^{ T, \eta_\ell - i \gamma_\ell} ( \tau^i) [ e_\ell].
\end{equation*}
Fix $ \bm{ e } \in \mathds{Z}_3 ^ \ell$; the characters $\ch V_L^{ T, \eta_1 - i \gamma_1} ( \tau^i) [ e_1] \otimes \cdots \otimes
V_L^{ T, \eta_\ell - i \gamma_\ell} ( \tau^i) [ e_\ell]$ are all the same  for any  $ ( \gamma_1, \cdots, \gamma_\ell) \in \cD $ .
Thus,
\begin{align*}
  & \ch V^{ T, \bm{  \eta}}_\LCD   ( \tau^i)  [ \varepsilon ]
   = \left| \cD \right|
    \bigoplus_{  \bm{e} \in S_\varepsilon  }
    \ch V_L^{ T, \eta_1 - i \gamma_1} ( \tau^i) [ e_1] \otimes \cdots \otimes V_L^{ T, \eta_\ell - i \gamma_\ell} ( \tau^i) [ e_\ell]\\
   =& \left| \cD \right|  \ch V^{ T, \eta}_\LC ( \tau^i)  [ \varepsilon ].
\end{align*}
 As before we have
\begin{equation*}
   \qdim_{ V^\tau_\LCD}V^{ T, \bm{ \eta}}_\LCD   ( \tau^i)  [ \varepsilon ]  = \qdim_{ V^\tau_\LC}  V^{ T, \bm{ \eta} }_\LC   ( \tau^i)  [ \varepsilon ],
\end{equation*}
which is $2^\ell /|\cC|$.
\end{proof}

\paragraph{ \textbf{Global Dimension}}

   Let $V$ be a VOA with only finitely many irreducible modules, the \emph{global dimension} of V \cite{DJX12}  is defined as
   \begin{align}\label{eq:def:glob}
\glob (V ) := \sum_{ M \in \mathrm{Irr}(V )} \qdim(M)^2.
\end{align}
Assume $G$ is a finite subgroup of $ \Aut(V)$, it is conjectured in~\cite{DJX12} that
\begin{align*}
\left|  G \right|^2 \glob(V)= \glob (V^ G).
\end{align*}
We will verify this conjecture in our case, \textit{ i.e.,\/} $ V= V_\LCD$ and $ G = \langle \tau \rangle$.

Since all irreducible $ V_\LCD$-modules are simple current, we have
\begin{align*}
\glob ( V_\LCD) =  \left| \cC ^\perp/ \cC  \right| \left| \cD ^\perp/ \cD \right| \cdot 1^2.
\end{align*}

The global dimension of  $V^ \tau_\LCD$ will be  computed below. We count the number of irreducibles that have the same quantum dimensions.
 \begin{enumerate}[label=(\roman*)]
   \item $ \qdim_{ V^\tau_\LCD} V_{ L_{    \cC \times     ( \cdd + \cD)  }} [ \varepsilon ]  =  1$.  There are $ \left| \cD ^\perp/ \cD \right| \cdot 3 $ irreducible modules of this type.
   \item $ \qdim_{ V^\tau_\LCD} V_{ L_{     ( \bm{ \lambda} +  \cC)  \times   ( \bm{  \delta}  + \cD) }}=  3$ if $ \cZ \neq \bm{ \lambda}  + \cC \in  \cC ^\perp_{\equiv \tau} \bmod \cC $. There are $\left| \cD ^\perp/ \cD \right|\cdot \frac{ \left| \cC ^\perp/ \cC \right| -1}{3}$ irreducible modules of this type.
   \item $ \qdim_{ V^\tau_\LCD}    V^{ T, \bm{ \eta} }_\LC ( \tau^i)  [ \varepsilon ]  = \frac{ 2^\ell}{  \left| \cC \right|}$. There are $\left| \cD ^\perp/ \cD \right| \cdot 3 \cdot 2$
irreducible modules of this type.
\end{enumerate}
Note that $\big(   2^\ell/\left| \cC \right|  \big) ^2 =  \left| \cC ^\perp/ \cC \right|$. Therefore,
   \begin{align*}
   \glob  V^ \tau_\LCD &=   \left| \cD ^\perp/ \cD \right|
   \bigg( 3 +  \frac{ \left| \cC ^\perp/ \cC \right| -1}{3} \cdot 3^2 +  6   \left| \cC ^\perp/ \cC \right| \bigg)
   = 9  \left| \cC ^\perp/ \cC \right|\left| \cD ^\perp/ \cD \right|.
\end{align*}
Hence we  have $ \glob ( V_\LCD) \cdot 3^2 =\glob ( V^ \tau_\LCD)$. This  verified the  conjecture of Dong, Jiao and Xu in this special case.

\section{Fusion Rules}
 In this section, we  compute  the  fusion rules of $ V^ \tau_\LCD$-modules.  The next three propositions are crucial to our calculations.

\begin{prop}[{\cite[Prop.4.5]{TY06}}] \label{thm:fusionL}
Let $\varepsilon,\varepsilon_1,\varepsilon_2,j,j_1,j_2,k\in \mathds{Z}_3$ and $i=1,2$.
Then
\begin{enumerate}[label=(\roman*)]
 \item $V_{L^{(0,j_1)}}[\varepsilon_1]\times V_{L^{(0,j_2)}}[\varepsilon_2]
                 = V_{L^{(0,j_1+j_2)}}[\varepsilon_1+\varepsilon_2]$;
\item $V_{L^{(0,j_1)}}[\varepsilon]\times V_{L^{(c,j_2)}}
                    = V_{L^{(c,j_1+j_2)}}    $;
\item $V_{L^{(c,j_1)}}\times V_{L^{(c,j_2)}}
                 = \sum_{\rho=0}^{2}V_{L^{(0,j_1+j_2)}}[\rho]+2V_{L^{(c,j_1+j_2)}};  $
\item  $V_{L^{(0,j)}}[\varepsilon_1]\times \vt{L}{k}{\tau^i}{\varepsilon_2} 
                        = \vt{L}{k-ij}{\tau^i}{i\varepsilon_1+\varepsilon_2}; $
\item $V_{L^{(c,j)}}\times \vt{L}{k}{\tau^i}{\varepsilon}
                                 = \sum_{\rho=0}^{2}\vt{L}{k-ij}{\tau^i}{\rho}. $
 \end{enumerate}
\end{prop}
\begin{prop}\cite{C13a}
 \label{thm:fusion:L:twisted}
We have the following fusion rules among irreducible $V_L^ \tau$-modules of twisted type.
\begin{enumerate}[label=(\roman*)]
   \item $V_L^{ T, i}( \tau^l)[ \varepsilon] \times V_L^{ T, j}( \tau^l )[ \varepsilon ^\prime]  =   V_L^{ T, -( i +j) }( \tau^{2l} )[ -( \varepsilon+ \varepsilon ^\prime) ] + V_L^{ T, -( i +j) }( \tau^{2l} )[   2 -( \varepsilon+ \varepsilon ^\prime)] $;
   \item $V_L^{ T, i}( \tau)[ \varepsilon] \times V_L^{ T, j}( \tau^2 )[ \varepsilon^\prime  ]  = V_{L^{ (0, i + 2j)} } [ \varepsilon + 2 \varepsilon^\prime] + V_{L^{ (c, i + 2j)} }$,
\end{enumerate}
where $l\in \{1,2\}$, $ i, j,   \varepsilon, \varepsilon ^\prime \in \{ 0,1,2 \} $. 
\end{prop}

\begin{prop}[{\cite[Prop 2.9]{ADL05}}]\label{thm:inj:fusion}
Let $V$ be a vertex operator algebra and let $M^1$, $M^2$, $M^3$ be
$V$-modules  among which $M^1$ and $M^2$ are irreducible.
Suppose that $U$ is a vertex operator subalgebra
of $V$ (with the same Virasoro element) and that
$N^1$ and $N^2$ are irreducible $U$-submodules of
$M^1$ and $M^2$, respectively.
Then the restriction map from $I_V\fusion{M^1}{M^2}{M^3}$
to $I_U\fusion{N^1}{N^2}{M^3}$
is injective. In particular,
\begin{align}
\dim I_V\fusion{M^1}{M^2}{M^3}\leq\dim I_U\fusion{N^1}{N^2}{M^3}.
\end{align}
\end{prop}
In our case, we consider the following chain of subVOAs:
\begin{align*}
 V_\LCD \supset  V^\tau_\LCD \supset V^ \tau_\LC\supset \VLtl.
\end{align*}
For simplicity, we denote
\begin{align*}
  \fuCD{-}{-}{-} &=\dim I_{ V_\LCD} \fusion{-}{-}{-},&
  \fuCDt{-}{-}{-} &=\dim I_{ V^\tau_\LCD} \fusion{-}{-}{-},\\
 \fult{-}{-}{-} &=\dim I_{ ( V_L^\tau)^{ \otimes \ell} } \fusion{-}{-}{-},&
 \fut {-}{-}{-} &=\dim I_{  V_L^\tau  } \fusion{-}{-}{-}.
\end{align*}

The basic idea is to  use Proposition~\ref{thm:inj:fusion} and the quantum dimensions of $V_L^ \tau$-modules to show that many fusion
coefficients are zero. This gives some  inequalities on fusion rules. Next by using quantum dimensions,  we show that these
inequalities are actually equalities.

Let $ \cll +\cC, \cll^1+\cC, \cll^2 +\cC \in \cC ^\perp/ \cC$, $ \cdd+ \cD, \cdd^1+ \cD, \cdd^2+ \cD \in \cD ^\perp/ \cD $, $ \cee,
\cee^1, \cee^2 \in \cD ^\perp \bmod \cD$ and $ \varepsilon, \varepsilon^1, \varepsilon^2 \in \mathds{Z}_3$.
We will compute
fusion rules separately in the following cases:
\begin{enumerate}[label=(\Roman*)]
   \item Fusion rules of the form $  \VL{    \cC \times   ( \cdd + \cD )    } [ \varepsilon ]  \times M$ for any irreducible module $M$(see Prop.~\ref{thm:fusion:I});
   \item Fusion rules of the form $ \VL{     ( \cll^1 +  \cC)  \times   ( \cdd^1 + \cD )} \times \VL{     ( \cll^2 +  \cC)  \times   ( \cdd^2 + \cD )}$ (see Prop.~\ref{thm:fusion:II});
   \item Fusion rules of the form  $  \VL{     ( \cll +  \cC)}  \times V^{ T, \cee }_\LCD ( \tau^i)  [ \varepsilon ]  $ (see  Prop.~\ref{thm:fusion:III});
   \item Fusion rules of the form  $ V^{ T, \cee^1 }_\LCD ( \tau)  [ \varepsilon^1 ]  \times V^{ T, \cee^2 }_\LCD ( \tau^2)  [ \varepsilon^2 ] $ (see Prop.~\ref{thm:fusion:IV});
   \item Fusion rules of the form $ V^{ T, \cee^1 }_\LCD ( \tau^i)  [ \varepsilon^1 ]  \times V^{ T, \cee^2 }_\LCD ( \tau^i)  [ \varepsilon^2 ] $.  
   In this case, we first determine the fusion coefficients up to a permutation in  Prop.~\ref{thm:fusion:V}. Then we use modular invariance property of trace functions to get an explicit result in ~Prop.\ref{thm:fusionV:general}.
\end{enumerate}

We start with Case (I).
\begin{prop}\label{thm:fusion:I}
We have the following fusion rules.
\begin{subequations}\begin{align}
 &  \mathrm{(i)}&\VL{    \cC \times    ( \cdd^1 + \cD )   } [ \varepsilon^1 ]  \times \VL{    \cC \times     ( \cdd^2 + \cD )  } [ \varepsilon^2 ]  &= \VL{    \cC \times    ( \cdd^1 + \cdd^2 + \cD   )} [ \varepsilon^1 + \varepsilon^2 ] ;&&&\label{eq:fusion:VCe:VCe2}\\
& \mathrm{(ii)}&\VL{    \cC \times    ( \cdd^1 + \cD )   } [ \varepsilon^1 ]   \times \VL{     ( \cll +  \cC)  \times    ( \cdd^2 + \cD )  } &= \VL{     ( \cll +  \cC)  \times   ( \cdd^1 + \cdd^2 + \cD )   } ;&&& \label{eq:fusion:VCe:Vg+C}\\
  & \mathrm{(iii)}&\VL{    \cC \times    ( \cdd^1 + \cD )   } [ \varepsilon^1 ]    \times  V^{ T, \cdd^2 }_\LCD ( \tau^i)  [ \varepsilon^2 ]    &= V^{ T, \cdd^2  - i  \cdd^1 }_\LCD ( \tau^i)  [ i \varepsilon^1 +  \varepsilon^2 ],&\label{eq:fusion:VCe:VCtwist}
 \end{align} \end{subequations}
 where $ \cdd^1 + \cD, \cdd^2+ \cD \in \cD ^\perp/ \cD $, $ \cZ \neq \bm{ \lambda}  + \cC \in  \cC ^\perp_{\equiv \tau} \bmod \cC$ and $ \varepsilon^1, \varepsilon^2 \in \mathds{Z}_3$.
\end{prop}
\begin{proof}
 (i)  Observe that $ \qdim  \big( \VL{    \cC \times    ( \cdd^1 + \cD )   } [ \varepsilon^1 ]  \times \VL{    \cC \times     ( \cdd^2 + \cD )  } [ \varepsilon^2 ]    \big)  =1$; therefore the fusion product $ \VL{    \cC \times    ( \cdd^1 + \cD )   } [ \varepsilon^1 ]  \times \VL{    \cC \times     ( \cdd^2 + \cD )  } [ \varepsilon^2 ]  $  is irreducible.

Recall the fusion rules of $\VL{ \cC \times \cD }$-modules:
\begin{align*}
 1 = \fuCD{\VL{ \cC \times ( \cdd^1 + \cD) } }{ \VL{ \cC \times  ( \cdd^2 + \cD) }}{\VL{ \cC \times (   \cdd^1 + \cdd^2 + \cD) }}.
\end{align*}
By restricting to $ V^ \tau_{ L_{ \cC \times \cD } }$-modules and using Prop.~\ref{thm:inj:fusion},  we have
\begin{align*}
  1 \le \fuCDt{ \VL{    \cC \times    ( \cdd^1 + \cD )   } [ \varepsilon^1 ]  }{ \VL{    \cC \times    ( \cdd^2 + \cD )   } [ \varepsilon^2 ] }{ \VL{ \cC \times (   \cdd^1 + \cdd^2 + \cD) }}
  = \sum_{ \varepsilon =0}^2 \fuCDt { \VL{    \cC \times    ( \cdd^1 + \cD )   } [ \varepsilon^1 ]  }{ \VL{    \cC \times    ( \cdd^2 + \cD )   } [ \varepsilon^2 ] }{ \VL{ \cC \times (   \cdd^1 + \cdd^2 + \cD) }[ \varepsilon]}.
\end{align*}
Therefore, we know
\begin{align*}
  \VL{    \cC \times    ( \cdd^1 + \cD )   } [ \varepsilon^1 ]  \times \VL{    \cC \times     ( \cdd^2 + \cD )  } [ \varepsilon^2 ]  &= \VL{    \cC \times    ( \cdd^1 + \cdd^2 + \cD   )} [ \varepsilon],
\end{align*}
for some $ \varepsilon \in \mathds{Z}_3$.
For simplicity, we let $M^i := \VL{    \cC \times    ( \cdd^i + \cD )   } [ \varepsilon^i ]$ and $M := \VL{    \cC \times    ( \cdd^1 + \cdd^2 + \cD   )} [ \varepsilon]    $.

Recall the decompositions 
\begin{align*}
  \VL{    \cC \times    ( \cdd + \cD )   } [ \varepsilon ]
  = \VL{\cZ \times ( \cdd + \cD) }[ \varepsilon] \oplus
  \bigoplus_{ \cZ \neq \cgg \in \cC_{\Ceqt }} \VL{ \cgg \times ( \cdd + \cD)};
\end{align*}
%
\begin{align*}
  \VL{\cZ \times  \cdd  }[ \varepsilon]
  = \bigoplus_{ \bm{ e } \in S_ \varepsilon} V_{L^{ ( 0, \delta^i_1 )}}[  e^i_1] \otimes      \cdots \otimes V_{L^{ ( 0, \delta^i_ \ell )}} [ e^i_ \ell].
\end{align*}
Now  fix an irreducible $  ( V_L^\tau)^{ \otimes \ell}$-submodule 
\begin{align*}
  N^i :=  V_{L^{ ( 0, \delta^i_1 )}}[  e^i_1] \otimes      \cdots \otimes V_{L^{ ( 0, \delta^i_ \ell )}} [ e^i_ \ell]
  \subset \VL{\cZ \times  ( \cdd^i + \cD)}[ \varepsilon^i] \subset M^i
\end{align*}
for some $ \bm{ e }^i := ( e^i_1, \cdots, e^i_ \ell) \in S_ { \varepsilon_i}$.
Since 
   \begin{equation}
  M := \VL{    \cC \times    ( \cdd^1 + \cdd^2 + \cD   )} [ \varepsilon]   \cong
  \VL{\cZ \times ( \cdd^1 + \cdd^2  + \cD)}[ \varepsilon] \oplus \bigoplus_{ \cZ \neq  \cgg \in   \Ceqt} \VL{  \cgg  \times (  \cdd^1 + \cdd^2 + \cD)},
\end{equation}
we have the fusion coefficient
\begin{align}\label{eq:not:2}
 1 = \fuCDt{M^1}{M^2}{M} \le   \fult{N^1}{N^2}{ \VL{\cZ \times ( \cdd^1 + \cdd^2  + \cD)}[ \varepsilon] }
 + \sum_{ \cZ \neq  \cgg \in   \Ceqt}   \fult{N^1}{N^2}{ \VL{  \cgg  \times (  \cdd^1 + \cdd^2 + \cD)} }.
\end{align}
We claim that
\begin{align*}
 \fult{N^1}{N^2}{ \VL{  \cgg  \times (  \cdd^1 + \cdd^2 + \cD)} }=0 \quad \text{ for all } \cZ \neq  \cgg \in   \Ceqt.
\end{align*}

We have
\begin{align}\label{eq:not:1}
      \fult{N^1}{N^2}{ \VL{  \cgg  \times (  \cdd^1 + \cdd^2 + \cD)} }
  = \sum_{ \Delta \in \cD} \prod_{ k =1}^ \ell \fut{ V_{L^{ ( 0, \delta^1_k )}}[  e^1_k] }{ V_{L^{ ( 0, \delta^2_k )}}[ e^2_k] }{  V_{L^{  ( \gamma_k ,     \delta_k^1 + \delta_k^2 + \Delta_k)}} }.
\end{align}
Now, since $ \cgg \neq \cZ$ we have $ \gamma_h \neq 0$ for some $ 1 \le h \le \ell$ and hence
 \begin{align*}
  \fut{ V_{L^{ ( 0, \delta^1_h )}}[  e^1_h] }{ V_{L^{ ( 0, \delta^2_h )}}[ e^2_h] }{  V_{L^{  ( \gamma_h ,     \delta_h^1 + \delta_h^2 + \Delta_h)}} } =0.
\end{align*}
This proves our claim and equation ~\eqref{eq:not:2} becomes
\begin{align}\label{eq:not:I-i:1}
 1 \le  \fult{N^1}{N^2}{ \VL{\cZ \times ( \cdd^1 + \cdd^2  + \cD)}[ \varepsilon] }.
\end{align}
Now, we set $ ( e^i_1, \cdots, e^i_ \ell) = ( \varepsilon^i, 0, \cdots, 0)$ for $i =1,2$. Then we have
\begin{align*}
   1\le&  \fult{N^1}{N^2}{ \VL{\cZ \times ( \cdd^1 + \cdd^2  + \cD)}[ \varepsilon] }
   = \sum_{ \Delta \in \cD} \fult{N^1}{N^2}{ \VL{\cZ \times ( \cdd^1 + \cdd^2  + \Delta )}[ \varepsilon] } \\
   =&  \sum_{ \Delta \in \cD} \fult{ V_{L^{ ( 0, \delta^1_1 )}}[ \varepsilon^1] \otimes  V_{L^{ ( 0, \delta^1_2 )}} [ 0] \otimes  \cdots \otimes V_{L^{ ( 0, \delta^1_ \ell  )}} [ 0]}{ N^2 }
   { \sum_{ \bm{ r } \in S_ \varepsilon } V_{L^{ ( 0, \delta^1_1 + \delta^2_1 + \Delta_1 )}}[ r_1 ] \otimes  \cdots \otimes V_{L^{ ( 0, \delta^1_ \ell + \delta^2_ \ell + \Delta_ \ell )}}[ r_ \ell ]   }   \\
   = &  \sum_{ \overset{\bm{ r } \in S_ \varepsilon }{ \Delta \in \cD}} \bigg(
   \fut{ V_{L^{ ( 0, \delta^1_1 )}}[ \varepsilon^1]}{ V_{L^{ ( 0, \delta^2_1 )}}[ \varepsilon^2]}{ V_{L^{ ( 0, \delta^1_1 + \delta^2_1 + \Delta_1 )}}[ r_1 ] } \prod_{ k =2}^ \ell \fut{ V_{L^{ ( 0, \delta^1_k )}}[ 0]}{ V_{L^{ ( 0, \delta^2_k )}}[ 0 ]}{ V_{L^{ ( 0, \delta^1_k + \delta^2_k + \Delta_k )}}[ r_k ] }\bigg).
\end{align*}
By Prop.~\ref{thm:fusionL} we know if $ ( r_2, \cdots, r_ \ell) \neq (0, \cdots, 0)$ then
\begin{align*}
  \fut{ V_{L^{ ( 0, \delta^1_1 )}}[ \varepsilon^1]}{ V_{L^{ ( 0, \delta^2_1 )}}[ \varepsilon^2]}{ V_{L^{ ( 0, \delta^1_1 + \delta^2_1 + \Delta_1 )}}[ r_1 ] } \prod_{ k =2}^ \ell \fut{ V_{L^{ ( 0, \delta^1_k )}}[ 0]}{ V_{L^{ ( 0, \delta^2_k )}}[ 0 ]}{ V_{L^{ ( 0, \delta^1_k + \delta^2_k + \Delta_k )}}[ r_k ] } =0.
\end{align*}
Thus only $ \bm{ r } = (r_1, 0, \cdots, 0 ) \in S_ \varepsilon$ contributes a nonzero summand. Therefore
 \begin{align*}
 1 &\le \fult{N^1}{N^2}{ \VL{\cZ \times ( \cdd^1 + \cdd^2  + \cD)}[ \varepsilon] }\\
&= \sum_{   \Delta \in \cD} \bigg(
   \fut{ V_{L^{ ( 0, \delta^1_1 )}}[ \varepsilon^1]}{ V_{L^{ ( 0, \delta^2_1 )}}[ \varepsilon^2]}{ V_{L^{ ( 0, \delta^1_1 + \delta^2_1 + \Delta_1 )}}[  r_1 ] } \prod_{ k =2}^ \ell \fut{ V_{L^{ ( 0, \delta^1_k )}}[ 0]}{ V_{L^{ ( 0, \delta^2_k )}}[ 0 ]}{ V_{L^{ ( 0, \delta^1_k + \delta^2_k + \Delta_k )}}[ 0 ] }\bigg).
\end{align*}
Since $ \bm{ r } \in S_ \varepsilon$, we must have $ r_1 = \varepsilon = \varepsilon^1 + \varepsilon^2$. This proves (i).

 (ii) We know the fusion coefficient of $ V_\LCD$-modules:
 \begin{align*}
 1 =  \fuCD{ \VL{    \cC \times    ( \cdd^1 + \cD )   }}{ \VL{     ( \cll +  \cC)  \times    ( \cdd^2 + \cD )  }}{ \VL{     ( \cll +  \cC)  \times    ( \cdd^1 + \cdd^2 + \cD )  }}.
\end{align*}
By restricting to $ V^ \tau_\LCD$-modules, we have
\begin{align*}
  1 \le  \fuCDt{ \VL{    \cC \times    ( \cdd^1 + \cD )}[ \varepsilon^1]   }   { \VL{     ( \cll +  \cC)  \times    ( \cdd^2 + \cD )  }} { \VL{     ( \cll +  \cC)  \times    ( \cdd^1 + \cdd^2 + \cD )  }}.
 \end{align*}
 Since $ \qdim \VL{     ( \cll +  \cC)  \times    ( \cdd^1 + \cdd^2 + \cD ) }  = \qdim \big( \VL{    \cC \times    ( \cdd^1 + \cD )}[ \varepsilon^1]    \times  { \VL{     ( \cll +  \cC)  \times    ( \cdd^2 + \cD )  }}\big)$, we prove (ii).

(iii)
Since $\VL{    \cC \times    ( \cdd^1 + \cD )   } [ \varepsilon^1 ]  $ is simple current and $ \qdim \big( \VL{    \cC \times    ( \cdd^1 + \cD )   } [ \varepsilon^1 ]    \times  V^{ T, \cdd^2 }_\LCD ( \tau^i)  [ \varepsilon^2 ] \big) = \frac{ 2^ \ell}{ \left| \cC \right|} $, we know the fusion product $ \VL{    \cC \times    ( \cdd^1 + \cD )   } [ \varepsilon^1 ]    \times  V^{ T, \cdd^2 }_\LCD ( \tau^i)  [ \varepsilon^2 ]   $ is either $ V^{ T, \cdd }_\LCD ( \tau^j)  [ \varepsilon ]  $ for some $ \cdd + \cC \in \cC ^\perp/ \cC , \varepsilon \in \mathds{Z}_3$ and $j = 1,2$ or $ \VL{    \cC \times    ( \cdd + \cD )   } [ \varepsilon ] $ if $ \left| \cC \right| = 2 ^ \ell$.

Assume
\begin{align*}
  \VL{    \cC \times    ( \cdd^1 + \cD )   } [ \varepsilon^1 ]    \times  V^{ T, \cdd^2 }_\LCD ( \tau^i)  [ \varepsilon^2 ]
  = \VL{    \cC \times    ( \cdd + \cD )   } [ \varepsilon ].
\end{align*}
Then we have
\begin{align*}
   & \VL{    \cC \times    ( - \cdd^1 + \cD )   } [  - \varepsilon^1 ]  \times \VL{    \cC \times    ( \cdd^1 + \cD )   } [ \varepsilon^1 ]    \times  V^{ T, \cdd^2 }_\LCD ( \tau^i)  [ \varepsilon^2 ]
  = \VL{    \cC \times    ( - \cdd^1 + \cD )   } [ -  \varepsilon^1 ]  \times \VL{    \cC \times    ( \cdd + \cD )   } [ \varepsilon ],
\intertext{and hence  by ~\eqref{eq:fusion:VCe:VCe2} }
& V^{ T, \cdd^2 }_\LCD ( \tau^i)  [ \varepsilon^2 ]
=
   \VL{    \cC \times      \cD    } [  0 ]    \times  V^{ T, \cdd^2 }_\LCD ( \tau^i)  [ \varepsilon^2 ]
   = \VL{    \cC \times    ( \cdd - \cdd^1 + \cD )   } [ \varepsilon -  \varepsilon^1 ],
\end{align*}
a contradiction.  Therefore,
\begin{align*}
  \VL{    \cC \times    ( \cdd^1 + \cD )   } [ \varepsilon^1 ]    \times  V^{ T, \cdd^2 }_\LCD ( \tau^i)  [ \varepsilon^2 ]
  =   V^{ T, \cdd^3 }_\LCD ( \tau^j)  [ \varepsilon^3 ],
\end{align*}
for some $j = 1,2$, $ \cdd^h + \cC\in \cC ^\perp/ \cC $, $ \varepsilon^h \in \mathds{Z}_3$, for $h = 1,2,3$.

Similar to (i), we pick the following irreducible $ \VLtl$-modules
\begin{align*}
   & V_{L^{ ( 0, \delta^1_1 )}}[  e^1_1] \otimes      \cdots \otimes V_{L^{ ( 0, \delta^1_ \ell )}} [ e^1_ \ell]
  \subset \VL{\cZ \times  ( \cdd^1 + \cD)}[ \varepsilon^1];\\
  & V_L^{ T, \delta_1^2 } ( \tau^i) [  e^2_1 ] \otimes \cdots \otimes V_L^{ T, \delta^2_\ell } ( \tau^i) [ e^2_\ell] \subset V^{ T, \cdd^2 }_\LCD ( \tau^i)  [ \varepsilon^2 ];
\end{align*}
  of $M^i$ for some $ \bm{ e }^h := ( e^h_1, \cdots, e^h_ \ell) \in S_ { \varepsilon_h}$, $ h = 1,2$.

Prop.~\ref{thm:inj:fusion} suggests that
\begin{align*}
 1 =& \fuCDt {\VL{    \cC \times    ( \cdd^1 + \cD )   } [ \varepsilon^1 ] }{  V^{ T, \cdd^2 }_\LCD ( \tau^i)  [ \varepsilon^2 ] }{ V^{ T, \cdd^3 }_\LCD ( \tau^j)  [ \varepsilon^3 ]}\\
 &\le N^ \circ\fusion{ V_{L^{ ( 0, \delta^1_1 )}}[  e^1_1] \otimes      \cdots \otimes V_{L^{ ( 0, \delta^1_ \ell )}} [ e^1_ \ell] }{ V_L^{ T, \delta_1^2 } ( \tau^i) [  e^2_1 ] \otimes \cdots \otimes V_L^{ T, \delta^2_\ell } ( \tau^i) [ e^2_\ell]}{ \sum_{ \bm{ e^3 } \in S_{ \varepsilon_3 } } V_L^{ T, \delta_1^3 } ( \tau^i) [  e^3_1 ] \otimes \cdots \otimes V_L^{ T, \delta^3_\ell } ( \tau^i) [ e^3_\ell]}\\
 =&  \sum_{ \bm{ e^3 } \in S_{ \varepsilon_3 } } \prod_{ k =1}^ \ell \fut{ V_{L^{ ( 0, \delta^1_1 )}}[  e^1_k]  }{ V_L^{ T, \delta_k^2 } ( \tau^i) [  e^2_k ]  }{ V_L^{ T, \delta_k^3 } ( \tau^j) [  e^3_k ]}.
\end{align*}
If $ j \neq i$, then Prop.~\ref{thm:fusionL} gives $ 1 \le 0$, a contradiction. Therefore $ j=i$.
If  there exists $ 1 \le k \le \ell$ such that $ \delta_k^3 \neq \delta_k^2 - i \delta_k^1$ or $ e^3_k \neq i e^1_k + e^2_k$, again Prop.~\ref{thm:fusionL} gives $ 1 \le 0$, a contradiction.

Therefore, we must have $ \delta_k^3 =  \delta_k^2 - i \delta_k^1$ and  $ e^3_k  =  i e^1_k + e^2_k$ for all $k$. This gives
$ \cdd^3 = \cdd^2 - i \cdd^1$ and $ \varepsilon_3 \equiv \sum_{ k =1}^ \ell e^3_k  =  \sum_{ k =1}^ \ell i e^1_k + e^2_k \equiv  i \varepsilon^1 + \varepsilon^2 \bmod 3$. This completes the proof.
\end{proof}

Using the above proposition, we can find the contragredient dual of  irreducible modules. Recall   there are natural isomorphisms between the following fusion rules:
\begin{align*}
 N\fusion{A}{B}{C} =  N\fusion{B}{A}{C} = N\fusion{A}{C ^\prime} {B ^\prime},
\end{align*}
for every $V$-modules $A, B$ and $C$.

\begin{prop}\label{thm:dual}
   The contragredient dual of   irreducible $V^ \tau_\LCD$-modules is listed below.
   \begin{enumerate}[label=(\roman*)]
   \item $ \big(   \VL{    \cC \times  (  \cdd +    \cD) } [ \varepsilon ] \big) ^\prime = \VL{    \cC \times  (  -\cdd +    \cD) } [ -\varepsilon ]$;
   \item $ \big(  \VL{     ( \bm{ \lambda} +  \cC)  \times  ( \cdd  +    \cD) } \big) ^\prime = \VL{     (  - \bm{ \lambda} +  \cC)  \times  (  - \cdd  +    \cD) } = \VL{     (   \bm{ \lambda} +  \cC)  \times  (  - \cdd  +    \cD) }   $;
    \item $ \big(  V^{T,  \cee }_\LCD( \tau^i) [  \varepsilon ] \big) ^\prime  =   V^{T,  \cee }_\LCD( \tau^{ 2 i}) [  \varepsilon ] $.
\end{enumerate}
\end{prop}
\begin{proof}
   It is discussed in Prop.~\ref{thm:orbV:self_dual} that $ \big(   \VL{    \cC \times     \cD } [ 0 ]\big) ^\prime \cong \VL{    \cC \times     \cD } [ 0 ] $ is self-dual. We know the fusion rule:
   \begin{align*}
1 = \fuCDt{\VL{    \cC \times     \cD } [ 0 ] } { \VL{    \cC \times  (  \cdd +    \cD) } [ \varepsilon ] } {  \VL{    \cC \times  (  \cdd +    \cD) } [ \varepsilon ] }
= \fuCDt { \big(   \VL{    \cC \times  (  \cdd +    \cD) } [ \varepsilon ] \big) ^\prime} {  \VL{    \cC \times  (  \cdd +    \cD) } [ \varepsilon ] } { \big(   \VL{    \cC \times     \cD } [ 0 ]\big) ^\prime }.
\end{align*}
Since $ \qdim M = \qdim M ^\prime$ for any module $M$, by Prop.~\ref{thm:fusion:I} we may assume   $ \big( \VL{    \cC \times  (  \cdd +    \cD) } [ \varepsilon ] \big) ^\prime \cong \VL{    \cC \times  (  \cdd ^\prime +    \cD) } [ \varepsilon ^\prime ] $  for some $ \cdd ^\prime, \varepsilon ^\prime$. Now using equation  ~\eqref{eq:fusion:VCe:VCe2} we must have
\begin{align*}
\big(   \VL{    \cC \times  (  \cdd +    \cD) } [ \varepsilon ] \big) ^\prime = \VL{    \cC \times  (  -\cdd +    \cD) } [ -\varepsilon ].
\end{align*}
This proves (i). Similarly, using equation ~\eqref{eq:fusion:VCe:Vg+C} we have (ii).

(iii)
We take a different approach. We first  consider the contragredient dual of the irreducible $ V^ \tau_L$-modules of twisted type.
We know $ V^ \tau_L$ is self-dual. Let $ i, \varepsilon \in \mathds{Z}_3$, then
\begin{align*}
1= \fut{V^ \tau_L  }{ V_L^{ T, i}( \tau^l)[ \varepsilon]  }{ V_L^{ T, i}( \tau^l)[ \varepsilon]  }
=  \fut{ V_L^{ T, i}( \tau^l)[ \varepsilon]  }{ \big(   V_L^{ T, i}( \tau^l)[ \varepsilon] \big) ^\prime  } {V^ \tau_L  }.
\end{align*}
 By quantum dimensions and fusion rules of $ V^ \tau_L$-modules, we must have
\begin{align*}
 \big(  V_L^{ T, i}( \tau^l)[ \varepsilon]  \big) ^\prime = V_L^{ T, i}( \tau^{2l} )[ \varepsilon].
\end{align*}

Now, consider the decomposition of  $ \VLtl$-modules:
\begin{align*}
V^{ T, \bm{\eta}}_\LCD ( \tau^i)  [ \varepsilon ] &\cong
    \bigoplus_{  \cdd  \in \cD}
    \bigoplus_{e_1 +  \cdots +  e_\ell \equiv \varepsilon  \bmod 3 }
    V_L^{ T, \eta_1 - i \delta_1} ( \tau^i) [ e_1] \otimes \cdots \otimes V_L^{ T, \eta_\ell - i \delta_\ell} ( \tau^i) [ e_\ell].
\end{align*}

Taking contragredient dual as $ \VLtl$-modules, we have
\begin{align*}
\big(   V^{ T, \bm{\eta}}_\LCD ( \tau^i)  [ \varepsilon ] \big) ^\prime
&\cong
    \bigoplus_{  \cdd  \in \cD}
    \bigoplus_{e_1 +  \cdots +  e_\ell \equiv \varepsilon  \bmod 3 }
    V_L^{ T, \eta_1 - i \delta_1} ( \tau^{2i}) [ e_1] \otimes \cdots \otimes V_L^{ T, \eta_\ell - i \delta_\ell} ( \tau^{ 2i}) [ e_\ell].
\end{align*}

Since $ V^{ T, \bm{\eta}}_\LCD ( \tau^{ 2i})  [ \varepsilon ] $ is the only irreducible $ V^ \tau_\LCD$-module admitting the above decomposition of $ \VLtl$-modules,  we must have
\begin{align*}
\big(   V^{ T, \bm{\eta}}_\LCD ( \tau^i)  [ \varepsilon ] \big) ^\prime  \cong V^{ T, \bm{\eta}}_\LCD ( \tau^{ 2i})  [ \varepsilon ] .
\end{align*}

%
%

\end{proof}

\noindent \textbf{Case (II):}  $ \VL{     ( \cll^1 +  \cC)  \times   ( \cdd^1 + \cD )} \times \VL{     ( \cll^2 +  \cC)  \times   ( \cdd^2 + \cD )}$

\begin{prop}\label{thm:fusion:II}
We have the fusion rules
\begin{align*}
  \VL{     ( \cll^1 +  \cC)  \times   ( \cdd^1 + \cD )} \times \VL{     ( \cll^2 +  \cC)  \times   ( \cdd^2 + \cD )}
  = \bigoplus_{h = 0}^2  \VL{     ( \cll^1 + \omega^h \cll^2 +  \cC)  \times   ( \cdd^1 + \cdd^2+ \cD )}.
\end{align*}
\end{prop}
\begin{proof}
 Fix $ 0 \le h \le 2 $ we have  the fusion rules of $V_\LCD$-modules:
 \begin{align*}
 1 = \fuCD{ \VL{     ( \cll^1 +  \cC)  \times   ( \cdd^1 + \cD )}  }{  \VL{     ( \omega^h \cll^2 +  \cC)  \times   ( \cdd^2 + \cD )}  }{ \VL{     ( \cll^1 + \omega^h \cll^2 +  \cC)  \times   ( \cdd^1 + \cdd^2+ \cD )}  }
 \le \fuCDt{ \VL{     ( \cll^1 +  \cC)  \times   ( \cdd^1 + \cD )}  }{  \VL{     ( \omega^h \cll^2 +  \cC)  \times   ( \cdd^2 + \cD )}  }{ \VL{     ( \cll^1 + \omega^h \cll^2 +  \cC)  \times   ( \cdd^1 + \cdd^2+ \cD )}  }.
\end{align*}
Since $  \omega^h  \cll^2 +  \cC$, $ 0 \le h\le 2$, are identical in $\cC ^\perp_{\equiv \tau} \bmod \cC $, there is an isomorphism of $ V^ \tau_\LCD$-modules
\begin{align*}
 \VL{     (   \cll^2 +  \cC)  \times   ( \cdd^2 + \cD )}
 \cong \VL{     ( \omega \cll^2 +  \cC)  \times   ( \cdd^2 + \cD )}
 \cong \VL{     ( \omega^2 \cll^2 +  \cC)  \times   ( \cdd^2 + \cD )}.
\end{align*}
Therefore, we  can write
\begin{align*}
 1 \le  \fuCDt{ \VL{     ( \cll^1 +  \cC)  \times   ( \cdd^1 + \cD )}  }{  \VL{     ( \cll^2 +  \cC)  \times   ( \cdd^2 + \cD )}  }{ \VL{     ( \cll^1 + \omega^h \cll^2 +  \cC)  \times   ( \cdd^1 + \cdd^2+ \cD )}  } ,
\end{align*}
for all $ h$.
Since $ \cll^1 + \omega^h \cll^2 +  \cC $, $ 0 \le h \le 2$, are distinct in $\cC ^\perp_{\equiv \tau} \bmod \cC$, by counting quantum dimensions, we prove
\begin{align*}
  \VL{     ( \cll^1 +  \cC)  \times   ( \cdd^1 + \cD )} \times \VL{     ( \cll^2 +  \cC)  \times   ( \cdd^2 + \cD )}
  = \bigoplus_{h = 0}^2  \VL{     ( \cll^1 + \omega^h \cll^2 +  \cC)  \times   ( \cdd^1 + \cdd^2+ \cD )}.
\end{align*}
This completes the proof.
\end{proof}
\begin{rmk}
 Note that if $ \cll^1 + \omega^h \cll^2 = 0 $ for some $h$, then the module $ \VL{     ( \cll^1 + \omega^h \cll^2 +  \cC)  \times   ( \cdd^1 + \cdd^2+ \cD )}$ is not irreducible and admits a decomposition of irreducible modules of $V^ \tau_\LCD$-modules:
 \begin{align*}
 \VL{     ( \cll^1 + \omega^h \cll^2 +  \cC)  \times   ( \cdd^1 + \cdd^2+ \cD )}
 = \sum_{ \varepsilon =0}^2 \VL{     ( \cll^1 + \omega^h \cll^2 +  \cC)  \times   ( \cdd^1 + \cdd^2+ \cD )}[ \varepsilon].
\end{align*}
\end{rmk}

\noindent\textbf{Case (III): } $\VL{ (\cgg + \cC) \times   ( \cdd^1  + \cD) } \times V ^{ T, \cdd^2}_\LCD ( \tau^i)  [  \varepsilon  ]$

\begin{prop}\label{thm:fusion:III}
  We have
  \begin{subequations}\begin{align}
   &\mathrm{(i)}&&  \VL{ (\cgg + \cC) \times   \cD } \times V ^{ T, \cZ}_\LCD ( \tau^i)  [  0  ]
  =  \bigoplus_{  \rho = 0}^2 V ^{ T, \cZ}_\LCD ( \tau^i)  [  \rho  ]; \\
    &\mathrm{(ii)}&&  \VL{ (\cgg + \cC) \times   ( \cdd^1  + \cD) } \times V ^{ T, \cdd^2}_\LCD ( \tau^i)  [  \varepsilon  ]
  = \bigoplus_{  \rho = 0}^2 V ^{ T,  \cdd^2   }_\LCD ( \tau^i)[    \rho],
\end{align}\end{subequations}
           where $ \cdd^1 + \cD, \cdd^2+ \cD \in \cD ^\perp/ \cD $, $ \cZ \neq \bm{ \lambda}  + \cC \in  \cC ^\perp_{\equiv \tau} \bmod \cC$.
 \end{prop}
 \begin{proof}
   (i) Similar to Prop.~\ref{thm:fusion:I}(iii), we quickly have
   \begin{align*}
 0 = \fuCDt{ \VL{ (\cgg + \cC) \times   \cD }}{ V ^{ T, \cZ}_\LCD ( \tau^i)  [  0  ]  }{ V ^{ T, \cdd}_\LCD ( \tau^{ j })},
\end{align*}
when (1) $ i =j $ and  $ \cdd \neq \cZ$ or (2) $ i \neq j$.
Also, by Prop.~\ref{thm:dual} , Prop.~\ref{thm:fusion:I} and Prop.~\ref{thm:fusion:II}
we have
\begin{align*}
& \fuCDt{ \VL{ (\cgg + \cC) \times   \cD }}{ V ^{ T, \cZ}_\LCD ( \tau^i)  [  0  ]  }  { \VL{      \cC  \times   ( \cdd + \cD )} [ \varepsilon] }
= \fuCDt{ \VL{ (\cgg + \cC) \times   \cD }}  { \VL{      \cC  \times   ( -\cdd + \cD )} [ -\varepsilon] } { V ^{ T, \cZ}_\LCD ( \tau^{ 2 i })  [  0  ]  }  =0,\\
& \fuCDt{ \VL{ (\cgg + \cC) \times   \cD }}{ V ^{ T, \cZ}_\LCD ( \tau^i)  [  0  ]  } { \VL{     ( \cll +  \cC)  \times   ( \cdd + \cD )} }
= \fuCDt{ \VL{ (\cgg + \cC) \times   \cD }} { \VL{     ( -\cll +  \cC)  \times   ( - \cdd + \cD )} }{ V ^{ T, \cZ}_\LCD ( \tau^{ 2 i})  [  0  ]  } =0.
\end{align*}
Therefore
\begin{align} \label{eq:fusion:III:1}
  \VL{ (\cgg + \cC) \times   \cD } \times V ^{ T, \cZ}_\LCD ( \tau^i)  [  0  ]
  =& \bigoplus_{  \rho = 0}^2  n_ \rho V ^{ T, \cZ}_\LCD ( \tau^i)  [  \rho  ],
\end{align}
for some $ n_ \rho \in \mathds{N}$. Multiply the equation ~\eqref{eq:fusion:III:1}  by $ \VL{  \cC \times   \cD}[ h] $, $ h =1,2$,  we have
\begin{align*}
   \big(  \VL{  \cC \times   \cD}[ h]  \times   \VL{ (\cgg + \cC) \times   \cD }   \big) \times V ^{ T, \cZ}_\LCD ( \tau^i)  [  0  ]
   =   \VL{  \cC \times   \cD}[ h]  \times  \bigoplus_{  \rho = 0}^2  n_ \rho V ^{ T, \cZ}_\LCD ( \tau^i)  [  \rho  ].
\end{align*}
By Prop.~\ref{thm:fusion:I}, the left hand side is equal to
\[
\VL{ (\cgg + \cC) \times   \cD }   \times V ^{ T, \cZ}_\LCD ( \tau^i)  [  0  ]
=\bigoplus_{  \rho = 0}^2  n_ \rho V ^{ T, \cZ}_\LCD ( \tau^i)  [  \rho  ]
\]
while the right hand side is $\bigoplus_{  \rho = 0}^2  n_ \rho V ^{ T, \cZ}_\LCD ( \tau^i)  [  \rho  + h]$; thus,
we have
\begin{align*}
 \bigoplus_{  \rho = 0}^2  n_ \rho V ^{ T, \cZ}_\LCD ( \tau^i)  [  \rho  ] =  \bigoplus_{  \rho = 0}^2  n_ \rho V ^{ T, \cZ}_\LCD ( \tau^i)  [  \rho  + h],
\end{align*}
for all $ 0 \le h \le 2$. This gives $ n_0 = n_1 =n_2$. Finally, by comparing the quantum dimensions of both sides of ~\eqref{eq:fusion:III:1}, we have $ 3(2^\ell/|\cC|) = ( n_0 + n_1 + n_2)(2^\ell/|\cC|)$ and hence $ n_0 = n_1 =n_2  =1 $. This proves (i).

(ii) By Prop. ~\ref{thm:fusion:I} we have
\begin{align*}
  V ^{ T, \cdd^2}_{ L_{ \cC \times \cD }} ( \tau^i)  [  \varepsilon  ]  = \VL{    \cC \times  (  ( -1)^i \cdd^2 + \cD) } [  (-1)^{ i+1} \varepsilon ]   \times V ^{ T, \cZ}_{L_{ \cC \times \cD }} ( \tau^i)  [  0  ].
\end{align*}
 Therefore,
\begin{align*}
   & \VL{ (\cgg + \cC) \times   ( \cdd^1  + \cD) } \times V ^{ T, \cdd^2}_\LCD ( \tau^i)  [  \varepsilon  ]
   =  \VL{ (\cgg + \cC) \times   ( \cdd^1  + \cD) } \times \big(   \VL{    \cC \times  (  ( -1)^i \cdd^2 + \cD) } [  (-1)^{ i+1} \varepsilon ]  \times  V ^{ T, \cZ}_\LCD ( \tau^i)  [  0  ]  \big)\\
   =&  \VL{    \cC \times  (  ( -1)^i \cdd^2 + \cD) } [  (-1)^{ i+1} \varepsilon ]  \times
   \big( \VL{ (\cgg + \cC) \times   ( \cdd^1  + \cD) } \times V ^{ T, \cZ}_{L_{ \cC \times \cD }} ( \tau^i)  [  0  ]   \big)\\
   =&  \VL{    \cC \times  (  ( -1)^i \cdd^2 + \cD) } [  (-1)^{ i+1} \varepsilon ]  \times   \bigoplus_{  \rho = 0}^2 V ^{ T, \cZ}_\LCD ( \tau^i)  [  \rho  ]   \\
   =& \bigoplus_{  \rho = 0}^2 V ^{ T, - i  ( -1)^i \cdd^2   }_\LCD ( \tau^i)[ i  (-1)^{ i+1} \varepsilon + \rho]
   = \bigoplus_{  \rho = 0}^2 V ^{ T,  \cdd^2   }_\LCD ( \tau^i)[    \rho - \varepsilon ]
   = \bigoplus_{  \rho = 0}^2 V ^{ T,  \cdd^2   }_\LCD ( \tau^i)[    \rho].
\end{align*}
This completes the proof.
 \end{proof}

\noindent\textbf{Case (IV):} $ V^{ T, \cee^1 }_\LCD ( \tau)  [ \varepsilon^1 ]  \times V^{ T, \cee^2 }_\LCD ( \tau^2)  [ \varepsilon^2 ] $

  \begin{prop}\label{thm:fusion:IV}
 We have the fusion rules:
 \begin{subequations}\begin{align}
    & \mathrm{(i)}& V ^{ T, \cZ}_\LCD ( \tau)  [  0  ]    \times V ^{ T, \cZ}_\LCD ( \tau^2)  [  0 ]
=& V_{ L_{    \cC \times   \cD }} [  0 ]
\oplus \bigoplus_{ \cZ \neq \cgg \in \cC ^\perp_{ \equiv_ \tau} \bmod \cC} \VL{ (  \cgg + \cC) \times   \cD }\\
 & \mathrm{(ii)}& V ^{ T, \cee^1}_\LCD  ( \tau)  [  \varepsilon^1  ]    \times V ^{ T, \cee^2}_\LCD  ( \tau^2)  [  \varepsilon^2 ]
=& \VL{    \cC \times    ( \cee^2 -  \cee^1 + \cD )  } [    \varepsilon^1 - \varepsilon^2  ]\notag \\
&&&\oplus \bigoplus_{ \cZ \neq \cgg \in \cC ^\perp_{ \equiv_ \tau} \bmod \cC} \VL{ (   \cgg + \cC) \times   (  \cee^2 -  \cee^1 +\cD)  } .
\end{align}\end{subequations}
In particular, if $ \cC$ is self-dual, we have
\begin{align*}
V ^{ T, \cee^1}_\LCD  ( \tau)  [  \varepsilon^1  ]    \times V ^{ T, \cee^2}_\LCD  ( \tau^2)  [  \varepsilon^2 ]
=& \VL{    \cC \times    ( \cee^2 -  \cee^1 + \cD )  } [    \varepsilon^1 - \varepsilon^2  ].
\end{align*}
\end{prop}
\begin{proof}
   (i) By Prop.~\ref{thm:dual} we have
   \begin{align*}
\fuCDt{ V ^{ T, \cZ}_\LCD ( \tau)  [  0  ]  } {  V ^{ T, \cZ}_\LCD ( \tau^2)  [  0 ]} {  V_{ L_{    \cC \times   \cD }} [   \varepsilon ]}
= \fuCDt{ V ^{ T, \cZ}_\LCD ( \tau)  [  0  ]  }{  V_{ L_{    \cC \times   \cD }} [   2\varepsilon ]} {  V ^{ T, \cZ}_\LCD ( \tau)  [  0 ]}.
\end{align*}
By Prop.~\ref{thm:fusion:I} we have
\begin{align*}
\fuCDt{ V ^{ T, \cZ}_\LCD ( \tau)  [  0  ]  } {  V ^{ T, \cZ}_\LCD ( \tau^2)  [  0 ]} {  V_{ L_{    \cC \times   \cD }} [   \varepsilon ]}
= \begin{cases}
     1 & \textrm{ if } \varepsilon =0;\\
    0 & \textrm{ if } \varepsilon =1 ,2.
  \end{cases}
\end{align*}

Similarly, by Prop.~\ref{thm:dual} and Prop.~\ref{thm:fusion:III} we have
\begin{align*}
\fuCDt{ V ^{ T, \cZ}_\LCD ( \tau)[0]} {  V ^{ T, \cZ}_\LCD ( \tau^2)  [  0 ]} { \VL{ (  \cgg + \cC) \times   \cD } }
 =\fuCDt{ V ^{ T, \cZ}_\LCD ( \tau)[0]}{ \VL{ (  -\cgg + \cC) \times   \cD } }  {  V ^{ T, \cZ}_\LCD ( \tau)  [  0 ]}
= 1.
\end{align*}
Therefore,
\begin{align*}
 V ^{ T, \cZ}_\LCD ( \tau)  [  0  ]    \times V ^{ T, \cZ}_\LCD ( \tau^2)  [  0 ] \ge
V_{ L_{    \cC \times   \cD }} [  0 ]
\oplus \bigoplus_{ \cZ \neq \cgg \in \cC ^\perp_{ \equiv_ \tau} \bmod \cC} \VL{ (  \cgg + \cC) \times   \cD }.
\end{align*}
Recall that 
\begin{align*}
&\qdim \big( V ^{ T, \cZ}_\LCD ( \tau)  [  0  ]    \times V ^{ T, \cZ}_\LCD ( \tau^2)  [  0 ] \big)
= \big(  \frac{ 2^ \ell}{ \left| \cC \right|} \big)^2 = \left| \cC ^\perp/ \cC \right|  ;\\
 & \qdim  V_{ L_{    \cC \times   \cD }} [  0 ] =1.
\end{align*}
Moreover,
\begin{align*}
\qdim \bigg( \bigoplus_{ \cZ \neq \cgg \in \cC ^\perp_{ \equiv_ \tau} \bmod \cC} \VL{ (  \cgg + \cC) \times   \cD } \bigg)
= \# \{ \cZ \neq \cgg \in \cC ^\perp_{ \equiv_ \tau} \bmod \cC\}   \cdot 3.
\end{align*}
Since
\begin{align*}
\# \{ \cZ \neq \cgg \in \cC ^\perp_{ \equiv_ \tau} \bmod \cC\}   = \frac{   1}{3 }   \big( \left| \cC ^\perp/ \cC \right|- 1 \big)  ,
\end{align*}
we know
\begin{align*}
& \qdim \bigg( V_{ L_{    \cC \times   \cD }} [  0 ]
\oplus \bigoplus_{ \cZ \neq \cgg \in \cC ^\perp_{ \equiv_ \tau} \bmod \cC} \VL{ (  \cgg + \cC) \times   \cD } \bigg)
= \left| \cC ^\perp/ \cC \right|\\
=& \qdim \big( V ^{ T, \cZ}_\LCD ( \tau)  [  0  ]    \times V ^{ T, \cZ}_\LCD ( \tau^2)  [  0 ] \big).
\end{align*}
This proves (i).

(ii) By Prop.~\ref{thm:fusion:I} we have
\begin{align*}
  V ^{ T, \cee^i}_{L_{ \cC \times ( \cgg^i + \cD) }} ( \tau^i)  [  \varepsilon^i  ]  = \VL{    \cC \times  ( ( -1)^i \cee^i + \cD)  } [  (-1)^{ i+1} \varepsilon^i ]   \times V ^{ T, \cZ}_{L_{ \cC \times \cD }} ( \tau^i)  [  0  ].
\end{align*}
Therefore,
\begin{align*}
 &V ^{ T, \cee^1}_{L_{ \cC \times ( \cgg^1 + \cD) }} ( \tau)  [  \varepsilon^1  ]    \times V ^{ T, \cee^2}_{L_{ \cC \times ( \cgg^2 + \cD)  }} ( \tau^2)  [  \varepsilon^2 ]\\
 =&  \VL{    \cC \times    ( -  \cee^1 + \cD) } [    \varepsilon^1  ]   \times \VL{    \cC \times     ( \cee^2 + \cD) } [   -  \varepsilon^2  ]
 \times V ^{ T, \cZ}_{L_{ \cC \times \cZ }} ( \tau)  [  0  ] \times V ^{ T, \cZ}_{L_{ \cC \times \cZ }} ( \tau^2)  [  0  ] \\
 =&  \VL{    \cC \times    ( \cee^2 -  \cee^1 + \cD)  } [    \varepsilon^1 - \varepsilon^2  ]  \times \bigg(  V_{ L_{    \cC \times   \cZ }} [  0 ]
\oplus \bigoplus_{ \cZ \neq \cgg \in \cC ^\perp_{ \equiv_ \tau} \bmod \cC} \VL{ (  \cgg + \cC) \times   \cD } \bigg) \\
=& \VL{    \cC \times    ( \cee^2 -  \cee^1 + \cD)  } [    \varepsilon^1 - \varepsilon^2  ] \oplus \bigoplus_{ \cZ \neq \cgg \in \cC ^\perp_{ \equiv_ \tau} \bmod \cC} \VL{ ( \cee^2 -  \cee^1 + \cgg + \cC) \times   \cD } .
\end{align*}
This proves (ii).
\end{proof}

\noindent\textbf{Case (V):} $ V ^{ T, \cdd^1}_\LCD ( \tau^i)  [ \varepsilon^1  ]    \times V ^{ T, \cdd^2}_\LCD ( \tau^i)  [  \varepsilon^2 ]$

We first consider the case $ \cdd^1= \cdd^2= \cZ$ and $ \varepsilon^1 = \varepsilon^2 =0$.
By the similar analysis as in the previous few cases, we can show quickly that many fusion coefficients are zero.
Assume
\begin{equation}
 \label{eq:fusion:TT:try1}
 \begin{aligned}
  &V ^{ T, \cZ}_\LCD ( \tau^i)  [  0  ] \times V ^{ T, \cZ}_\LCD ( \tau^i)  [  0 ]\\
 =&   x V ^{ T, \cZ}_\LCD ( \tau^{ 2i})  [  0 ]  +  y  V ^{ T, \cZ}_\LCD ( \tau^{ 2i})  [  1]
 + z   V ^{ T, \cZ}_\LCD ( \tau^{ 2i})  [  2],
\end{aligned}
 \end{equation}
for some $ x, y, z \in \mathds{Z}_{ \ge 0}$.

For simplicity, we denote
\begin{align*}
  T[j] :=&V ^{ T, \cZ}_\LC ( \tau)  [  j  ];& \check{T}[j] :=&V ^{ T, \cZ}_\LC ( \tau^2)  [  j  ];&  S[j]:=&V_\LCD[j].
\end{align*}
Equation ~\eqref{eq:fusion:TT:try1} then becomes
\begin{align*}
 T[0] \times T[0] = x \check{T}[0] + y \check{T}[1] +z \check{T}[2]
 = \check{T}[0] \times \big(  x S[0] + y S[2] + z S[1] \big).
\end{align*}

We multiply this equation by $\check{T}[0] $ and get
\begin{align}\label{eq:fu:V:TT:1}
  \check{T}[0]  \times  T[0] \times T[0]
  = \check{T}[0]  \times \check{T}[0] \times \big(  x S[0] + y S[2] + z S[1] \big).
\end{align}
On the left hand side of ~\eqref{eq:fu:V:TT:1}, by Prop.~\ref{thm:fusion:I} and Prop.~\ref{thm:fusion:III}, we have
\begin{equation}\label{eq:fu:V:1}\begin{aligned}
    & \check{T}[0] \times  T[0] \times T[0]  =  \big(  \check{T}[0] \times  T[0]  \big) \times  T[0]
 = \bigg(  S[0] \oplus \bigoplus_{ \cZ \neq \cgg \in \cC ^\perp_{ \equiv_ \tau} \mod \cC}   \VL{ (  \cgg + \cC) \times   \cD }   \bigg) \times T[0]\\
 =& T[0] + Q \big(  T[0] + T[1] + T[2] \big),
\end{aligned}\end{equation}
where $Q = \# \{ \cZ \neq \cgg \in \cC ^\perp_{ \equiv_ \tau} \mod \cC \}= \frac{ 4^{ \ell - 2d} -1} {3} $ and  $ d = \dim \cC$.


 By symmetry between automorphisms $ \tau$ and $ \tau^2$, we can rewrite ~\eqref{eq:fusion:TT:try1} as
\begin{align*}
  \check{T}[0] \times \check{T}[0] = x  T[0] + y T[1] +z T[2].
\end{align*}
Therefore, the right hand side of ~\eqref{eq:fu:V:TT:1} becomes
\begin{equation}\begin{aligned}\label{eq:fu:V:2}
    &(\check{T}[0]  \times  \check{T}[0] ) \times \big(  x S[0] + y S[2] + z S[1] \big)\\
  =&  \big(  x  T[0] + y T[1] +z T[2]  \big) \times \big(  x S[0] + y S[2] + z S[1] \big)\\
  =& \big(  x^2 + y^2  + z ^2 \big) T[0] +
 \big( xy + yz + zx  \big) T[1] +  \big( xy + yz + zx  \big) T[2].
\end{aligned}\end{equation}
Comparing ~\eqref{eq:fu:V:1} and ~\eqref{eq:fu:V:2} we have
\begin{subequations}\label{eq:sys_eq:twisted}
\begin{align}
  & x + y + z = \qdim  T[0]    =2^{ \ell - 2d};\label{eq:x+y+z}\\
    &   xy + yz + zx  = \frac{ 4^{ \ell - 2d} -1} {3};\\
    &  x^2 + y^2  + z ^2 =  xy + yz + zx +1 \label{eq:sys_eq:twisted:c}
\end{align}
\end{subequations}
 From ~\eqref{eq:sys_eq:twisted:c} we know
\begin{align*}
  ( x -y)^2 + ( y -z)^2 + ( z-x)^2 =2.
\end{align*}
Assuming $ x \ge y \ge z \ge 0$, then we have  $ x - z =1$ and thus $ z+1 \ge y \ge z$. Therefore,
\begin{align*}
 x &= y= \frac{ 2^{ \ell - 2d} +1 }{3},&   z =&  \frac{ 2^{ \ell - 2d}  - 2 }{3},& \textrm{ if }   2^{ \ell - 2d}  \equiv 2 \mod 3;\\
 x &=  \frac{ 2^{ \ell - 2d} + 2 }{3},&  y=z =&  \frac{ 2^{ \ell - 2d}  - 1 }{3},& \textrm{ if }    2^{ \ell - 2d}  \equiv 1 \mod 3,
\end{align*}
where $ d =\dim \cC$.
Note that $ 2^{ \ell - 2d}  \equiv 2^ \ell \mod 3 $.
As a summary, we have the proposition.

\begin{prop}\label{thm:fusion:V}
We have fusion rules:
\begin{equation}\begin{aligned}
    &V ^{ T, \cZ}_\LCD ( \tau^i)  [  0  ] \times V ^{ T, \cZ}_\LCD ( \tau^i)  [  0 ]\\
 =&   x ^\prime V ^{ T, \cZ}_\LCD ( \tau^{ 2i})  [  0 ]  +  y ^\prime  V ^{ T, \cZ}_\LCD ( \tau^{ 2i})  [  1]
 + z ^\prime   V ^{ T, \cZ}_\LCD ( \tau^{ 2i})  [  2],
\end{aligned}\end{equation}
for some $ x ^\prime, y ^\prime, z ^\prime \in \mathds{N}$.

If $ ( x\ge y \ge z) $ is  the decreasing reordering of $ ( x ^\prime, y^\prime, z ^\prime)$, then we have
\begin{align*}
 x &= y= \frac{ 2^{ \ell - 2d} +1 }{3},&   z =&  \frac{ 2^{ \ell - 2d}  - 2 }{3},& \textrm{ if  \quad }  \ell  \textrm{ is odd }; &&&\\
 x &=  \frac{ 2^{ \ell - 2d} + 2 }{3},&  y=z =&  \frac{ 2^{ \ell - 2d}  - 1 }{3},& \textrm{ if \quad}    \ell  \textrm{ is even },
\end{align*}
where $ d =\dim \cC$.  Note that $ 2^{ \ell - 2d} = \sqrt{ \left| \cC ^\perp/\cC \right|} $.
\end{prop}

\paragraph{ $S$-\textbf{matrix and Verlinde formula}}
Next we compute the exact fusion rules using  Verlinde formula and $S$-matrix. First we review Dong-Li-Mason's theory on trace functions\cite{DLM00}.

Let $ V$ be a rational VOA and $ g, h \in \Aut V$ be commuting automorphisms of finite orders. Let $M$ be a $g$-twisted $h$-stable $V$-module. There exists a linear isomorphism $ \varphi(h)$ of $M$  such that
\begin{align*}
\varphi(h) Y_M ( u, z) &= Y_M( h u, z ) \varphi(h).
\end{align*}
For a homogeneous $ v \in V$  with $L(1) v =0$ we define  the trace function
\begin{equation*}
T_M(v, g,h; z ) :=\tr_{_M}\varphi(h) o( v) q^{L(0)-c/24}= q^{\lambda-c/24}\sum_{n\in\frac{1}{ \mid g \mid  } \mathds{Z}_+}\tr_{_{M_{ \lambda +n}}} o( v) \varphi(h)q^{n},
\end{equation*}
where  $o( v)$ is the degree zero operator of $v$, $ \lambda $ is the conformal weight of $M$, $ c$ is the central charge of $V$ and $ q = e^{( 2 \pi \sqrt{-1} z )}$.

\begin{prop}\cite{DLM00}
Let $ C_1 (g, h)$ be the $ \mathds{C}$-vector space
\begin{align*}
C_1 (g, h) := \mathrm{Span}_ \mathds{C} \{  T_M(v,g,h; z ) \mid M \mbox{ is a $g$-twisted $h$-stable $V$-module} \}.
\end{align*}
   Then
   \begin{enumerate}[label=(\roman*)]
   \item $ C_1(g, h)$ has a basis:
   \begin{align*}
 \{  T_M(v,g,h; z ) \mid M \mbox{ is an irreducible $g$-twisted $h$-stable $V$-module} \}.
\end{align*}
\item Modular invariance: Let $ T_M(v,g,h; z ) \in C_1( g, h)$ and $ \Gamma =  \begin{psmallmatrix} a & b \\ c & d \end{psmallmatrix}\in SL_2( \mathds{Z})$. Then we have $ T_M(v,g,h;  \Gamma \circ z ) \in C_1( g, h) \circ \Gamma$ in the sense that
\begin{align*}
T_M(v,  g,h ; \frac{ a z + b }{ c z + d}) \in C_1 ( g^a h^c, g^b h^d).
\end{align*}
In fact, if $M $ is a  $g$-twisted $h$-stable $V$-module, then
\begin{align*}
  T_M(v,  g,h ; \frac{ a z + b }{ c z + d}) = \sum S_N^{(g,h )} T_{N} (v, g, h; z ),
\end{align*}
where $N$ runs over irreducible $ g^a h^c$-twisted $ g^b h^d$-stable $V$-module, and the coefficients $S_N^{(g,h )}$ are independent of $v$.
\end{enumerate}
\end{prop}

In particular, when $ g = h = \mathrm{ id }$, $ \Gamma = \begin{psmallmatrix} 1 & 0 \\ 0 & -1 \end{psmallmatrix}$ and $ V=M^0, \cdots, M^d$ are all inequivalent irreducible $V$-modules, we have  \begin{equation}\label{eq:S-def}
T_{M^i}(v, \mathrm{ id }, \mathrm{ id } ;-\frac{1}{z })= \sum_{j=0}^d S_{i,j} T_{M^j}(z).
\end{equation}

For simplicity, we denote
\begin{align*}
{M} (g, h; z) =&  Z_{M} (g, h; z) := T_M(\idd, g,h; z ),
\shortintertext{and}
M(z) := & Z_M(  \mathrm{ id }, \mathrm{ id };z) = \ch M(z).
\end{align*}

\begin{defi}
The matrix $S=(S_{i,j})$ defined in equation \eqref{eq:S-def} is called the $S$-matrix.
\end{defi}

\begin{thm}\cite{H08}\label{Verlinde} Let $V$ be a rational and $C_2$-cofinite simple VOA of CFT type and assume $V\cong V'.$ Let $S=(S_{i,j})_{i,j=0}^d$ be the $S$-matrix as defined in \eqref{eq:S-def}. Then
\begin{enumerate}[label=(\roman*)]
   \item $(S^{-1})_{i,j}=S_{i,j'}=S_{i',j},$ and $S_{i',j'}=S_{i,j}; $
   \item $S$ is symmetric and $S^2=(\delta_{i,j'});$
   \item $N_{i,j}^k=\sum_{s=0}^d\frac{S_{j,s}S_{i,s}S^{-1}_{s,k}}{S_{0,s}};$
   \item The $S$-matrix diagonalizes the fusion matrix $N(i)=(N_{i,j}^k)_{j,k=0}^d$ with diagonal entries $\frac{S_{i,s}}{S_{0,s}},$ for  $i, s=0,\cdots,d.$ More explicitly, $S^{-1}N(i)S= \mathrm{diag} (\frac{S_{i,s}}{S_{0,s}})_{s=0}^d.$ In particular, $S_{0,s}\neq 0$ for $s=0,\cdots,d.$
\end{enumerate}
\end{thm}

\begin{prop}\cite{DJX12}
    Let $V$ be a simple, rational and $C_2$-cofinite VOA of CFT type. Let $M^0,\ M^1, \cdots ,\ M^d$ be as before with the corresponding conformal weights $ \lambda_i>0$ for $0<i\leq d$. Then $0<\qdim_{V}M^i<\infty$ for any $0\leq i\leq d.$ In fact, we have
    \begin{equation}\label{eq:qdim=Si0/S00}
       \qdim_V M^i  = \frac{S_{i, 0}}{S_{0, 0}}.
    \end{equation}
\end{prop}

\paragraph{ \textbf{The case that  $ \cD$ is self-dual}}

Denote  $ \xi = e^{( 2 \pi \sqrt{-1} )/3 }$  a primitive cubic root of unity.

We define a function $ \Xi : \mathds{Z} \to \{ -1, 2 \}$ by
\begin{align*}
 \Xi ( n) := \xi^n + \xi^{ 2n },\quad \text{ for }  n \in \mathbb{Z}.
\end{align*}
Note that
\begin{align*}
 \Xi ( n ) =
 \begin{cases}
2 & \textrm{ if } n \equiv 0 \mod 3,\\
-1 & \textrm{ otherwise}.
\end{cases}
\end{align*}

\begin{prop}\label{thm:D_selfdual}
 Let $ \cD$ be a self-dual $ \mathds{Z}_3$-code of length $ \ell$, then we have
 \begin{align*}
T[0] \times T[0] =
 &\frac{ 2^{ \ell - 2d}+ \Xi( \ell) }{3} \check{T}[0]
 + \frac{ 2^{ \ell - 2d}+ \Xi( \ell+ 2) }{3} \check{T}[1]
 +\frac{ 2^{ \ell - 2d}+ \Xi( \ell  + 1) }{3} \check{T}[2]\\
 = &\sum_{ \varepsilon = 0,1,2} \frac{ 2^{ \ell - 2d}+ \Xi( \ell - \varepsilon ) }{3} \check{T}[ \varepsilon ] .
\end{align*}
\end{prop}
\begin{proof}
We mimic the proof of \cite[Lemma 18]{Mi13a}.

   Let $V$ denote the lattice VOA $ \VLCD$. Since $ \cD$ is self-dual, we know $V$ is the  unique $ \tau$-stable irreducible module. Thus,  $V$ has exactly one $ \tau^i$-twisted module  for each $ i =1,2$. We denoted them by $ T$ and $ \check{T}$, respectively.  Let
   \begin{align*}
M^i :=& V[ i],&  M^{ 3 +i} :=& T[ i],&  M^{ 6 +i} :=&  \check{T}[ i],&
\end{align*}
for $ i = 0,1,2$. Then we know $ M^j,  ( j = 0, \cdots , 8)$ are irreducible $ V^ \tau$-modules. Note that there are also irreducible $V^ \tau$-modules which are not $\tau$-stable, but we won't need in the proof.

Denote $ C_1(g, h)$ the vector space generated by trace functions of  $g$-twisted and $h$-stable $V$-modules. By \cite{DLM00} we know the modular transformation $ \Gamma :z \mapsto \frac{-1}{ z}$ maps $ C_1 (g, h) $ to $ C_1 ( h, g ^{-1})$. In particular, $ \Gamma $ sends $ C_1( \tau, \tau^j ) $ to $ C_1( \tau^j , \tau^2 ) $ for $ j = 0,1,2$.

First, we know $ Z_{ T } ( \tau , 1  ; \frac{-1}{z}) \in C_1( 1, \tau^2)$ which is spanned by $  Z_V ( 1, 1; z)$. Therefore, we can write
\begin{align*}
Z_{ T } ( \tau , 1  ; \frac{-1}{z}) &= \lambda_1 Z_V(1,1; z),
\end{align*}
for some $ \lambda_1 \in \mathds{C}$.

Denote $W^i(g,h,z) = Z_{M^i}(g,h;z)$ for any $i$. Then we have
\begin{align}\label{eq:S_tau:1}
 W^3(  \frac{-1}{z} ) + W^4( \frac{-1}{z} ) + W^5 ( \frac{-1}{z} )
&= \lambda_1  \big(  W^0( z)  + W^1( z)  + W^2( z)  \big).
\end{align}

Similarly, using $ Z_{ T  } ( \tau , \tau^j   ; \frac{-1}{z}) \in C_1( \tau^j, \tau^2)$ for $ j =1,2$, we can write
\begin{equation}\label{eq:tau_decomp:1}
\begin{gathered}
W^3( 1, \tau , \frac{-1}{z} )  + W^4( 1, \tau , \frac{-1}{z} ) + W^5 ( 1, \tau, \frac{-1}{z} )
  =  \mu_1  \bigg(  W^3 ( 1, \tau^2 ; z)   + W^4( 1, \tau^2 ; z)  + W^5( 1, \tau^2; z)  \bigg),
  \shortintertext{and}
W^3( 1, \tau^2 , \frac{-1}{z} )  + W^4( 1, \tau^2 , \frac{-1}{z} ) + W^5 ( 1, \tau^2, \frac{-1}{z} )
 = \mu_2  \bigg(  W^6 ( 1, \tau^2 ; z)  + W^7( 1, \tau^2 ; z)  + W^8( 1, \tau^2; z)  \bigg),
\end{gathered}
\end{equation}
for some $ \mu_1, \mu_2 \in \mathds{C}$.

We can define a linear isomorphism $ \varphi( \tau^j)$ as following:
\begin{align*}
\varphi( \tau^j ) =& \xi^{  ij} \textrm{ on } M^{ 3 +i } \textrm{ and } M^{ 6 + i }.
\end{align*}
Therefore we can rewrite the  equation \eqref{eq:tau_decomp:1}  as
\begin{equation}\label{eq:tau_decomp:2}\begin{aligned}
  W^3( \tau, 1; \frac{-1}{z} ) + \xi W^4 ( \tau, 1; \frac{-1}{z} ) + \xi^2 W^5 ( \tau, 1; \frac{-1}{z} )
=&  \mu_1  \bigg(   W^3( \tau, 1; z ) + \xi^2  W^4( \tau, 1; z ) + \xi W^5( \tau, 1; z ) \bigg), \\
 W^3( \tau, 1; \frac{-1}{z} ) + \xi^2 W^4 ( \tau, 1; \frac{-1}{z} ) + \xi W^5( \tau, 1; \frac{-1}{z} )
=&  \mu_2  \bigg(   W^6 ( \tau^2, 1;z ) + \xi^2  W^7( \tau^2, 1; z ) + \xi W^8( \tau^2, 1; z ) \bigg).
\end{aligned}\end{equation}

Solving equations ~\eqref{eq:S_tau:1} and ~\eqref{eq:tau_decomp:2}, we know
\begin{align*}
   W^3( \frac{-1}{z} )  =&
 \frac{ \lambda_1}{3} \big( W^0(z )  + \xi^2 W^1 (z ) + \xi W^2(z ) \big)\\
 &+ \frac{ \mu_1}{3} \big( W^3 (z )  + \xi^2  W^4 (z ) + \xi  W^5 (z ) \big)
 + \frac{ \mu_2}{3} \big( W^6 (z )  + \xi   W^7 (z ) + \xi^2  W^8  (z ) \big), \\
  W^4( \frac{-1}{z} )  =&
 \frac{ \lambda_1}{3} \big( W^0(z )  + \xi^2 W^1 (z ) + \xi W^2(z ) \big)\\
 &+ \frac{ \mu_1}{3} \big( \xi^2 W^3 (z )  + \xi  W^4 (z ) +   W^5 (z ) \big)
 + \frac{ \mu_2}{3} \big( \xi W^6 (z )  + \xi^2   W^7 (z ) +   W^8  (z ) \big), \\
  W^5( \frac{-1}{z} )  =&
 \frac{ \lambda_1}{3} \big( W^0(z )  + \xi^2 W^1 (z ) + \xi W^2(z ) \big)\\
 &+ \frac{ \mu_1}{3} \big( \xi W^3 (z )  +  W^4 (z ) +   \xi^2 W^5 (z ) \big)
 + \frac{ \mu_2}{3} \big( \xi^2 W^6 (z )  +       W^7 (z ) +  \xi  W^8  (z ) \big).
\end{align*}

In other words, the rows $ S_{i, j}$ for $ i = 3,4,5$ are given by
\begin{align*}
\frac{1}{3}
\begin{pmatrix}
\lambda_1&\xi^2\lambda_1&\xi \lambda_1& \mu_1&\xi^2 \mu_1& \xi \mu_1 &\mu_2&\xi \mu_2&\xi^2 \mu_2 &0 & \cdots &0 \cr
\lambda_1&\xi^2\lambda_1&\xi \lambda_1& \xi^2 \mu_1&\xi \mu_1&\mu_1&\xi \mu_2&\xi^2 \mu_2&\mu_2 &0 &\cdots   &0\cr
\lambda_1&\xi^2\lambda_1&\xi \lambda_1&  \xi \mu_1& \mu_1&\xi^2 \mu_1&\xi^2 \mu_2&\mu_2&\xi \mu_2 &0 &  \cdots&0
 \end{pmatrix}.
\end{align*}

Since $ S^2_{0, 0} \mathrm{glob}(V^ \tau) =1 $, we know
\begin{align*}
& S_{0,0}^2  \cdot 9 \left| \cC ^\perp/ \cC \right| \left| \cD ^\perp/ \cD \right| =1,\\
&S_{0, i}/S_{ 0, 0 } = \qdim M^i,  \\
\shortintertext{and}
& \qdim M^i =
\begin{cases}
1, & \textrm{ if } i = 0,1,2 \\
\frac{ 2^ \ell}{ \left| \cC \right|} , & \textrm{ if }  3 \le i \le 8.
\end{cases}
\end{align*}
 This gives $ S_{0,0} = S_{0,1} = S_{0,2} = \frac{\pm 2^{ 2 d -  \ell }} {3} $ and $ \lambda_1 = 3 S_{ 0, h} = \pm 1 $ for $ 3\le h \le 8$.

By Verlinde formula, we know fusion rules are given by
\begin{align*}
N_{3,3}^6 &= \frac{ 3 \cdot ( \lambda_1/3)^3}{ S_{0,0}} + \frac{ 3 \big(  ( \mu_1/3)^3 + ( \mu_2/3)^3 \big) } { ( \lambda_1/3)}
=
\frac{ \pm \big(   2^{ \ell - 2d }  +    \mu_1^3 +  \mu_2^3   \big) }{3  }, \\
N_{3,3}^7  &= \frac{   \pm \big(   2^{ \ell - 2d }     +    \xi^2 \mu_1^3 + \xi \mu_2^3    \big)  }{3 }     ,\\
N_{3,3}^8  &= \frac{ \pm \big(   2^{ \ell - 2d }     +    \xi \mu_1^3 + \xi^2 \mu_2^3    \big)  }{3  }   .
\end{align*}
Since
\begin{align*}
N_{3,3}^6 + N_{3,3}^7 + N_{3,3}^8 = 2^{ \ell - 2 d},
\end{align*}
we know
\begin{align*}
 S_{0,0}  &= 2^{ 2 d - \ell },
 \shortintertext{and}
N_{3,3}^6 &= \frac{     2^{ \ell - 2d }  +    \mu_1^3 +  \mu_2^3     }{3  },	 \\
 N_{3,3}^7  &= \frac{        2^{ \ell - 2d }     +    \xi^2 \mu_1^3 + \xi \mu_2^3      }{3 }     ,\\
N_{3,3}^8  &= \frac{      2^{ \ell - 2d }     +    \xi \mu_1^3 + \xi^2 \mu_2^3     }{3  }   .
\end{align*}

By  $ S_{3,6} = 1$  we have
\begin{align*}
9 = 3 \lambda_1 ^2 + 6 \mu_1 \mu_2,
\shortintertext{and hence}
\mu_1 \mu_2 = 1.
\end{align*}

Notice that the weights of irreducible $V^ \tau$-modules of twisted type  are
\begin{align*}
\wt V^{ T, \cZ }_\LCD ( \tau^i)  [ \varepsilon ]   \in   \ell /9   + 1/3 ( \sum_{ \bm{ e } \in S_ \varepsilon} e_i ) + \mathds{Z}
= \varepsilon/3 + \ell /9 + \mathds{Z},
\end{align*}
for $ i =1,2$.
By considering the characters, we have from the above $S$-matrix that 
 \begin{align*}
  Z_V (1,\tau; z)=&
\ch(W^0)+\xi\ch(W^1)+\xi^2\ch(W^2),\\
Z_V(1,\tau;-1/z)=&\lambda_1\{\ch(W^3)+\ch(W^4)+\ch(W^5)\}, \\
Z_V(1,\tau;-1/(z+1))=& e^{2\pi \sqrt{-1}N/24} \cdot e^{2\pi \sqrt{-1} \ell /9} \lambda_1\{\ch(W^3)+\xi\ch(W^4)+\xi^2\ch(W^5)\},\\
Z_V(1,\tau;-1/((-1/z)+1))=&e^{2\pi \sqrt{-1}N/24} \cdot e^{2\pi \sqrt{-1} \ell /9}  \lambda_1\mu_1\{\ch(W^3)+\xi^2\ch(W^4)+\xi\ch(W^5)\},
\end{align*}
where $ N = 2 \ell$ is the rank of the lattice $ \LCD$.  On the other hand, since
\begin{align*}
  Z_V(1,\tau;-1/((-1/z)+1))=&Z_V(1,\tau;-1-\frac{1}{z-1}) \\
=&e^{-2\pi \sqrt{-1}N/24}    Z_V(1,\tau,-1/(z-1)) \\
=&e^{-4\pi \sqrt{-1}N/24} \cdot e^{ -2\pi \sqrt{-1}  \ell /9}  \lambda_1\{\ch(W^3)+\xi^2\ch(W^4)+\xi \ch(W^5)\}
\end{align*}
we have $\mu_1 \cdot e^{ 6\pi \sqrt{-1}N/24}   \cdot e^{4\pi \sqrt{-1}  \ell /9}  =1$. Since $ N = 2 \ell$ and $ \ell$ is a multiple of 4, we know $ 8 | N$ and   $\mu_1 =  e^{ - 4\pi \sqrt{-1} \ell /9}   $.  Using $ \mu_1 \mu_2 =1$, we have $ \mu_1^3 = \xi^ \ell$ and $ \mu^3_2 =   \xi^{ 2 \ell }    $. This gives
\begin{align*}
  & T[0] \times T[0] \\
=& \frac{ 2^{ \ell - 2d}+ \xi^ \ell   + \xi^{ 2 \ell}}{3} \check{T}[0]
 + \frac{ 2^{ \ell - 2d} + \xi^ {  \ell +2 } + \xi^{  2 \ell +1 }   }{3} \check{T}[1 ]
+\frac{ 2^{ \ell - 2d} + \xi^{   \ell+ 1}  + \xi^{  2 \ell + 2 }  }{3} \check{T}[2]\\
= &\frac{ 2^{ \ell - 2d}+ \Xi( \ell) }{3} \check{T}[0]
 + \frac{ 2^{ \ell - 2d}+ \Xi( \ell+ 2) }{3} \check{T}[1]
 +\frac{ 2^{ \ell - 2d}+ \Xi( \ell  + 1) }{3} \check{T}[2],
\end{align*}
and completes the proof.
\end{proof}

\paragraph{ \textbf{General Case}}

Recall the decomposition
\begin{align}\label{eq:decomp_twisted_D_to_Z}
  V^{ T, \bm{\eta}}_\LCD ( \tau^i)  [ \varepsilon ] \cong
    \bigoplus_{ \cgg \in \cD}
       V^{ T, \bm{\eta}  - i \cgg }_\LC ( \tau^i)  [ \varepsilon ].
\end{align}

In the following,  we will denote
\begin{align*}
  T_{\cC \times \cD}^\cee [ \varepsilon] & \coloneqq V^{ T, \bm{\eta}}_\LCD ( \tau)  [ \varepsilon ], \\
    \check{T}_{\cC \times\cD}^\cee [ \varepsilon] & \coloneqq V^{ T, \bm{\eta}}_\LCD ( \tau^2)  [ \varepsilon ], \\
    \shortintertext{in addition, we let}
     T_{\cC \times\cD} [ \varepsilon] & \coloneqq V^{ T, \cZ }_\LCD ( \tau)  [ \varepsilon ], \\
    \check{T}_{\cC \times\cD} [ \varepsilon] & \coloneqq V^{ T, \cZ }_\LCD ( \tau^2)  [ \varepsilon ].
\end{align*}
We also let $ \cZ$ be the trivial $ \mathds{Z} _3$-code of various length depending on context.

\begin{prop}\label{thm:C:self_dual}
  Let $ \cB$ be a self-dual $ \GF$-code of length 2. Then
\begin{align*}
T_{ \cB \times \cZ} [0] \times T_{ \cB \times \cZ} [0] = \check{T}_{ \cB \times \cZ} [ 2].
\end{align*}
 \end{prop}

 \begin{proof}
In this case, all irreducible modules are simple current. It suffices to find the nonzero fusion rules.

 Let $ \cB^2 := \cB \oplus \cB$ be a self-dual code of length 4 and let $ \cS$ be a self-dual $  \mathds{Z}_3$-code of length 4. By Prop.~\ref{thm:D_selfdual}, we know
   \begin{align*}
T_{ \cB^2 \times \cS} [ 0] \times T_{ \cB^2 \times \cS} [ 0]
= \check{T}_{ \cB^2 \times \cS} [  1 ].
\end{align*}

Considering the sub VOA $ V^ \tau_{ L_{ \cB^2 \times \cZ} } \subset V^ \tau_{ L_{ \cB^2 \times \cS} }$, we have the decomposition of $V^ \tau_{ L_{ \cB^2 \times \cZ} }$-modules
\begin{align*}
T_{ \cB^2 \times \cS} [  \varepsilon ] &= \oplus_{ \cee \in \cS } T_{ \cB^2 \times \cZ}^ \cee [  \varepsilon ],\\
\check{T}_{ \cB^2 \times \cS } [  \varepsilon ] &= \oplus_{ \cee \in \cS } \check{T}_{ \cB^2 \times \cZ}^ \cee [  \varepsilon ].
\end{align*}
By Prop. \ref{thm:inj:fusion} we have
\begin{align*}
1 = N\fusion{ T_{ \cB^2 \times \cS} [ 0]}{ T_{ \cB^2 \times \cS} [ 0] }{ \check{T}_{ \cB^2 \times \cS} [ 1]}
\le N\fusion{ T_{ \cB^2 \times \cZ} [ 0] }{ T_{ \cB^2 \times \cZ} [ 0] }{  \oplus_{ \cee \in \cS } \check{T}_{ \cB^2 \times \cZ}^ \cee [ 1]}
=  N\fusion{ T_{ \cB^2 \times \cZ} [ 0] }{ T_{ \cB^2 \times \cZ} [ 0] }{   \check{T}_{ \cB^2 \times \cZ} [  1]},
\end{align*}
where the last equality follows from Prop.~\ref{thm:fusion:V}.
Since $ \cB^2$ is self-dual, then we have
\begin{align*}
  T_{ \cB^2 \times \cZ} [ 0] \times T_{ \cB^2 \times \cZ} [ 0] = \check{T}_{ \cB^2 \times \cZ} [ 1 ].
\end{align*}

Now consider the subVOA $ V^\tau_{ \cB \times \cZ} \otimes V^\tau_{ \cB \times \cZ} \subset  V^ \tau_{ L_{ \cB^2 \times \cZ} } $ and the decomposition
\begin{align*}
T_{ \cB^2 \times \cZ} [ \varepsilon ] =
\oplus_{ \varepsilon_0 = 0,1,2} T_{ \cB \times \cZ} [  \varepsilon_0 ] \otimes  T_{ \cB \times \cZ} [  \varepsilon - \varepsilon_0 ]
\end{align*}
of $ V^\tau_{ \cB \times \cZ} \otimes V^\tau_{ \cB \times \cZ} $-modules.

Using the decomposition we have
\begin{equation}\label{eq:cB_ineq}
\begin{aligned}
    1=& N\fusion{ T_{ \cB^2 \times \cZ} [ 0] }{ T_{ \cB^2 \times \cZ} [ 0] }{   \check{T}_{ \cB^2 \times \cZ}^ \cee [  1 ]}
\le N\fusion{ T_{ \cB\times \cZ} [ 0] \otimes T_{ \cB\times \cZ} [ 0]  }{ T_{ \cB \times \cZ} [ 0] \otimes T_{ \cB\times \cZ} [ 0]  }{  \sum_{ \varepsilon_0 = 0,1,2} \check{T}_{ \cB \times \cZ} [  \varepsilon_0 ] \otimes \check{T}_{ \cB \times \cZ} [  1 - \varepsilon_0 ] } \\
= &  \sum_{ \varepsilon_0 = 0,1,2}  N\fusion{ T_{ \cB\times \cZ} [ 0] }{ T_{ \cB \times \cZ} [ 0] }{   \check{T}_{ \cB \times \cZ}[  \varepsilon_0 ]} N\fusion{ T_{ \cB \times \cZ} [ 0] }{ T_{ \cB \times \cZ} [ 0] }{   \check{T}_{ \cB \times \cZ}  [  1 - \varepsilon_0 ]}.
\end{aligned}\end{equation}

Now since $ \cB$ is self-dual, only one of the fusion rules $ N\fusion{ T_{ \cB\times \cZ} [ 0] }{ T_{ \cB \times \cZ} [ 0] }{   \check{T}_{ \cB \times \cZ}  [  \varepsilon_0 ]} , ( \varepsilon = 0,1,2)$ is nonzero. To have the above inequality \eqref{eq:cB_ineq}, we must have $ N\fusion{ T_{ \cB\times \cZ} [ 0] }{ T_{ \cB \times \cZ} [ 0] }{   \check{T}_{ \cB \times \cZ}  [  \varepsilon_0 ]}  = \delta_{ \varepsilon_0, 2}$. This completes the proof.
 \end{proof}

\begin{prop}\label{thm:fusionV:general}
 Let $ \cC$ and $ \cD$ be self-orthogonal codes of length $ \ell$.
 \begin{enumerate}[label=(\roman*)]
   \item If the length $ \ell $ is even, then we have
\begin{align*}
  T[0] \times T[0] = &\sum_{ \varepsilon = 0,1,2} \frac{ 2^{ \ell - 2d}+ \Xi( \ell - \varepsilon ) }{3} \check{T}[ \varepsilon ] .
\end{align*}


   \item If the length $ \ell $ is odd, then we have
\begin{align*}
 T[0] \times T[0] =
 &\sum_{ \varepsilon = 0,1,2} \frac{ 2^{ \ell - 2d} - \Xi( \ell - \varepsilon ) }{3} \check{T}[ \varepsilon ].
\end{align*}

\end{enumerate}
As a summary, we have
\begin{align*}
 T[0] \times T[0] =
 &\sum_{ \varepsilon = 0,1,2} \frac{ 2^{ \ell - 2d} + ( -1)^ \ell  \Xi( \ell - \varepsilon ) }{3} \check{T}[ \varepsilon ].
\end{align*}
It also  implies
\begin{align*}
  V^{ T,  \cee_1 }_\LCD ( \tau^i)  [ \varepsilon_1 ] \times V^{ T,  \cee_2 }_\LCD ( \tau^i)  [ \varepsilon_2 ]
  =&  \sum_{ \varepsilon = 0,1,2} \frac{ 2^{ \ell - 2d} + ( -1)^ \ell  \Xi( \ell - \varepsilon ) }{3}
  V^{ T,  -(\cee_1 + \cee_2) }_\LCD ( \tau^{ 2i})  [  \varepsilon - \varepsilon_1 - \varepsilon_2 ].
\end{align*}
\end{prop}
\begin{proof}

 (i)  First we assume  $ \ell$ is a multiple of 4. Then there exists a self-dual code $ \cS$ of length $ \ell$.

Restricting to $ V^ \tau_{\LC}$-modules, we know
\begin{align}\label{eq:D:zero_ineq}
 N\fusion{ T_{ \cC \times \cS}[0] } { T_{ \cC \times \cS}[0]  }{ \check{ T}_{ \cC \times \cS} [ \varepsilon] }
 \le N\fusion {  T_{ \cC \times \cZ} [0] } {  T_{ \cC \times \cZ} [0] } { \oplus_{ \cee \in \cS } \check{T}^\cee_{\cC \times \cZ} [ \varepsilon ] }
 = N\fusion {  T_{ \cC \times \cZ} [0] } {  T_{ \cC \times \cZ} [0] } {  \check{T}_{ \cC \times \cZ} [ \varepsilon] }.
\end{align}
On the other hand, we know  from \eqref{eq:x+y+z} that
\begin{align*}
 \sum_{ \varepsilon = 0,1,2 }  N\fusion{ T_{ \cC \times \cS}[0] } { T_{ \cC \times \cS}[0]  }{ \check{ T}_{ \cC \times \cS} [ \varepsilon] }
 = \sum_{ \varepsilon = 0,1,2 } N\fusion {  T_{ \cC \times \cZ} [0] } {  T_{ \cC \times \cZ} [0] } { \check{T}_{ \cC \times \cZ} [ \varepsilon] }.
\end{align*}
Therefore, the inequality in ~\eqref{eq:D:zero_ineq} must attain equality and we prove (i) when $ \cD$ is the trivial code of length  divisible by 4.

Now let $ \cD$ be a self-orthogonal code of length $ \ell$. Similarly, we have
\begin{align*}
  N\fusion{T_{ \cC \times \cD  }[0] } {T_{ \cC \times \cD  }[0]  }{ \check{T}_{ \cC \times \cD  } [ \varepsilon] }
 \le N\fusion {  T_{ \cC \times \cZ} [0] } {  T_{ \cC \times \cZ} [0] }{  T_{ \cC \times \cZ} [ \varepsilon] }.
\end{align*}
The same argument as in the case for $\cD=\mathbf{0}$ shows
\begin{align*}
  N\fusion{T_{ \cC \times\cD  }[0] } {T_{ \cC \times \cD  }[0]  }{ \check{T}_{ \cC \times \cD  } [ \varepsilon] }
 = N\fusion {  T_{ \cC \times \cZ} [0] } {  T_{ \cC \times \cZ} [0] }{  T_{ \cC \times \cZ} [ \varepsilon] }.
\end{align*}
This implies
\begin{align*}
  N\fusion{T_{ \cC \times\cD  }[0] } {T_{ \cC \times \cD  }[0]  }{ \check{T}_{ \cC \times \cD  } [ \varepsilon] }
=   N\fusion{T_{ \cC \times\cS  }[0] } {T_{ \cC \times \cS  }[0]  }{ \check{T}_{ \cC \times \cS  } [ \varepsilon] },
\end{align*}
and proves (i) by ~Prop. \ref{thm:D_selfdual} when $ \ell $ is a multiple of 4.

  Now assume $ \ell \equiv  2 \mod 4 $. Let $ \cB$ be a self-dual $\GF$-code of length 2. Then  $   { \cC \oplus \cB} $ is a self-orthogonal code of length divisible by 4 and $ ( \cD \oplus \cZ) $ is a self-orthogonal code of the same length.

Restricting to
\begin{align*}
V^ \tau_\LCD \otimes V^ \tau_{L_{ \cB \times \cZ} } \subset V^ \tau_{L_{ (  \cC \oplus \cB) \times ( \cD \oplus \cZ) }},
\end{align*}
we know
\begin{align*}
  T_{(  \cC \oplus \cB) \times ( \cD \oplus \cZ) } [ \varepsilon ]  &= \bigoplus_{ \varepsilon_0 = 0,1,2} T_{ \cC \times \cD} [ \varepsilon_0 ] \otimes T_{ \cB \times \cZ} [ \varepsilon - \varepsilon_0];
\intertext{moreover}
    &N\fusion{ T_{(  \cC \oplus \cB) \times ( \cD \oplus \cZ) } [0] }{ T_{(  \cC \oplus \cB) \times ( \cD \oplus \cZ) } [0] } { \check{T} _{(  \cC \oplus \cB) \times ( \cD \oplus \cZ) } [ \varepsilon ] } \\
   \le &\bigoplus_{ \varepsilon_0  = 0,1,2} N\fusion{ T_{\cC \times \cD} [0] } { T_{\cC \times \cD} [0] }{ \check{T} _{\cC \times \cD} [ \varepsilon - \varepsilon_0  ] } N\fusion{ T_{\cB \times \cZ } [0] }{ T_{\cB \times \cZ } [0] }{ \check{T}_{\cB \times \cZ } [   \varepsilon_0 ]}.
\end{align*}
By Prop.~\ref{thm:C:self_dual} we know
\begin{align*}
  T_{\cB \times \cZ } [0] \times T_{\cB \times \cZ } [0] = \check{T}_{\cB \times \cZ } [2 ];
\end{align*}
 therefore the above inequality becomes
\begin{align*}
  N\fusion{ T_{(  \cC \oplus \cB) \times ( \cD \oplus \cZ) } [0] }{ T_{(  \cC \oplus \cB) \times ( \cD \oplus \cZ) } [0] } { \check{T} _{(  \cC \oplus \cB) \times ( \cD \oplus \cZ) } [ \varepsilon ] }
  \le N\fusion{ T_{\cC \times\cD} [0] } { T_{\cC \times\cD} [0] }{ \check{T}_{\cC \times\cD} [   \varepsilon -2 ] }.
\end{align*}
We know $ \sqrt{ \abs{ \frac{  ( \cC \oplus \cB) ^\perp}{ \cC \oplus \cB}}  } = \sqrt{ \abs{ \frac{ \cC ^\perp}{ \cC} } }$ and hence by~\eqref{eq:x+y+z}
\begin{align*}
   \sum_{ \varepsilon = 0,1,2 } N\fusion{ T_{(  \cC \oplus \cB) \times ( \cD \oplus \cZ) } [0] }{ T_{(  \cC \oplus \cB) \times ( \cD \oplus \cZ) } [0] } { \check{T} _{(  \cC \oplus \cB) \times ( \cD \oplus \cZ) } [ \varepsilon ] }
   = \sum_{ \varepsilon = 0,1,2 } N\fusion{ T_{\cC \times\cD} [0] } { T_{\cC \times\cD} [0] }{ \check{T}_{\cC \times\cD} [ \varepsilon  ] }.
\end{align*}
Therefore, we must have
\begin{align*}
   N\fusion{ T_{(  \cC \oplus \cB) \times ( \cD \oplus \cZ) } [0] }{ T_{(  \cC \oplus \cB) \times ( \cD \oplus \cZ) } [0] } { \check{T} _{(  \cC \oplus \cB) \times ( \cD \oplus \cZ) } [ \varepsilon ] }
   =   N\fusion{  T_{\cC \times\cD}[0] } { T_{\cC \times\cD} [0] }{ \check{T}_{\cC \times\cD} [  \varepsilon  - 2  ] }.
\end{align*}
Note that   $ \cC \oplus \cB$ has length $ \ell + 2$, thus we have
 \begin{align*}
  &N\fusion{  T_{\cC \times\cD}[0] } { T_{\cC \times\cD} [0] }{ \check{T}_{\cC \times\cD} [  \varepsilon  ] }
  = N\fusion{ T_{(  \cC \oplus \cB) \times ( \cD \oplus \cZ) } [0] }{ T_{(  \cC \oplus \cB) \times ( \cD \oplus \cZ) } [0] } { \check{T} _{(  \cC \oplus \cB) \times ( \cD \oplus \cZ) } [ \varepsilon +  2 ] }\\
  =&  \frac{ 2^{ \ell - 2d} +  \Xi( \ell +2  - \varepsilon - 2   )   }{3}
  =  \frac{ 2^{ \ell - 2d} +  \Xi( \ell   -  \varepsilon)   }{3}.
\end{align*}
This proves (i) when $ \ell \equiv 2 \mod 4$.

Now assume $ \ell $ is odd, let $ \cC_e :=  \cC \oplus \cZ $ and $ \cD_e := \cD  \oplus \cZ $ be  self-orthogonal codes of even length $ \ell +1$.
Restricting to the subVOA $ V^\tau_{\LCD} \otimes V^\tau_L$, we have decomposition of $ V^\tau_{\LCD} \otimes V^\tau_{ \cZ \times \cZ }$-modules
\begin{align*}
   T_{ \cC_e \times \cD_e} [ 0 ] =     \bigoplus_{ \varepsilon_0 = 0,1,2} T_{ \cC \times \cD} [ \varepsilon_0 ] \otimes T_{ \cZ \times \cZ} [   - \varepsilon_0].
\end{align*}
We know that $ V^\tau_{ \cZ \times \cZ } = V^ \tau_L$. Recall  the fusion rules of $ V^ \tau_L$:
\begin{align*}
  T_{ \cZ \times \cZ}[0] \times T_{ \cZ \times \cZ} [0] = T_{ \cZ \times \cZ} [0] + T_{ \cZ \times \cZ} [2].
\end{align*}

Therefore,
\begin{equation}\label{eq:sum_even_2odd}
\begin{aligned}
N\fusion{ T_{ \cC_e \times \cD_e} [0]  } { T_{ \cC_e \times \cD_e} [0] } { \check{T}_{ \cC_e \times \cD_e} [0] }
 & \le N\fusion{ T_{ \cC \times \cD} [0]  } { T_{ \cC \times \cD} [0]  } { \check{T}_{ \cC \times \cD} [0]  }
  + N\fusion{ T_{ \cC \times \cD} [0]  } { T_{ \cC \times \cD} [0]  } { \check{T}_{ \cC \times \cD} [1]  }, \\
N\fusion{ T_{ \cC_e \times \cD_e} [0]  } { T_{ \cC_e \times \cD_e} [0] } { \check{T}_{ \cC_e \times \cD_e} [1 ] }
 & \le N\fusion{ T_{ \cC \times \cD} [0]  } { T_{ \cC \times \cD} [0]  } { \check{T}_{ \cC \times \cD} [1]  }
  + N\fusion{ T_{ \cC \times \cD} [0]  } { T_{ \cC \times \cD} [0]  } { \check{T}_{ \cC \times \cD} [2]  }, \\
  N\fusion{ T_{ \cC_e \times \cD_e} [0]  } { T_{ \cC_e \times \cD_e} [0] } { \check{T}_{ \cC_e \times \cD_e} [ 2 ] }
 & \le N\fusion{ T_{ \cC \times \cD} [0]  } { T_{ \cC \times \cD} [0]  } { \check{T}_{ \cC \times \cD} [0]  }
  + N\fusion{ T_{ \cC \times \cD} [0]  } { T_{ \cC \times \cD} [0]  } { \check{T}_{ \cC \times \cD} [2]  }.
\end{aligned}\end{equation}

This gives
\begin{align*}
\sum_{ \varepsilon = 0,1,2} N\fusion{ T_{ \cC_e \times \cD_e} [0]  } { T_{ \cC_e \times \cD_e} [0] } { \check{T}_{ \cC_e \times \cD_e} [  \varepsilon ] }
& \le 2 \sum_{ \varepsilon = 0,1,2} N\fusion{ T_{ \cC\times \cD} [0]  } { T_{ \cC \times \cD} [0] } { \check{T}_{ \cC \times \cD} [  \varepsilon ] }.
\end{align*}
From \eqref{eq:x+y+z} we know
\begin{align*}
  \sum_{ \varepsilon = 0,1,2} N\fusion{ T_{ \cC_e \times \cD_e} [0]  } { T_{ \cC_e \times \cD_e} [0] } { \check{T}_{ \cC_e \times \cD_e} [  \varepsilon ] } &= 2^{ \ell +1 - 2d } = 2 \cdot 2^{ \ell - 2d }
= 2 \sum_{ \varepsilon = 0,1,2} N\fusion{ T_{ \cC\times \cD} [0]  } { T_{ \cC \times \cD} [0] } { \check{T}_{ \cC \times \cD} [  \varepsilon ] }.
\end{align*}
Therefore inequalities in ~\eqref{eq:sum_even_2odd} must attain equalities. This gives
\begin{align*}
  N\fusion{ T_{ \cC \times \cD} [0]  } { T_{ \cC \times \cD} [0]  } { \check{T}_{ \cC \times \cD} [0 ]  }
=& 2^{ \ell - 2d } - \bigg(  N\fusion{ T_{ \cC \times \cD} [0]  } { T_{ \cC \times \cD} [0]  } { \check{T}_{ \cC \times \cD} [1]  }
  + N\fusion{ T_{ \cC \times \cD} [0]  } { T_{ \cC \times \cD} [0]  } { \check{T}_{ \cC \times \cD} [2]  }  \bigg) \\
  =&  2^{ \ell - 2d } - N\fusion{ T_{ \cC_e \times \cD_e} [0]  } { T_{ \cC_e \times \cD_e} [0] } { \check{T}_{ \cC_e \times \cD_e} [1 ] }
  =2^{ \ell - 2d } - \frac{ 2^{ \ell+1  - 2d} +  \Xi( \ell + 1 -1  )   }{3} \\
  =&  \frac{ 2^{ \ell  - 2d}  -  \Xi( \ell  )   }{3}.
\end{align*}
Similarly,  we have
\begin{align*}
   N\fusion{ T_{ \cC \times \cD} [0]  } { T_{ \cC \times \cD} [0]  } { \check{T}_{ \cC \times \cD} [1 ]  }
= &  2^{ \ell - 2d } - N\fusion{ T_{ \cC_e \times \cD_e} [0]  } { T_{ \cC_e \times \cD_e} [0] } { \check{T}_{ \cC_e \times \cD_e} [ 2 ] }
  = \frac{ 2^{ \ell  - 2d}  -  \Xi( \ell + 1 -   2 )   }{3},\\
   N\fusion{ T_{ \cC \times \cD} [0]  } { T_{ \cC \times \cD} [0]  } { \check{T}_{ \cC \times \cD} [2 ]  }
= &  2^{ \ell - 2d } - N\fusion{ T_{ \cC_e \times \cD_e} [0]  } { T_{ \cC_e \times \cD_e} [0] } { \check{T}_{ \cC_e \times \cD_e} [0  ] }
=   \frac{ 2^{ \ell  - 2d}  -  \Xi( \ell +1   )   }{3}.
\end{align*}
This proves (ii). The final statement follows immediately from ~Prop.\ref{thm:fusion:I}.
\end{proof}

\section{$ \mathbb{Z}_3$-orbifold construction and the Monster  group}

In this section, we discuss an application of Corollary \ref{thm:VLC:simp_curr:iff} and Proposition
\ref{thm:fusionV:general}. The main purpose is to construct certain $3$-local subgroups of the Monster simple group.
\medskip

\paragraph{ $\mathbb{Z}_3$-\textbf{orbifold of the Leech lattice VOA}}
Let $ \Lambda$ denote the Leech lattice and let $ \tau$ be a fixed point free isometry of $ \Lambda$. It is well-known \cite{DLM00} that the lattice VOA $ V_ \Lambda$ has exactly one irreducible $ \tau^i$-twisted module for $ i = 1,2$. We denote these twisted modules by $V_ \Lambda ^{T,1}$ and $ V_ \Lambda^{ T,2}$, respectively.
We define the $ V_ \Lambda[0]$-module
\begin{align*}
 V^ \sharp :=&   V_ \Lambda[0] \oplus  (V_ \Lambda ^{T, 1})_ \mathds{Z} \oplus  (V_ \Lambda ^{T, 2})_ \mathds{Z},
\end{align*}
where $ (V_ \Lambda ^{T, i})_ \mathds{Z} $ is the subspace of $ V_ \Lambda ^{T, i}$ of integral weights.


The next proposition is proved in \cite{Mi13a}.

\begin{prop}[\cite{Mi13a}]
The module  $V^\sharp$ has a natural VOA structure. Moreover, it is a $\mathbb{Z}_3$ simple current extension of the VOA $V_\Lambda^\tau$.
\end{prop}

\begin{rmk}
That $V^\sharp$ has a natural VOA structure was first announced by Dong and Mason~\cite{DM93}. They also claimed that the
full automorphism group of $V^\sharp$ is isomorphic to the Monster and $V^\sharp\cong V^\natural$ as a VOA. However, the
complete proof has not been published.
\end{rmk}

In \cite{SS10}, a $3$-local characterization of the Monster simple group has been obtained.

\begin{thm}\label{thm:3local_monster} \cite[Thm.  1.1]{SS10}
   Let $G$ be a finite group and $ S \in \Syl_3 (G)$. Let $ H_1$ and $H_2$ be subgroups of $G$ containing $S$ such that
\begin{description}
   \item[M1] \label{3Monster:1} $H_1 = N_G( Z( O_3(H_1)))$, $ O_3( H_1)$ is extraspecial group of order $ 3^{13}$, $H_1/{ O_3 (H_1)} \cong 2 \Suz : 2$ and $\Ct_{ H_1} ( O_3 ( H_1)) = Z( O_3(H_1))$.
   \item[M2]\label{3Monster:2} $ H_2/ O_3( H_2) \cong \Omega_8^-(3 )$, $ O_3( H_2)$ is elementary abelian of order $ 3^8$ and $ O_3 (H_2)$ is a natural $ H_2/ O_3( H_2)$-module.
   \item[M3]\label{3Monster:3}  $ ( H_1 \cap H_2) / O_3 ( H_2)$ is an extension of an elementary abelian group of order $ 3^6$ by $ 2 \cdot PSU_4(3) : 2^2$.
  \end{description}
Then $G$ is isomorphic to the largest sporadic simple group, the Monster.
\end{thm}

In this section, we will construct explicitly certain subgroups $H_1$ and $H_2$ of $Aut(V^\sharp)$ such that $H_1$ and $H_2$ satisfy the hypotheses [M1] to [M3] above.

\medskip

\paragraph{ \textbf{Simple Current extension}}
Let $V(0)$ be a simple rational $C_2$-cofinite VOA of
CFT type and let $\{ V(\alpha) \mid \alpha \in A\}$ be a set of inequivalent irreducible $V(0)$-modules indexed by an abelian group
$A$. A simple VOA $V=\oplus_{\alpha\in A} V(\alpha)$ is called a {\it $A$-graded extension} of $V(0)$ if $V(0)$ is a full sub VOA of
$V$ and $V$ carries a $A$-grading, i.e., $Y(x^\alpha,z)x^\beta\in V^{\alpha+\beta}$ for any $x^\alpha \in V^\alpha$, $x^\beta \in
V^\beta$. In this case, the group $A^*$ of all irreducible characters of $A$ acts naturally on $V$: for $\chi\in A^*$,  $\chi
(v)=\chi(\alpha) v$, $v\in V(\alpha)$. In other words, $V(\alpha)$ is an eigenspace of $A^*$ for all $\alpha\in A$ and $V(0)$ is the
fixed points of $A^*$.

If all $V^\alpha$, $\alpha\in A$, are simple current $V(0)$-modules, then $V$ is referred to as a {\it $A$-graded simple current
extension} of $V(0)$. The abelian group $A$ is automatically finite since $V(0)$ is rational.

\begin{thm}[\cite{sh2,SY}]\label{thm:uni_lift_auto}
Let $V=\oplus_{\alpha\in A}V(\alpha)$ and $V^\prime=\oplus_{\alpha\in A}V^\prime(\alpha)$ be simple VOAs graded by a finite
abelian group $A$. Suppose that $V(0)= V^\prime(0)$ is a simple rational $C_2$-cofinite VOA of CFT type  and $V(\alpha)$ and
$V'(\alpha)$ are simple current $V(0)$-modules for all $\alpha\in A$. Let $g$ be an automorphism of $V(0)$ which maps the set
of isomorphism classes of $\{V(\alpha)|\ \alpha\in A\}$ to those of $\{V^\prime(\alpha)|\ \alpha\in A\}$. Then there exists an
isomorphism $\tilde{g}$ from $V^\prime$ to $V$ such that $\tilde{g}_{|V(0)}=g$.
\end{thm}

\begin{nota}
Let $S_A$ be the set of the isomorphism classes of the irreducible $V(0)$-modules $V(\alpha), (\alpha\in A)$.
For an automorphism $g$ of $V(0)$, we set $S_A\circ g=\{W\circ g|\ W\in S_A\}$, where $W\circ g$ denotes the $g$-conjugate of $W$, i.e., $W\circ g=W$ as a vector space and the vertex operator $Y_{W\circ g}( u,z)= Y_W(gu,z)$, for $u\in V$.
\end{nota}

Define
\begin{align*}
H^N_A&=\{h\in\Aut(V(0)) \mid \ S_A\circ h=S_A\},\\
H^C_A&=\{h\in\Aut(V(0)) \mid \ W\circ h=W\ \text{for all}\ W\in S_A\}.
\end{align*}
Then we obtain the restriction homomorphisms
\begin{align*}
\Phi^N_A&:N_{\Aut(V)}(A^*)\to H^N_A,\\
 \Phi^C_A&:C_{\Aut(V)}(A^*)\to H^C_A,
\end{align*}
Applying Theorem  to the case $V=V^\prime$, we show that $\Phi^N_A$ is surjective.
Since each $V(\alpha)$ is irreducible, $\Ker\Phi^N_A=A^*$.
By similar arguments, $\Phi^C_A$ is surjective, and $\Ker\Phi^C_A=A^*$.

\begin{thm}[\cite{sh2} (cf. \cite{SY})]\label{NCthm}
   Let $V=\oplus_{\alpha\in A}V(\alpha)$ be a simple VOA graded by a finite abelian group $A$.
Suppose that the fusion rule $V(\alpha)\times V(\beta)=V(\alpha+\beta)$ holds for all $\alpha,\beta\in A^*$.
Then the restriction homomorphism $\Phi^N_A$ and $\Phi^C_A$ are surjective and $\Ker\Phi^C_A=\Ker\Phi^N_A=A^*$. That is, we have short exact sequences
\begin{align*}
 0 \longrightarrow A ^\ast \longrightarrow N_{ \Aut(V)}( A ^\ast) \longrightarrow H^N_A \longrightarrow 0,\\
 0 \longrightarrow A ^\ast \longrightarrow C_{ \Aut(V)}( A ^\ast) \longrightarrow H^C_A \longrightarrow 0,
\end{align*}
\end{thm}

\paragraph{\textbf{A subgroup of the shape $3^{1+12}( 2\Suz:2)$}}
Next we will construct a subgroup of  the shape $3^{1+12}( 2\Suz:2)$ in $ \Aut( V^ \sharp)$.
First we recall a theorem from \cite{LY13}.

\begin{prop}[{\cite[Thm. 5.15]{LY13}}]\label{thm:ses:C_AutVL: C_OL}
 Let $L$ be an even positive definite rootless lattice.
Let $\nu$ be a fixed point free isometry of $L$ of prime order $p$ and
$\hat{\nu}$ a lift of $\nu$ in $O(\hat{L})$.
Then we have an exact sequence
\begin{equation}\label{eq:ses:C_AutVL: C_OL}
   1 \longrightarrow \Hom(L/(1-\nu)L, \mathds{Z}_p)
  \longrightarrow  \Ct_{\Aut(V_L)}(\hat{\nu})
  \stackrel{\varphi} \longrightarrow  \Ct_{O(L)}(\nu) \longrightarrow 1.
\end{equation}
\end{prop}

Recall that  $$V^\sharp = V_ \Lambda[0]  \oplus  (V_ \Lambda ^{T, 1})_ \mathds{Z} \oplus (V_ \Lambda ^{T, 2})_ \mathds{Z}.$$
There is a natural automorphism $\tau'$ of order $3$ which acts on $V^\sharp$ as $1$ on $V_ \Lambda[0] $ , as $\xi$ on $  (V_ \Lambda ^{T, 1})_ \mathds{Z}$, and as $\xi^2$ on $  (V_ \Lambda ^{T, 2})_ \mathds{Z} $.

\begin{prop}
Let $ \tau'$ be defined as above. Then
   \begin{align*}
 N_{\Aut(V^\sharp)} (\langle \tau' \rangle  )  \cong  3^{ 1 +12}. ( 2 \Suz :2).
\end{align*}
\end{prop}
 \begin{proof}
  Let $ S_A := \{  V_ \Lambda[0]  ,  (V_ \Lambda ^{T, 1})_ \mathds{Z},  (V_ \Lambda ^{T, 2})_ \mathds{Z} \}$ and
 \begin{align*}
H^N_A&=\{h\in\Aut  V_ \Lambda[0]  \mid \ S_A\circ h=S_A\}.
\end{align*}
Recall that $V_\Lambda[0]$ has exactly nine irreducible modules and only $(V_ \Lambda ^{T, 1})_ \mathds{Z}$ and $(V_ \Lambda ^{T, 2})_ \mathds{Z} $  have top weight $2$ (see \cite{Mi13a}). Therefore, $ S_A$ is invariant under the action of  $ \Aut V_ \Lambda[0] $ and   $ H^N_A = \Aut V_ \Lambda[0] $.

Now consider the simple current extension:
\begin{align*}
V_ \Lambda = V_ \Lambda[0]  \oplus V_ \Lambda[1]  \oplus V_ \Lambda[2],
\end{align*}
which is graded by $ \mathds{Z}_3$.

Let $ S_B := \{   V_ \Lambda[0]  , V_ \Lambda[1]  , V_ \Lambda[2] \}$. Then we have a short exact sequence
\begin{align*}
 0 \xrightarrow{ } \langle \hat{\tau}\rangle  \xrightarrow{ } N_{ \Aut( V_ \Lambda) } ( \langle \hat{\tau}\rangle)  \xrightarrow{\varphi}  N^N_B \xrightarrow{} 0,
\end{align*}
where
\begin{align*}
H^N_B := \{ h \in \Aut ( V_ \Lambda^ \tau) \mid S_B \circ h = S_B \}.
\end{align*}
Since $ V_ \Lambda[1]  $ and $ V_ \Lambda[2]$ are the only irreducible modules of $ V_ \Lambda[0]$ which have the top weight $1$,    all elements of  $ \Aut ( V_ \Lambda ^ \tau) $ lift to $\Aut(V_\Lambda)$ and hence
\[
 H^N_B = \Aut ( V_ \Lambda ^ \tau) \cong N_{\Aut(V_\Lambda)}(\langle \hat{\tau} \rangle)/\langle \hat{\tau}\rangle\cong 3^{12}.N_{O(\Lambda)}(\tau)/\langle \tau\rangle.
\]

Finally we will show that the subgroup $\varphi^{-1}(3^{12})$ is extra-special of order $3^{13}$.

Let $V_{\Lambda}^{T,i} = S[\tau]\otimes T_i$ be the irreducible $ \tau^i$-twisted module of $V_\Lambda$ for $ i = 1,2$. Recall the central extension (see \cite{Dl96})
\[
1\ \longrightarrow\  <-\xi >\ \longrightarrow\  \hat{\Lambda}_{\tau^i}\
\bar{\longrightarrow\ } \Lambda \longrightarrow\  1
\]
associated to the commutator map $c^{\tau^i}(\alpha, \beta)= \xi^{\langle \alpha, \beta\rangle -\langle \tau^i\alpha, \beta\rangle}$, where $\xi^3=1$ and $i=1,2$.

Set $K=\{a^{-1}\hat{\tau}^i(a)\mid a\in \hat{\Lambda}_{\tau^i}\}$. Then $\hat{\Lambda}_{\tau^i}/ K \cong 2\times  3^{1+12}$ and  $\hat{\Lambda}_{\tau^i}/ K$ acts faithfully on $T_i$. By \cite[Prop. 5.5.3]{FLM}, there is an exact sequence
\[
1\to \mathbb{C}^\times \to N_{\Aut(T_i)}(\pi(\hat{\Lambda}_{\tau^i}/K)) \xrightarrow{int} \Aut(\hat{\Lambda}_{\tau^i}) \to 1,
\]
where $\pi$ is the representation of $\hat{\Lambda}_{\tau^i}/K$ on $T_i$ and $int(g)(x)=gxg^{-1}$.

Consider a subgroup  $$X=\{(g, g', g'')\in \hat{\Lambda}/ K \times \hat{\Lambda}_{\tau^1}/ K\times \hat{\Lambda}_{\tau^2}/ K\mid  \varphi(g)=int(g')=int(g'')\}.$$
Then $X\cong 2\times 3^{1+12}$.
By the similar argument as in \cite[Section 10.4]{FLM}, one can show that $X$ acts on the space
$
V_\Lambda\oplus V_\Lambda^{T,1}\oplus V_\Lambda^{T,2}.
$
Moreover, there is a canonical embedding
$
\nu: X \to N_{\Aut(V^\sharp)}(\langle \tau'\rangle)
$
such that $\varphi(\nu(3^{1+12})) =Hom(\Lambda/(1-\tau)\Lambda, \mathbb{Z}_3)$.
\end{proof}





\paragraph{\textbf{A subgroup of the shape} $3^8 .\Omega^-(8,3).2$}
We will construct a subgroup $ H_2$ of $ \Aut( V^ \sharp)$ satisfying the following  conditions:
\begin{enumerate}[label=(\roman*)]
   \item  $ H_2/ O_3( H_2) \cong \Omega_8^-(3 ) : 2$.
   \item $ O_3( H_2)$ is elementary abelian of order $ 3^8$.
   \item $ O_3 (H_2)$ is a natural $ H_2/ O_3( H_2)$-module.
\end{enumerate}

Recall that  the Coxeter-Todd lattice $ K_{12}$ can be constructed by using the Hexacode.

Let $ \pi:  \mathds{Z} [ \xi ] \to  \mathds{Z}[ \xi]/ 2 \mathds{Z}[ \xi] \cong \F_4 $ be the natural quotient, where $ \xi $ is the primitive cubic root of unity. We know the lattice $ K_{12}$ can be defined as
   \begin{equation*}
      K_{12} := \{ (x_1, \cdots, x_6) \in ( \mathds{Z}[ \xi ])^6 \mid (\pi x_1, \cdots, \pi x_6  ) \in \cH \},
   \end{equation*}
where $ \cH$ is the hexacode.


Since $ \cH$ is self-dual, every irreducible $V^ \tau_{ K_{12}}$-module is a simple current module. There are exactly $ 3^6 \cdot 3 \cdot 3 = 3^8$ of them. These irreducible modules form an abelian group isomorphic to $ \mathds{Z}_3^8$ under the fusion product.  Denote $ R( \Ktau) $ the set of  irreducible $V^ \tau_{ K_{12}}$-modules.

As before, we denote irreducible $V^ \tau_{ K_{12}}$-modules as
\begin{align*}
    S^a[ x ], T^a[ x ], \textrm{ and } \check{T}^a[ x ]
\end{align*}
for $a \in K_{12} ^\ast \mod K_{12}, x \in \mathds{Z}_3 $. If $ a = \cZ$, we will  omit this superscript. The fusion rules are given as follows:
\begin{align}\label{FR:K12}
 S^a[x]  +  S^b[ y ] =& S^{ a +b} [ x +y],&
 S^a[x]  +  T^b[ y ] =& T^{ b -a } [ x +y],\\
 S^a[x]  +  \check{T} ^b[ y ] =& \check{T}^{ a +b} [ y- x],&
 T^a[x]  +  \check{T}^b[ y ] =& S^{b -a   } [ x -y],\\
 T^a[x]  +  T^b[ y ] =& \check{T}^{- (a +b)  } [- ( x +y)],&
 \check{T}^a[x]  +  \check{T}^b[ y ] =& {T}^{- (a +b)  } [- ( x +y)].&
\end{align}
Note that the operation $ +$ on the left hand side of the above equations is  the fusion product.

%
%
%

Recall  that  $ K_{12} ^\ast/ K_{12} \cong (\cH \times \cZ)^\ast/(\cH \times \cZ) = 0 \times \F_3^6$ as an abelian group.  It also forms a non-singular quadratic space of minus type (see for example \cite{CS})
if we define the quadratic form $$q_F(a+K_{12})= 3\langle a,a\rangle \mod 3.$$

 \begin{prop}\label{thm:O83minus}
We define a map $ q : R(\Ktau) \to \mathds{Z}_3$ by
\begin{align*}
 q( S^a[x])  =& q_F( a ),& &\textrm{ and }&  q ( T^a [x ]) =& q ( \check{T}^a [x ])=q_F ( a) + x +1.
\end{align*}
Then $q$ is a quadratic form on $ R( \Ktau)$ and $( R(\Ktau),  q)$ is non-singular space of minus type.
\end{prop}

\begin{proof}
 Let $B( x, y ) = \frac{1}{2} \big(  q ( x+ y ) - q(x) - q(y) \big) $. Then $B$ is  a symmetric form and we have
 \begin{enumerate}[label=(\roman*)]
   \item $ B( S^a[x], S^b[y] ) = 3\langle a, b \rangle$;
   \item $ B(S^a[x], T^b[y] ) = 6\langle a, b \rangle + 2 x$ and $ B(S^a[x], \check{T}^b[y] ) =  3\langle a, b \rangle +  x$;

   \item $ B(T^a[x], T^b[y] ) =  B( \check{T}^a[x],  \check{T}^b[y] ) =3\langle a, b \rangle + 2 ( x +y -1 )$;
   \item $ B( \check{T}^a[x], T^b[y] )= B( T^a[x], \check{T}^b[y] ) = 6 \langle a, b \rangle + x +y - 1$.

\end{enumerate}
It is straightforward to check that this form is  bilinear with respect to the fusion products and hence $ q$ is a quadratic form.

Using the bilinear form, it is clear that
   \begin{align*}
R( \Ktau ) = \{ S^ a[0] \mid a \in { K_{12}} ^\ast \mod K_{12}   \} \perp \Span \{ S[1], T[1]  \}.
\end{align*}
Note that $  \{ S[1], T[1]  \}$ is a hyperbolic pair and $ \{ S^ a[0] \mid a \in { K_{12}} ^\ast \mod K_{12}   \} $ is a quadratic space isometric to $K_{12}^*/K_{12}\cong \mathbb{F}_3^6$, which is a non-singular space of minus type. This completes our proof.
\end{proof}

\begin{lem}
   This quadratic form is $ \Aut V_{ K_{12}}^ \tau$-invariant.  That is for any $ M \in R(V_{ K_{12}}^ \tau) $ and $ g \in \Aut V_{ K_{12}}^ \tau$, we have $ q(M) = q( M \circ g)$.
\end{lem}
\begin{proof} Recall that the weights of irreducible $ V_L^ \tau$-modules are given by (see Tanabe and Yamada\cite[(5,10)]{TY07}):
  \begin{align*}
 \wt V_{ L^{0, j }} [ \varepsilon ] &\in  \frac{ 2 j^2 }{3} + \mathds{Z}, \\
 \wt V_L^{T, j}( \tau^i ) [ \varepsilon ] &\in  \frac{ 10 -  3 (  j^2 + \varepsilon ) }{9 } + \mathds{Z},
\end{align*}
for $ i = 1, 2, j, \varepsilon \in \mathds{Z}_3$.

Therefore, by the decompositions given in Prop.~\ref{thm:decomp:into:L}, we know that the  weights of irreducible $V_{ K_{12}}^ \tau$-modules are
\begin{align*}
 \wt S^a [ \varepsilon ] & \equiv  \frac{2 \langle a, a  \rangle}{3}= \frac{2 q( S^a[ \varepsilon ])}{3}  \mod  \mathds{Z}, \\
 \wt T^a[ \varepsilon ] = \wt \check{T}^a [ \varepsilon ] &  \equiv  \frac{ 2 ( 1 +  \varepsilon +   \langle a, a  \rangle) }{3 } = \frac{2 q( T^a[ \varepsilon ])}{3}  \mod \mathds{Z}.
\end{align*}
Since the action $ M \mapsto M \circ g$ preserves weights, we are done.
 \end{proof}

\begin{prop}
 Define a map $ \phi : \Aut V_{ K_{12}}^ \tau \to O( R( \Ktau ) , q) $ by
 \begin{align*}
 g \mapsto ( M \mapsto M \circ g ).
\end{align*}
Then $ \phi$ is a group monomorphism and $\mathrm{Im}\,\phi= {}^+\Omega^-_8(3)\cong \Omega^-_8(3).2$.
\end{prop}


\begin{proof}
 That $ \phi$ is a group homomorphism follows from that $ Y_{ M \circ g}( u, z) = Y_M ( g  u, z)$.

We now prove that $\phi$ is injective.  Let $ g \in \Ker \phi$ and consider the simple current extension:
  \begin{align*}
 V_{ K_{12}} =  V_{ K_{12}}^ \tau \oplus V_{ K_{12}} [1] \oplus V_{ K_{12}}[2].
\end{align*}
Since $ M \circ g \cong M  $ for $ M \in R(V_{K_{12}}^\tau)$,  the set $S_B=\{V_{ K_{12}}^ \tau, V_{ K_{12}} [1], V_{ K_{12}}[2]\}$ is also fixed by $g$ pointwise. Therefore, $ g $ can be lifted to $  \Ct_{ \Aut( V_ { K_{12}}) } (\langle \hat{\tau}\rangle) $.  By Prop. \ref{thm:ses:C_AutVL: C_OL}, we have the exact sequence
\begin{align*}
 0 \xrightarrow{ }\Hom( K_{12}/ ( 1 - \tau) K_{12}) \xrightarrow{ } \Ct_{ \Aut( V_ { K_{12}}) } (  \hat{\tau})  \xrightarrow{}  \Ct_{O(K_{12})}(\tau) \xrightarrow{} 0.
\end{align*}
Recall that $O(K_{12}) \cong 6. \mathrm{PSU}_4(3). 2^2\cong 3. O^-_6(3)$ (see for example \cite{CS}) and $\Ct_{O(K_{12})}(\tau)\cong 6. \mathrm{PSU}_4(3). 2$. Therefore, $\Ct_{O(K_{12})}(\tau) / \langle \tau\rangle \cong 2.\mathrm{PSU}_4(3). 2\cong \Omega^-_6(3)$.

By direct calculations, it is easy to verify that $\Ct_{ \Aut( V_ { K_{12}}) } (  \hat{\tau}) / \langle \hat{\tau}\rangle$ acts faithfully on $R(V_{K_{12}}^\tau)$. Hence $\ker \phi = id$.

Next we determine the image of $\phi$.  In \cite{LY13},
a subgroup $H \cong {}^+ \Omega^-(8,3)$ is constructed explicitly using $\sigma$-involutions associated to $c=4/5$ Virasoro vectors (see \cite[Rmk 5.52,Thm 5.64]{LY13}).  Since ${}^+ \Omega^-(8,3)$ is an index 2 subgroup of the full orthogonal group $O^-_8(3)$, we have $\mathrm{Im}\,\phi \cong {}^+ \Omega^-(8,3)$ or $O^-_8(3)$.

Suppose $\mathrm{Im}\,\phi \cong O^-_8(3)$. In this case, $Z(\Aut (V_{ K_{12}}^{\tau}))\cong \mathbb{Z}_2$.

Let $h$ be an order $2$ element in $Z(\Aut (V_{ K_{12}}^{\tau}))$. Then $\phi(h)$ is the $-1$ map on $R(V_{K_{12}}^\tau)$. By the fusion rules (see Equation \ref{FR:K12}), we have
\[
V_{K_{12}}[1]\circ h =\phi(h) (V_{K_{12}}[1]) \cong V_{K_{12}}[2]\quad \text{ and } \quad
V_{K_{12}}[2]\circ h =\phi(h) (V_{K_{12}}[2]) \cong V_{K_{12}}[1].
\]
Therefore, $h$ lifts to $\Aut(V_{K_{12}})$ and is contained in the subgroup isomorphic to $$N_{ \Aut V_{ K_{12}}} ( \hat{\tau} )/ \langle \hat{\tau}\rangle \cong 3^6: (2.\mathrm{PSU}_4(3). 2^2),$$ which is centerless. It is a contradiction and hence
$\mathrm{Im}\,\phi \cong {}^+ \Omega^-(8,3)\cong \Omega^-_8(3).2$.
\end{proof}

\medskip
For now on, we denote $R(V_{K_{12}}^\tau)$ by $R$ for simplicity.
Notice that  $(R, -q)$ also forms a non-singular quadratic space of minus type.
Therefore, $(R, -q)\cong (R, q)$.

Let $\eta: (R, -q) \to (R,q)$ be a linear isometry and set
\[
S_\eta= \{ (a, \eta(a))\in R\times R\mid a\in R\}.
\]

\begin{lem}\label{S1}
The set $S_\eta$ is a maximal totally singular subspace of $R\times R$. Moreover, the minimal conformal weight of $S_\eta$ is $2$.
\end{lem}

\begin{proof}
It is clear that $S_{\eta}$ is a vector subspace of $R\times R$ and $\dim_{\mathbb{F}_3} S_\eta =\dim_{\mathbb{F}_3} R=8$.

By the definition of $\eta$, we also have
\[
q(a, \eta(a))= q(a) +q(\eta(a))= q(a)-q(a)=0 \quad \text{ for all } a\in R
\]
Therefore, $S_\eta$ is totally singular. It is maximal since $\dim_{\mathbb{F}_3} S_\eta= 1/2 \dim_{\mathbb{F}_3} (R\times R)$.

Recall that the conformal weights of the elements in $R$ are given by
\[
\wt(a) =
\begin{cases}
0 & \text{ if } a=0, \\
1 & \text{ if } q(a)=0, a\neq 0,  \\
4/3 & \text{ if } q(a)=1, \\
2/3 & \text{ if } q(a)=2. \\
\end{cases}
\]
Therefore, $\wt(a, \eta(a)) =2$ if $a\neq 0$.
\end{proof}

\begin{lem}\label{S2}
Let $S$ be a maximal totally singular subspace of $R\times R$ such that the minimal conformal weight of $S$ is $\geq 2$. Then there is a linear isomorphism $\eta: R\to R$ such that $q(\eta(a)) =-q(a)$ for all $a\in R$ and $S= S_{\eta}$.
\end{lem}

\begin{proof}
Let $p_i: R\times R \to R, i=1,2,$ be the natural projection to the $i$-th coordinate.

Let $(a,b)\in S$ be a non-zero vector. Then neither $a$ nor $b$ is zero; otherwise, $\wt(a,b)= \wt(a) +\wt(b) =1$.
That means $p_i|_{S}$ is injective for any $i=1,2$ and hence $p_i|_{S}$ is bijective for any $i=1,2$ since $\dim_{\mathbb{F}_3}(S)=\dim_{\mathbb{F}_3}(R)$.

Let $\eta= p_2\circ (p_1|_{S})^{-1}$. Then $\eta: R \to R $ is a linear isomorphism and
\[
S= \{ (a, \eta(a))\mid a\in R\}.
\]
Since $S$ is totally singular,  we have $q(a, \eta(a)) = q(a)+ q(\eta(a))=0$ for any $a\in R$ and hence $q(\eta(a))=-q(a)$ for all $a\in R$.
\end{proof}

\begin{prop}
Let $\mathcal{S}$ be the set of all maximal totally singular subspaces of $R\times R$ such that the minimal conformal weight is $\geq 2$. Then $\mathcal{S}$ is transitive under the action of $O^-_8(3)\wr 2$.
\end{prop}

\begin{proof}
This follows immediately from Lemmas \ref{S1} and \ref{S2}.
\end{proof}

\begin{lem}\label{lem:Ssharp}
Let $S^\sharp$ be a maximal totally isomorphic subspaces of $R\times R$ such that $ \oplus_{M\in S} M = V^\sharp$. Then
\[
Stab_{\Aut(V_{K_{12}}^\tau\otimes V_{K_{12}}^\tau)} (S^\sharp) \cong \Aut((V_{K_{12}}^\tau) \cong \Omega^-_8(3).2.
\]
\end{lem}

\begin{proof}
By Lemma \ref{S2}, $S^\sharp= \{ (a, \eta(a))\mid a\in R\}$ for some linear isomorphism $\eta:R \to R$ such that $q(\eta(a))=-q(a)$.

Note also that for any $g\in O(R,q)$, $\eta g\eta^{-1}$ is also in $O(R,q)$. Moreover, the map
\[
\begin{split}
\xi: O(R,q) &\to O(R,q)\times O(R,q)\\
        g   & \mapsto (g, \eta g\eta^{-1})
\end{split}
\]
is a group monomorphism. It is also easy to verify that
$Stab_{O(R,q)\wr 2} (S^\sharp) = \xi(O(R,q))$.
Hence, we have the desired conclusion.
\end{proof}

\begin{prop}
Let $S^\sharp$ be defined as in Lemma \ref{lem:Ssharp}. Let $A$ be the abelian subgroup of $\Aut(V^\sharp)$, which acts on
$V^\sharp$ as the dual group of $S^\sharp$. Then
\[
N_{\Aut(V^\sharp)}(A) \cong 3^8.\Omega^-_8(3).2.
\]
\end{prop}

\begin{proof}
This follows from Theorem \ref{NCthm} and Lemma \ref{lem:Ssharp}.
\end{proof}

We have constructed two subgroups
\[
\begin{split}
&H_1= Stab_{\Aut(V^\sharp)} (V_\Lambda^\tau) = N_{\Aut(V^\sharp)}(\langle \tau'\rangle) \cong 3^{1+12}.(2.\Suz:2), \text{ and }\\
&H_2= Stab_{\Aut(V^\sharp)}(V_{K_{12}}^\tau \otimes V_{K_{12}}^\tau) = N_{\Aut(V^\sharp)}(3^8) \cong 3^8. \Omega^-_8(3).2.
\end{split}
\]

\begin{lem}
The intersection of $H_1$ and $H_2$ is the common stabilizer of the subVOAs $V_\Lambda^\tau$ and $V_{K_{12}}^\tau \otimes V_{K_{12}}^\tau$  and $H_1 \cap H_2 = 3^8 . ( 3^6 . (2  \mathrm{PSU}_4(3) . 2^2)).$
\end{lem}

\begin{proof}
Recall the exact sequence
\[
1 \to 3^8 \to H_2 \to \Omega^-_8(3).2.
\]
By definition, it is clear that the normal subgroup $3^8$ stabilizes all irreducible $V_{K_{12}}^\tau \otimes V_{K_{12}}^\tau$ submodules in $S^\sharp$ and hence it stabilizes $V_{\Lambda}^\tau$, also.

Note that  $H_2$ acts on $(R,q)$ as a subgroup of isometries. The subgroup that stabilizes $V_{\Lambda}^\tau$ will stabilize a $6$-dimensional non-singular subspace of $R$ and it has the shape $$3^6:(2.PSU_4(3).2^2)\qquad  (\text{see \cite[page 141]{Atlas}).}$$ Hence, $H_1\cap H_2 \cong 3^8.(3^6: (2.PSU_4(3).2^2)) $.
\end{proof}

\begin{rmk}
Unfortunately, we do not have a direct proof that $\Aut(V^\sharp)$ is finite and hence we cannot apply Theorem \ref{thm:3local_monster} to conclude that $\Aut(V^\sharp)$ is isomorphic to the Monster.
\end{rmk}

\bibliographystyle{amsalpha}
\newcommand{\etalchar}[1]{$^{#1}$}
\providecommand{\bysame}{\leavevmode\hbox to3em{\hrulefill}\thinspace}
\providecommand{\MR}{\relax\ifhmode\unskip\space\fi MR }
\providecommand{\MRhref}[2]{%
  \href{http://www.ams.org/mathscinet-getitem?mr=#1}{#2}
}
\providecommand{\href}[2]{#2}

\end{document}